\documentclass[12pt,english]{article}
\usepackage{lmodern}

\usepackage[T1]{fontenc}
\usepackage[utf8]{inputenc}
\usepackage[a4paper]{geometry}
\usepackage{color}
\usepackage{mathrsfs}
\usepackage{mathtools}
\usepackage{bm}
\usepackage{amsmath}
\usepackage{amssymb}
\usepackage{amsthm}
\usepackage{hyperref}
\usepackage{xcolor}

\makeatletter

\catcode`@=11 \@addtoreset{equation}{section} \catcode`@=12

\theoremstyle{plain}
	\newtheorem{theorem}{Theorem}[section]
	\newtheorem{lemma}[theorem]{Lemma}
	\newtheorem{proposition}[theorem]{Proposition}
	\newtheorem{corollary}[theorem]{Corollary}

\theoremstyle{definition}
	\newtheorem{definition}[theorem]{Definition}
	
	\newtheorem{open problem}[theorem]{Open Problem}
	\newtheorem{remark}[theorem]{Remark}

\makeatletter
\newcommand{\subjclass}[2][1991]{
  \let\@oldtitle\@title
  \gdef\@title{\@oldtitle\footnotetext{#1 \emph{Mathematics subject classification.} #2}}
}
\newcommand{\keywords}[1]{
  \let\@@oldtitle\@title
  \gdef\@title{\@@oldtitle\footnotetext{\emph{Key words and phrases.} #1.}}
}
\makeatother

\usepackage[english]{babel}
\usepackage{dsfont}
\usepackage{changepage}

\newcommand{\mean}[1]{\,-\hskip-1.08em\int_{#1}} 
\newcommand{\eps}{\varepsilon}
\renewcommand{\div}{\text{\sl div}} 
\newcommand{\uno}{\mathds{1}} 
\newcommand{\hatE}{\widehat{\uno_{E}}}
\newcommand{\hatEc}{\widehat{\uno_{E^c}}}
\newcommand{\tildef}[1]{\widetilde{ #1}}   
\newcommand{\X}{{\mathbb X}} 
\newcommand{\supp}{\mathrm{supp}} 
\renewcommand{\d}{\mathrm{d}}
\newcommand{\Lip}{\mathrm{Lip}} 
\newcommand{\bs}{\mathrm{bs}} 
\newcommand{\DM}{\mathcal{DM}} 
\newcommand{\h}{{\rm h}} 
\newcommand{\rcd}{{\mathsf{RCD}(K,\infty)}}
\newcommand{\weakstarto}{\stackrel{*}{\rightharpoonup}} 
\renewcommand{\a}{{\rm a}}

\newcommand{\weakto}{\rightharpoonup} 
\newcommand{\M}{\mathbf{M}} 
\newcommand{\R}{\mathbb{R}}
\newcommand{\der}{\mathrm{Der}}
\newcommand{\sder}{\mathop{\mbox{$\mathrm{S}$--$\der$}}}
\newcommand{\frd}{\mathfrak{D}}
\newcommand{\res}{\mathop{\hbox{\vrule height 7pt width .5pt depth 0pt
\vrule height .5pt width 6pt depth 0pt}}\nolimits}
\newcommand{\Leb}{\mathscr{L}^}
\newcommand{\ban}[1]{\left\langle  #1 \right\rangle}  
\newcommand{\Haus}[1]{{\mathscr H}^{#1}} 
\newcommand{\redbeu}{\mathscr{F}} 

\def\N{\mathbb{N}}

\def\mm{\mathscr{M}}
\def\nn{\mathscr{N}}
\def\vinf{\mathbb{V}^{\infty}}

\title{On $BV$ functions and essentially bounded divergence--measure fields in metric spaces}

\author{Vito Buffa,
Giovanni E. Comi,
Michele Miranda Jr.
}

\date{\today}

\subjclass[2020]{Primary 26A45, 26B20. Secondary 30L99, 53C23.}

\keywords{Functions of bounded variation, divergence--measure fields, Gauss--Green formula, normal traces, metric measure spaces, curvature dimension condition, cotangent module}

\makeatother

\usepackage{babel}
\begin{document}

\maketitle
\begin{abstract}\noindent
\begin{adjustwidth}{0,5cm}{0,5cm}
By employing the differential structure recently developed by N. Gigli, we first give a notion of functions of bounded variation ($BV$) in terms of suitable vector fields on a complete and separable metric measure space $(\X,d,\mu)$ equipped with a non--negative Radon measure $\mu$ finite on bounded sets. Then, we extend the concept of divergence--measure vector fields $\DM^p(\X)$ for any $p\in[1,\infty]$ and, by simply requiring in addition that the metric space is locally compact, we determine an appropriate class of domains for which it is possible to obtain a Gauss--Green formula in terms of the normal trace of a $\DM^\infty(\X)$ vector field. This differential machinery is also the natural framework to specialize our analysis for $\rcd$ spaces, where we exploit the underlying geometry to determine the Leibniz rules for $\DM^\infty(\X)$ and ultimately to extend our discussion on the Gauss--Green formulas.
 \end{adjustwidth}
\end{abstract}

\tableofcontents
\section{Introduction}

The main purpose of this paper is to investigate integration by parts formulas for essentially bounded divergence--measure fields and functions of bounded variation ($BV$) in the very abstract contexts of locally compact and $\rcd$ metric measure spaces $(\X,d,\mu)$.

In order to deal with ``vector fields'' on metric measure spaces, one needs of course to refer to some differential structure of the ambient space, in terms of which the usual differential objects of the ``smooth''  analysis and geometry find a consistent and equivalent counterpart.

To this aim, we chose to follow the approach of \cite{gi2}, where the author builds a first--order $L^2$--theory of differential forms and vector fields which is tailored for metric measure spaces, in particular -- when second--order differentiability issues are addressed -- those whose Ricci curvature is bounded from below by some $K\in\mathbb{R}$. This differential structure, which was anticipated in the analysis carried on by the same author in \cite{gi1}, relies on the theory of $L^p$--normed modules, $1\le p\le\infty$, whose construction is based on the idea of the $L^\infty$--modules previously considered in \cite{we} and which is discussed in detail in \cite{gi2}. There, the machinery of such theory serves as the starting point for the settlement of a metric counterpart to the cotangent bundle, namely the \textit{cotangent module} $L^2\left(T^*\X\right)$ defined as the $L^2$--normed module consisting of square--summable \textit{differential forms}. 

The concept of a differential form is characterized in terms of the minimal weak upper gradients of Sobolev functions -- defined via test plans, \cite{ags3,ags4,gi1} -- and of suitable partitions of the underlying space; by this procedure, the \textit{differential} of a Sobolev function arises as the 1--form given accordingly as an equivalence class consisting of the function itself and the whole space $\X$. Then, by duality -- in the sense of normed modules -- the \textit{tangent module} $L^2(T\X)$, namely the $L^2$--normed module having square-summable \textit{vector fields} as its own elements, is given as the dual module of $L^2(T^*\X)$. Of course, taking into account all the due -- albeit expected -- technicalities (see in particular Remark \ref{p-grad}), this approach can be extended to any couple of conjugate exponents $p,q\in[1,\infty]$ with $\frac{1}{p}+\frac{1}{q}=1$, giving rise to the corresponding modules $L^p(T^*\X)$ and $L^q(T\X)$. As this straightforward task was presented in \cite{bu}, we shall therefore follow that exposition. This is the main object of Section \ref{preliminaries}, where we also address all the other fundamental preliminary tools of our analysis, like the definition of $\rcd$ space and the related consequences of the lower curvature bound, especially in connection with the heat flow.

As one may expect, this language allows for well--posed characterizations of differential operators such as the \textit{gradient} of a function and the \textit{divergence} of a vector field, whose definitions are consistent with the classical ones of the Euclidean and Riemann-Finsler geometries.\\ 
This is for example the case of the just mentioned $\rcd$ spaces, whose core--characterization is that of infinitesimally Hilbertian spaces $\X$; that is, such that $W^{1,2}(\X)$ is a Hilbert space \cite{gi1}, whose Ricci curvature is bounded from below by some $K\in\mathbb{R}$. Such spaces exhibit several interesting properties, especially in regards with the \textit{heat flow} $\h_t:\,L^2(\X)\to L^2(\X)$, which is given as the gradient flow of the Cheeger--Dirichlet energy and defines a semigroup $\left(\h_t\right)_{t\ge 0}$ of linear and self-adjoint operators whose contraction properties in the metric setting with curvature control -- after the pioneering work \cite{ba} -- have been extensively studied, during the last decade, in notable papers like \cite{agmr,ags2,ags3,ags5,gko,sa}.
For us, the most important property of the heat flow is of course the \textit{Bakry---Émery contraction estimate} \eqref{eq:BE}, which gives an exponential decay for the norm of the gradient of a Sobolev function along the heat flow, a decay that is intimately related to the lower curvature bound of the space. The Bakry--Émery condition will be invoked repeatedly in Sections \ref{leib-rules} and \ref{GGformulas}, where it will play a fundamental role in establishing the correct forms of the Leibniz rules for divergence--measure vector fields, and therefore the Gauss--Green formulas.

The use of this differential machinery is essential also in our characterization of $BV$ functions on general metric measure spaces, which is the central topic of Section \ref{BV}. Indeed, by considering the subclass of those essentially bounded vector fields in $L^\infty(T\X)$ whose divergence belongs to $L^\infty(\X)$, namely $\frd^\infty(\X)$, we find that it is possible to characterize the total variation of a function in a familiar way: 
\begin{equation} \label{intro:tot_var_formula} \|Du\|(\Omega)=\sup\left\{\int_\Omega u\div(X)\d\mu;\;X\in\frd^\infty(\X),\;\supp(X)\Subset\Omega,\;|X|\le1\right\},
\end{equation}
where $\Omega\subset\X$ is any open set.

Our procedure is inspired by \cite{di}, where $BV$ functions are characterized by means of (bounded) \textit{Lipschitz derivations}, namely linear operators acting on Lipschitz functions, $\delta:\Lip_\bs(\X)\to L^0(\X)$, satisfying the Leibniz rule and the locality condition
\[|\delta(f)|(x)\le g(x) \Lip_{\mathrm{a}}(f)(x)\]
for $\mu$--almost every $x\in\X$, for all $f\in\Lip_\bs(\X)$ and for some $g\in L^0(\X)$, where $\Lip_{\mathrm{a}}(f)$ denotes the asymptotic Lipschitz constant of $f$, see Definition \ref{der}. As for the $L^\infty$ modules, also this class of objects was previously treated in more generality in \cite{we}.

Our characterization of $BV$ functions  instead, rather than relying on Lipschitz derivations, actually incorporates \textit{Sobolev derivations}, whose notion is taken from \cite{gi2} and reads as follows: given any two conjugate exponents $p,q\in[1,\infty]$, a linear map $L:D^{1,p}(\X)\to L^1(\X)$ satisfying
\[
 |L(f)|\le\ell|Df|
\]
for any $f\in D^{1,p}(\X)$ and for some $\ell\in L^q(\X)$ is a Sobolev derivation, see Definition \ref{sob-der} below. Under some additional assumptions, it turns out that the class of Lipschitz derivations can be actually extended to that of Sobolev ones, thanks to suitable approximation procedures.
Then, by arguments based on the theory of $L^p$--normed modules and on the properties of the modules $L^p(T^*\X)$ and $L^q(T\X)$, we are able to prove that there exists a one-to-one correspondence between the vector fields in the class $\frd^\infty(\X)$ and the bounded Lipschitz derivations used in \cite{di}.
This result (stated in Lemma \ref{isomorphism_1}, Lemma \ref{isomorphism_2} and Theorem \ref{der-iso-d}) allows us to define $BV$ functions. Indeed, it is the starting point for our construction of the $BV$ space in a similar manner as in \cite{di} and for the derivation of \eqref{intro:tot_var_formula}, which produces an equivalent characterization of functions with bounded variation (Theorem \ref{equal-BV}).

As the definition of $BV$ involves an integration by parts (or Gauss--Green) formula, it appears natural to study such formulas in more depth and detail.

The classical statement of the Gauss--Green formula in the Euclidean space $\R^n$ requires a vector field $F \in C^{1}_{c}(\R^{n}; \R^{n})$ and an open set $E$ such that $\partial E$ is a $C^1$ smooth $(n-1)$--dimensional manifold, in order to conclude that
\begin{equation*} \int_{E} \mathrm{div} F \, \d x = - \int_{\partial E} F \cdot \nu_{E} \, \d \Haus{n - 1}, \end{equation*}
where $\nu_{E}$ is the unit interior normal to $\partial E$ and $\Haus{n - 1}$ is the $(n-1)$--dimensional Hausdorff measure. These assumptions on the integration domain and on the field are clearly too strong for many practical purposes, and indeed an immediate consequence of the Euclidean theory of $BV$ functions is that we can extend this formula to the sets of finite perimeter (employing the reduced boundary $\redbeu E$ and the measure theoretic unit interior normal instead of the classical notions). Instead, looking for sharp regularity assumptions on the vector fields, it is convenient to consider the space of essentially bounded divergence--measure fields $\DM^\infty(\X)$, namely those $X\in L^\infty(T\X)$ whose distributional divergence is a finite Radon measure (among which we can mention, as examples in the metric setting, the subclass $L^\infty(T\X)\cap\frd^{1}(\X)$ of essentially bounded vector fields, where $\frd^1(\X)$ is the set of those $X\in\ L^1(T\X)$ with summable divergence\footnote{See Definition \ref{def-div} for the characterization of the spaces $\frd^q(\X)$, $q\ge1$.}).

The space $\DM^\infty(\X)$ was originally introduced in the Euclidean framework by G. Anzellotti in 1983 (\cite{anz}), in order to study the pairings between vector valued Radon measures and bounded vector fields. One of the main results of his seminal work is an integration by parts formula for $\DM^{\infty}$ vector fields and continuous $BV$ functions on a bounded open set $\Omega$ with Lipschitz boundary. In particular, in \cite[Theorem 1.2]{anz} the author proved the existence of essentially bounded normal traces on the boundary of the Lipschitz domain. This family of vector fields was then rediscovered in the early 2000's and studied by many authors aiming to a variety of applications of the generalized Gauss--Green formulas (we refer for instance to \cite{ACM, CF2, DMM, LeoSaracco1, scheven2016bv, scheven2017anzellotti, scheven2016dual, Sch, Si2, Si3}). At first, the integration domains were taken with $C^1$ or Lipschitz regular boundary, and subsequently the case of sets of finite perimeter was considered in \cite{ctz, comi2017locally, CoT, Si}. In particular, in \cite{comi2017locally} the authors extended to $\DM^\infty$ the methods of Vol'pert's proof of the integration by parts formulas for essentially bounded $BV$ functions, \cite{V, VH}. The generalized method consists in proving a Leibniz rule for $\DM^\infty$--fields and essentially bounded scalar functions of bounded variation and in applying this formula to the characteristic function of a set of finite perimeter: then, the Gauss--Green formulas are a simple consequence of some identities between Radon measures derived from the Leibniz rule.

More precisely, \cite[Theorem 3.2]{comi2017locally} states that, if $F \in \DM^{\infty}(\R^{n})$ and $E$ is a bounded set of finite perimeter in $\R^{n}$, then there exists interior and exterior normal traces $(\mathcal{F}_{i} \cdot \nu_{E}), (\mathcal{F}_{e} \cdot \nu_{E}) \in L^{\infty}(\redbeu E; \Haus{n-1})$ such that we have the following Gauss--Green formulas:
\begin{align*}
\mathrm{div}F(E^{1}) & = - \int_{\redbeu E} \mathcal{F}_{i} \cdot \nu_{E} \, d \Haus{n - 1} \\ \mathrm{div}F(E^{1} \cup \redbeu E) &  = - \int_{\redbeu E} \mathcal{F}_{e} \cdot \nu_{E} \, d \Haus{n - 1},
\end{align*}
where $E^1$ is the measure theoretic interior of $E$; that is, the set of points with Lebesgue density 1. We stress the fact that we obtain two formulas since the divergence of the field may be concentrated on an $(n-1)$--dimensional set (such as the reduced boundary), due to the possible jump discontinuities of the field. This explains the necessity of having two different normal traces. Indeed, as an example one may consider
\[F(x) = \frac{x}{|x|^{n}} \chi_{\R^n \setminus B(0, 1)}(x)\] 
(where $B(0,1)$ denotes the unit ball centered in the origin): it is possible to show that $(\mathcal{F}_{i} \cdot \nu_{B(0,1)}) = 0$, while $(\mathcal{F}_{e} \cdot \nu_{B(0,1)}) = 1$.

We refer to \cite{comi2017locally} for a more detailed exposition of the history of the divergence--measure fields and their applications in $\R^{n}$, and to \cite{CCT, crasta2017anzellotti, crastadecicco2, crastadeciccomalusa, LeoSaracco2} for some recent developments. 

It seems natural to investigate the possibility to extend the theory of divergence--measure fields and Gauss--Green formulas to non--Euclidean settings, and indeed there have been researches in this direction: in \cite{bu, bm, mms1} the authors considered doubling metric measure spaces supporting a Poincar\'e inequality, while in \cite{cm} the class of horizontal divergence--measure fields in stratified groups is studied; lastly, in the more recent paper \cite{bps}, a Gauss--Green formula for sets of finite perimeter was proved in the context of an $\mathsf{RCD}(K,N)$ space by means of Sobolev vector fields in the sense of \cite{gi2}. In particular, in \cite{mms1} the authors employed the Cheeger differential structure (see \cite{ch}) to prove a Gauss--Green formula on the so--called \textit{regular balls}, and later, in \cite{bu,bm}, a Maz'ya--type approach based on \cite[Section 9.5]{ma}, allowed to write a similar formula in terms of the \textit{rough trace} of a $BV$ function. 
On the other hand, in \cite{cm} the authors exploited the approach developed in \cite{comi2017locally} and proved that it is possible to extend the method to stratified groups. In Sections \ref{leib-rules} and \ref{GGformulas} we shall follow a similar method in the framework of locally compact $\rcd$ metric measure spaces; that is, we shall first derive a Leibniz rule and then use it to obtain suitable identities between Radon measures, from which the Gauss--Green formulas shall follow.

Having the appropriate class of vector fields to work with, we can start to discuss the main topic of our work, namely the investigation of the Gauss--Green formulas.  In Section \ref{reg-dom} we define the space of $p$--summable divergence--measure fields on a general metric measure space $\X$, $\DM^{p}(\X)$, and, as an interesting byproduct of the definitions and of the differential structure we are using, we show that, if $X\in\DM^p(\X)$, for $p \in (1, \infty]$, the measure $\div(X)$ is absolutely continuous with respect to the $q$--capacity, where $q$ denotes the conjugate exponent of $p$.

Then, we refine the analysis of \cite{bu} -- which was tailored to geodesic spaces -- to give a Gauss--Green formula on regular domains in a locally compact metric measure space, requiring no further structural assumptions to be satisfied. 

The hypothesis of local compactness will be essential here, and also in Sections \ref{leib-rules} and \ref{GGformulas}, since we will often need to rely on the Riesz Representation Theorem and its corollaries in our arguments. 

Inspired by \cite{mms1}, we generalize the concept of regular balls by considering the class of \textit{regular domains} formerly introduced in \cite{bu}, namely those open sets $\Omega\subset\X$ of finite perimeter for which the upper inner Minkowski content of their boundary satisfies
\[
 \mathfrak{M}^*_i(\partial\Omega)\coloneqq\underset{t\to 0}{\lim\sup}\frac{\mu(\Omega\setminus\Omega_t)}{t}=\|D\uno_\Omega\|(\X),
\]
where, for $t>0$, 
\[
 \Omega_t\coloneqq\left\{x\in\Omega;\;\text{dist}(x,\Omega^c)\ge t\right\}.
\]

From our analysis, it turns out that the main notable properties of any regular domain $\Omega$ are the following:
\begin{itemize}
 \item Just like regular balls, also regular domains admit a \textit{defining sequence} $(\varphi_{\eps})_{\eps > 0}$ of bounded Lipschitz functions, which somehow encodes the properties of the domain itself (Definition \ref{def-seq}).
 \item The perimeter measure of $\Omega$ can be approximated in the sense of Radon measures via the family of measures $|D\varphi_\eps|\mu$; that is, $|D\varphi_\eps|\mu\weakto\|D\uno_\Omega\|$, as $\eps\to0$ (Proposition \ref{w-conv-per}).
\end{itemize}
These two facts together entitle us to make use of similar arguments as in the proof of \cite[Theorem 5.7]{mms1} and therefore to establish the Gauss--Green formula given in Theorem \ref{gauss-geod} below, which features the first occurrence of the interior normal trace $\left(X\cdot\nu_\Omega\right)_{\partial\Omega}^{-}$ of a vector field $X\in\DM^\infty(\X)$. We then extend our discussion on regular domains by imposing further assumptions on our selected domain $\Omega$, finding the conditions which first allow us to determine the exterior normal trace $\left(X\cdot\nu_\Omega\right)_{\partial\Omega}^{+}$ of $X\in\DM^\infty(\X)$ and the related integration by parts formula, Theorem \ref{thm:IBP_ext_reg}. We conclude this section with Corollary \ref{cor:dom_reg_trace_abs_cont}, where we find suitable sufficient conditions on $\Omega$ and $\div(X)$ under which the interior and exterior normal traces coincide. In such case, they are denoted by $\left(X\cdot\nu_\Omega\right)_{\partial\Omega}$.

It is worth to mention that a similar connection between the normal traces and the Minkowski content of the boundary of the domain was considered in the Euclidean space in \cite[Proposition 6.1 and Proposition 6.2]{CCT}, in the spirit of \cite[Theorem 2.4]{Si3}.

From Section \ref{leib-rules} onwards our analysis moves definitively to the realm of (locally compact) $\rcd$ spaces. 

Here we address in particular the issue of Leibniz rules for $\DM^\infty(\X)$ vector fields and the action of $BV$ functions on such objects. Our aim is to extend the product rules already established in the Euclidean spaces in \cite[Theorem 3.1]{CF1} and \cite[Theorem 2.1]{Frid}. In order to do so, we have to use extensively the self--adjointness and the contraction properties of the heat flow $\h_t$, and above all the Bakry--Émery contraction estimate in the self-improved form \eqref{eq:BE}. Our arguments eventually lead to Theorem \ref{thm:Leibniz_DM_BV}, where we essentially prove that the action of $f\in BV(\X)$  -- the {\em pairing} between $Df$ and $X$ --  on $X\in\DM^\infty(\X)$ produces a vector field $fX$ still in $\DM^\infty(\X)$ and that any accumulation point $\bm{D}f(X)$ of the family of measures $(\d\h_t f(X) \mu)_{t > 0}$ in $\mathbf{M}(\X)$ is absolutely continuous with respect to the total variation measure $\|Df\|$. This fact will be then fundamental to define the normal traces of a divergence--measure field.

This is indeed the starting point of Section \ref{GGformulas}, where the \textit{interior and exterior distributional normal traces} of $X\in\DM^\infty(\X)$ on the boundary $\partial E$ of a set of finite perimeter $E\subset\X$ are given as the functions $\ban{X, \nu_E}^{-},\ban{X, \nu_E}^{+}_{\partial E}\in L^{\infty}(\X, \|D\uno_E\|)$ such that
\begin{align*}
 2\bm{D}\uno_E(\uno_E X) &=\ban{X, \nu_E}_{\partial E}^-\|D\uno_E\|, \\ 2\bm{D}\uno_E(\uno_{E^c}X) &=\ban{X, \nu_E}_{\partial E}^+\|D\uno_E\|.
\end{align*}
Due to the non--uniqueness of the pairing, a priori, we cannot ensure the uniqueness of these normal traces either. Indeed, from this point on, we shall assume to have fixed a sequence $t_{j} \to 0$ such that 
\begin{equation*}
\d\h_{t_j} (\uno_E)(\uno_E X) \mu \weakto \bm{D}\uno_E(\uno_E X) \ \ \text{and} \ \ \d\h_{t_j} (\uno_E)(\uno_E^c X) \mu \weakto \bm{D}\uno_E(\uno_E^c X).
\end{equation*}
Nevertheless, our arguments allow us to establish the relation between the interior and exterior distributional normal traces of $X$ on $E$ and on $E^c$ (Remark \ref{rem:normal_traces_compl}).
It is also not difficult to notice that the sequence $(\h_{t_{j}} \uno_E)_{j \in \N}$ admits a subsequence converging in the weak* topology of $L^\infty(\X,\|\div(X)\|)$ to some function $\tildef{\uno_E}$, which is again a priori not unique, being a weak* limit point of a uniformly bounded sequence.
The analysis of Section \ref{sec:refined_Leibniz_rule} starts with a discussion on the level sets of $\tildef{\uno_E}$, $\widetilde{E^s} := \{\tildef{\uno_E} = s \}$: we give a weaker notion of measure--theoretic interior and exterior of $E$ related to $\tildef{\uno_E}$, namely $\tildef{E^{1}}$ and $\tildef{E^{0}}$, and of measure--theoretic boundary 
\[
 \tildef{\partial^*E}\coloneqq\X\setminus\left(\tildef{E^0}\cup\tildef{E^1}\right).
\]
Taking then into account also the weak* accumulation points $\hatE$ of the sequence $(\h_{t_{j}}\uno_E)_{j \in \N}$ in $L^\infty(\X,\|D\uno_E\|)$, we first obtain in Theorem \ref{thm:Leibniz_rule_E_1} a refinement of the Leibniz rules for $\DM^\infty(\X)$ vector fields and then, with Proposition \ref{prop:elementary_prop_normal_traces} and Proposition \ref{rel-div-Es} we investigate further the properties of traces and the relation between the divergence--measure $\div(X)$ and the sets $\tildef{E^1}$ and $\tildef{\partial^*E}$. The discussion culminates with Theorem \ref{thm:Leibniz_rule_E_2}, which is a second refined version of the Leibniz rule between $X$ and $\uno_E$ and is fundamental in the derivation of the Gauss--Green formulas in Section \ref{sec:general_Gauss_Green}.

The issue of the dependence on the approximating sequence $(\h_{t_{j}} \uno_{E})_{j \in \N}$ can be solved under the additional assumption that $\|\div(X)\|\ll\mu$. Indeed, in Section \ref{div_abscont_mu}, it is showed that the interior and exterior distributional normal traces of $X\in\DM^\infty(\X)$ on the boundary of $E$ are uniquely determined and coincide. In this case, the unique normal trace is denoted by $\ban{X, \nu_E}_{\partial E}$.

We then reconsider \textit{en passant} a particular case of regular domains to give a characterization of the sets $\tildef{\Omega^0}$ and $\tildef{\Omega^1}$ to conclude that, when $\|\div(X)\|\ll\mu$, the normal trace coincides with the distributional one; that is, $\left(X\cdot\nu_\Omega\right)_{\partial\Omega}=\ban{X, \nu_\Omega}_{\partial \Omega}$.

Our survey on Gauss--Green formulas ends with Section \ref{hatE_onehalf}, where we formulate the hypothesis that the weak* limit points $\hatE$ of $\h_t\uno_E$ in $L^\infty(\X,\|D\uno_E\|)$ are constant functions.
Indeed, if $\X$ is a Euclidean or a Wiener spaces (see \cite[Proposition 4.3]{Ambrosio_Figalli}), given a measurable set $E$ satisfying either $\uno_E\in BV(\X)$ or $\uno_{E^c}\in BV(\X)$, we have the convergence
\[
 \h_t\uno_E\weakstarto\frac{1}{2}\;\text{in}\;L^\infty(\X,\|D\uno_E\|).
\]
However, the proof of this result relies heavily on a Leibniz rule for the total variation, a tool which is not available at present in the more abstract context of a general $\rcd$ space. Nevertheless, in Theorem \ref{thm:weak_star_conv_set_fin_per} we prove that 
\[
\mean{\X}\hatE\d\|D\uno_E\|=\frac{1}{2},
\]
for any weak* limit $\hatE$ of some subsequence of $\h_t \uno_E$. As an easy corollary, it follows that, if $\hatE$ is a constant function, it must be equal to $\frac{1}{2}$. It is worth to mention that this fact does not depend on the local compactness of the space, as one would expect, since the Wiener spaces are not locally compact in general. However, we need that $\supp(\mu) = \X$, an assumption which does not seem too restrictive for our purposes.

A first interesting consequence of the condition $\hatE = \frac{1}{2}$ is that the distributional normal traces of any $X\in\DM^\infty(\X)$ coincide $\|D\uno_E\|$--almost everywhere on the set $\tildef{E^0}\cup\tildef{E^1}$ and that the space $\X$ admits a ``tripartition'' up to a $\|\div(X)\|$--negligible set, namely
\[
 \|\div(X)\|\left(\X\setminus\left(\tildef{E^0}\cup\tildef{E^{1/2}}\cup\tildef{E^1}\right)\right)=0.
\]
Thanks to this result, we finally obtain a refinement of the Gauss--Green formulas previously given in Section \ref{sec:general_Gauss_Green}, as the weak* limit $\hatE$ does not appear anymore in the improved version of the results.


\section{Preliminaries}\label{preliminaries}

Throughout this paper, $(\X,d,\mu)$ shall always be a complete and separable metric measure space endowed with a non--negative Radon measure $\mu$; additional hypotheses on the space will
be made from time to time as soon as they are needed. A $\mu$--measurable set $E$ in $\X$ will be simply called \textit{measurable} when no ambiguity occurs. Given $x \in \X$ and $r > 0$, we denote by $B_{r}(x)$ the open ball centered in $x$ with radius $r$; that is,
\begin{equation*}
B_r(x) := \{ y \in \X : d(x, y) < r \}.
\end{equation*}

For any open set $\Omega \subset \X$, we shall denote by $\mathbf{M}(\Omega)$ the space of signed finite Radon measures on $\Omega$. 
Moreover, by $\Lip_b(\Omega)$ and by $L^0(\Omega)$ we shall indicate the spaces of bounded Lipschitz functions and of measurable functions in $\Omega$, respectively.

In addition, following the notation of \cite{ambhonda}, we set
\begin{equation*}
C_\bs(\Omega) := \{ f \in C_{b}(\Omega) : \mathrm{supp}(f) \Subset \Omega \},
\end{equation*}
where we write $A\Subset B$ to mean that $A\subset \!B$ is bounded and $\text{dist}\left(A,B^c\right)>0$. Analogously, we set $\Lip_{\bs}(\Omega)$ to be the set of Lipschitz functions in $C_{\bs}(\Omega)$. If $\Omega = \X$, then $C_{\bs}(\X)$ and $\Lip_{\bs}(\X)$ are the spaces of continuous and Lipschitz functions with bounded support, respectively. 

We denote by $\Lip_{\text{a}}(f)$ the asymptotic Lipschitz constant of $f$, namely \begin{equation}\label{lip-a}
  \Lip_\text{a}(f)(x)\coloneqq\underset{\rho\to 0}{\lim}\,\Lip(f,B_\rho(x)),\end{equation} where, for any set $E\subset\X$, one defines 
\[
 \Lip(f,E)\coloneqq\underset{x,y\in E;\,x\neq y}{\sup}\frac{|f(x)-f(y)|}{d(x,y)}.
                                                   \]

We recall a notion of weak convergence for Radon measures, following the monograph \cite[Chapter 8]{bo}.

\begin{definition} \label{def:weak_conv}
 Given a sequence of finite Radon measures $\left(\nu_{j}\right)_{j\in\mathbb{N}}\subset\bf{M}(\X)$, we say that $\left(\nu_{j}\right)_{j\in\mathbb{N}}$
is \textit{weakly convergent} to $\nu$ in $\mathbf{M}(\X)$ if
\[
\lim_{j\to+\infty}\int_{{\X}}fd\nu_{j}=\int_{{\X}}fd\nu,\qquad\forall f\in C_{b}({\X}),
\]
and we write $\nu_j\weakto\nu$.
\end{definition}

The above notion of convergence
is the natural generalization of the weak$^{*}$ convergence in the
duality of $C(K)$ with ${\bf M}(K)$, the space of finite Radon measures on $K$, for any compact set
$K\subset{\X}$.

This duality allows us to deduce that,
for a Radon measure $\nu\in{\bf M}({\X})$, its total variation
in ${\X}$ is given by
\[
\|\nu\|({\X})=\sup\left\{ \int_{{\X}}f \, \d\nu : f\in C_{b}({\X}),\:|f(x)|\leq1 \, \forall\,x\in{\X}\right\} .
\]

Since the class of bounded Lipschitz functions ${\rm Lip}_{b}({\X})$ constitutes
a subalgebra of $C_{b}({\X})$ and its restriction to $K$ for
any compact set $K\subset{\X}$ is a subalgebra of $C(K)$ containing
a non--zero constant function and which separates points,
it turns out that the restriction of ${\rm Lip}_{b}({\X})$ to
$K$ is dense in $C(K)$, so that we also obtain
\[
\|\nu\|({\X})=\sup\left\{ \int_{{\X}}f \, \d\nu : f\in{\rm Lip}_{b}({\X}),\:|f(x)|\leq1 \, \forall\,x\in{\X}\right\} .
\]

Of course, this argument can be also extended to Radon measures on
open sets $\Omega\subset{\X}$ to deduce that for any $\nu\in{\bf M}(\Omega)$
\[
\|\nu\|(\Omega)=\sup\left\{ \int_{\Omega}f \, \text{d}\nu : f\in{\rm Lip}_{b,c}(\Omega),\:|f(x)|\leq1 \, \forall\,x\in{\X}\right\} ,
\]
where
\[
{\rm Lip}_{b,c}(\Omega)\coloneqq\left\{ f\in{\rm Lip}_{b}({\X}) : \text{dist}\left({\rm supp}(f),\Omega^{c}\right)>0\right\} .
\]

We recall the notion of uniform tightness of a family of Radon measures and Prohorov's Theorem on the necessary and sufficient conditions for the weak convergence of Radon measures (for which we refer to \cite[Theorem 8.6.2]{bo}).

\begin{definition}
A family of Radon measures $\mathcal{M}$ on a metric space $\X$ is called {\it uniformly tight} if for all $\eps > 0$ there exists a compact set $K_{\eps}$ such that $\|\nu\|(\X \setminus K_{\eps}) < \eps$ for all $\nu \in \mathcal{M}$.
\end{definition}

\begin{theorem}\label{thm:bog_weak_conv}
Let $(\X, d)$ be a complete separable metric space and let $\mathcal{M}$ be a family of Radon measures on $\X$. Then the following conditions are equivalent:
\begin{enumerate}
\item every sequence $(\nu_{j})_{j \in \N} \subset \mathcal{M}$ contains a weakly convergent subsequence;
\item the family $\mathcal{M}$ is uniformly bounded and uniformly tight.
\end{enumerate}
\end{theorem}

\subsection{Sobolev spaces via Test Plans}

Let $C\left([0,1],\X\right)$ denote the space of continuous
curves equipped with the supremum norm. Note that, since the underlying metric space is complete and separable, $C\left([0,1],\X\right)$ will be complete and separable as well.

\begin{definition} We define the \textit{evaluation map} $\text{e}_{t}:\,C\left([0,1],\X\right)\rightarrow\X$,
$t\in[0,1]$ as
\[
\text{e}_{t}(\gamma)\coloneqq\gamma_{t}=\gamma(t)\quad\forall\,\gamma\in C\left([0,1],\X\right).
\]
Any curve $\gamma\in C\left([0,1],\X\right)$
is called \textit{$p$--absolutely continuous}, for some $p \in [1, \infty]$, if there exists $f\in L^{p}\left(0,1\right)$ such that
\begin{equation}
d\left(\gamma_{t},\gamma_{s}\right)\le\int_{t}^{s}f(r)\mathrm{d}r,\qquad\forall\,t,s\in(0,1)\;\text{with}\;t<s,\label{eq:p-abs-cont}
\end{equation}
and in this case we write $\gamma\in AC^{p}([0,1],\X)$.

By \cite[Theorem 1.1.2]{ags1}, to every $p$--absolutely continuous
curve we can associate the \textit{{metric derivative}
}(or the {\textit{speed}}) $t\mapsto\left|\dot{\gamma}_{t}\right|\in L^{p}\left(0,1\right)$
defined as the essential infimum of all the $f\in L^{p}\left(0,1\right)$
satisfying (\ref{eq:p-abs-cont}), and which is representable in terms
of an incremental ratio for almost every $t\in(0,1)$: 
\[
\left|\dot{\gamma}\right|(t)\coloneqq\underset{h\rightarrow0}{\lim}\frac{d\left(\gamma_{t+h},\gamma_{t}\right)}{h}.
\]
\end{definition}

Now let $\mathscr{P}\left(C([0,1],{\X})\right)$ be the space
of probability measures along continuous curves.

\begin{definition} Let $\bm{\pi}\in\mathscr{P}\left(C([0,1],{\X})\right)$.
We say that $\bm{\pi}$ has \textit{bounded compression} whenever there exists a constant $c=c(\bm{\pi})>0$
such that
\[
\left(e_{t}\right)_{\#}\bm{\pi}\le c(\bm{\pi})\mu\quad\forall\,t\in[0,1].
\]
If $p\in[1,\infty)$, $\bm{\pi}$ has bounded compression, it
is concentrated on $AC^{p}([0,1],{\X})$ and
\[
\int\!\!\int_{0}^{1}\left|\dot{\gamma}_{t}\right|^{p}\mathrm{d}t\mathrm{d}\bm{\pi}(\gamma)<\infty,
\]
then it will be called a $p$\textit{--test plan}.

If $p = \infty$, $\bm{\pi}$ has bounded compression, it
is concentrated on $AC^{\infty}([0,1],{\X})$ and $\Lip(\gamma) \in L^{\infty}(C([0, 1], \X), \bm{\pi})$, then it will be called an $\infty${\em --test plan}.
\end{definition}

As it is customary -- see for instance  \cite[Section 3.2]{hkst} -- given a non--negative locally finite Radon measure $\nu$ and $p \in [1, \infty]$, we say that a $\nu$--measurable function $f : \X \to \R$ is $p${\it --summable}, and we write $f \in \mathcal{L}^p(\X, \nu)$, if $\| f \|_{\mathcal{L}^p(\X, \nu)} < \infty$, where
\begin{equation*}
\| f \|_{\mathcal{L}^{p}(\X, \nu)} := \begin{cases} \displaystyle{\left ( \int_{\X} |f|^{p} \, \d \nu \right )^{\frac{1}{p}}} & \text{if} \ p \in [1, \infty),\\
\mbox{ } \\
\inf \{ \lambda > 0 : \nu(\{ |f| > \lambda \}) = 0 \}  & \text{if} \ p = \infty.
\end{cases}
\end{equation*}
Then, we introduce an equivalence relation on $\mathcal{L}^p(\X, \nu)$, by declaring $f_1, f_2 \in \mathcal{L}^p(\X, \nu)$ equivalent if and only if $$\nu(\{x \in \X : f_1(x) \neq f_2(x)\}) = 0.$$ 
Finally, we define the $L^p$--space $L^{p}(\X, \nu)$ as the vector space of such equivalence classes.
As it is well known, $L^{p}(\X, \nu)$ equipped with the $p$--norm $\| \cdot \|_{L^{p}(\X, \nu)} := \| \cdot \|_{\mathcal{L}^{p}(\X, \nu)}$ is a Banach space. With a little abuse of terminology, in the rest of the paper we shall refer to $p$--summable functions, rather than to equivalence classes. In addition, we shall simplify the notation in the case $\nu = \mu$ by removing the reference measure, namely  $L^{p}(\X) := L^{p}(\X, \mu)$.

\medskip

We can now give a definition of Sobolev--Dirichlet classes in terms of test plans. 

\begin{definition} \label{def-sob-class}Let $p\in[1,\infty]$
and let $q$ be its conjugate exponent, that is, $\frac{1}{p}+\frac{1}{q}=1$.
The {\it Sobolev--Dirichlet class} $D^{1,p}({\X})$ consists of all Borel functions
$f:{\X}\rightarrow\mathbb{R}$ for which there exists a non--negative
$g\in L^{p}(\X)$ satisfying
\begin{equation}
\int\left|f\left(\gamma_{1}\right)-f\left(\gamma_{0}\right)\right|\d \bm{\pi}(\gamma)\le\int\!\!\int_{0}^{1}g\left(\gamma_{s}\right)\left|\dot{\gamma}_{s}\right|\d s \d \bm{\pi}(\gamma)\label{eq:up-grad}
\end{equation}
for every $q$--test plan $\bm{\pi}$. Following the usual terminology, we shall say that $g$ is a $p${\it --weak upper gradient} of $f$.
\end{definition}

We note that, even though in the literature the above definition is mostly given for $p\in(1,\infty)$, the existence of $p$--test plans for any $p\in[1,\infty]$ entitles us to define the Sobolev--Dirichlet classes also for the limiting values of $p$.

We recall that, by \cite[Section 5.2]{ags3} and \cite[Section 4.5]{ags4}, for every $f\in D^{1,p}({\X})$ there exists a unique minimal $0\le g\in L^{p}(\X)$ such that (\ref{eq:up-grad}) holds. Such a function will be called the \textit{minimal $p$--weak upper gradient} of $f$ and we shall
denote it by $|Df|$. This allows to define a semi-norm on $D^{1,p}(\X)$, by setting 
\begin{equation} \label{eq:Sobolev_Dirichlet_seminorm}
\left\Vert f\right\Vert _{D^{1,p}({\X})}\coloneqq\left\Vert |Df|\right\Vert _{L^{p}(\X)}.
\end{equation}

\begin{remark}\label{rmk-prop-upgrad}As it is natural to expect, the minimal $p$--weak upper gradient satisfies the following properties (\cite{gi1,gi2} and the references therein):
\begin{enumerate}
 \item {\it Sub--linearity}: $|D(\alpha f+\beta g)| \le |\alpha||Df| + |\beta||Dg|$ for all $\alpha,\beta\in\R$, $f,g\in D^{1,p}(\X)$.
 \item {\it Weak Leibniz rule}: $|D(fg)|\le |f||Dg|+|g||Df|$ for all $f,g\in D^{1,p}(\X) \cap L^\infty(\X)\footnote{This intersection has to be intended in the sense that the Borel functions $f, g \in D^{1,p}(\X)$ are also essentially bounded.}$.
 \item {\it Locality}: $|Df|=|Dg|$ $\mu$--almost everywhere on $\{f=g\}$ for all $f,g\in D^{1,p}(\X)$.
 \item {\it Chain rule}: $|Df|=0$ on $f^{-1}({\cal N})$ for every $f\in D^{1,p}(\X)$ and every $\mathscr{L}^1$--negligible Borel set ${\cal N}\subset\R$. Moreover, if $\varphi:I\to\R$ is Lipschitz and $f\in D^{1,p}(\X)$, with $I\subset\R$ open such that $\mu(f^{-1}(\R\setminus I))=0$, then we have $\varphi\circ f\in D^{1,p}(\X)$ and 
 \[|D(\varphi\circ f)|=|\varphi'\circ f||Df|\] 
 $\mu$--almost everywhere, where $|\varphi'\circ f|$ is defined arbitrarily on the points where $\varphi$ fails to be differentiable.
\end{enumerate}
\end{remark}

We observe that the asymptotic Lipschitz constant $\Lip_\text{a}(f)$ is an upper gradient of $f$; so we get $|Df|\le\Lip_\text{a}(f)$ $\mu$--almost everywhere \cite{ags4,di}. This simple fact will play an important role in Lemma \ref{isomorphism_1}.

Thanks to this notion of Sobolev--Dirichlet class, the definition of
Sobolev spaces may be given in a traditional fashion.

\begin{definition} For all $p \in [1, \infty]$, we set
\[
\mathcal{W}^{1,p}({\X})\coloneqq \{ f \in D^{1,p}(\X) : \|f\|_{L^p(\X)} < \infty \},
\]
and we endow this space with the semi--norm
\begin{equation*}
\left\Vert f\right\Vert _{1,p} \coloneqq \begin{cases} \left ( \left\Vert f\right\Vert _{L^{p}(\X)}^{p}+\left\Vert f\right\Vert _{D^{1,p}(\X)}^{p} \right)^{\frac{1}{p}}  & \text{if} \ p \in [1, \infty), \\
\mbox{ } \\
\max\left \{\left\Vert f\right\Vert _{L^{\infty}(\X)}, \left\Vert f\right\Vert _{D^{1,\infty}(\X)}\right \}  & \text{if} \ p = \infty.
\end{cases}
\end{equation*}
where $\left\Vert f\right\Vert _{D^{1,p}({\X})}$ is given by \eqref{eq:Sobolev_Dirichlet_seminorm}. Let $\sim$ be the equivalence relation on $\mathcal{W}^{1,p}({\X})$ defined by setting $f \sim g$ if and only if $\left\Vert f - g\right\Vert _{1,p} = 0$. We define the {\em Sobolev space} $W^{1,p}(\X)$ as the quotient
\begin{equation*}
W^{1,p}(\X) \coloneqq \mathcal{W}^{1,p}({\X})/\sim
\end{equation*}
endowed with the norm $\| \cdot \| _{W^{1,p}({\X})} := \| \cdot \| _{1,p}$. \end{definition}

As in the case of the $L^p$--spaces, with a little abuse of terminology we shall refer to $W^{1,p}$--functions, rather than to equivalence classes.

To conclude this section, it is worth spending some words on an equivalent characterization of Sobolev spaces via test plans which involves the concept of ``negligibility'' of a family of curves with respect to any such probability measure. To this aim, below we shall recollect the salient arguments of \cite[Appendix B]{gi1}.

\begin{definition}
 A Borel set $\Gamma\subset C\left([0,1],\X\right)$ is called $p$\textit{--negligible} whenever $\bm{\pi}(\Gamma)=0$ for any $q$--test plan $\bm{\pi}$, where $p,q\in [1,\infty]$ are such that $\frac{1}{p}+\frac{1}{q}=1$.
 
 In this sense, any property will be said to hold for $p$\textit{--almost every curve} if the set where it fails is $p$--negligible.
\end{definition}

\begin{definition}\label{sobolev-along}
 A Borel function $f:\X\to\mathbb{R}$ is said to be \textit{Sobolev along $p$--almost every curve} if for $p$--almost every $\gamma$ the composition $f\circ\gamma$ coincides almost everywhere in $[0,1]$ with an absolutely continuous function $f_\gamma:[0,1]\to\mathbb{R}$.
 Thus, to any function $f$ which is Sobolev along $p$--almost every curve we associate a \textit{$p$--weak upper gradient}, i.e. a Borel map $g:\X\to[0,\infty]$ such that
 \[
  \left|f\left(\gamma_1\right)-f\left(\gamma_0\right)\right|\le\int_0^1 g\left(\gamma_s\right)\left|\dot{\gamma_s}\right|\d s
 \]
for $p$--almost every $\gamma$.
\end{definition}

This notion of Sobolev spaces turns out to be indeed equivalent to Definition \ref{def-sob-class}.

\begin{theorem}{\rm \cite[Theorem B.4]{gi1}}
Let $p,q\in[1,\infty]$ be any two conjugate exponents. If $f:\X\to\mathbb{R}$ and $g:\X\to[0,\infty]$ are two Borel functions, with $g\in L^p(\X)$, then the following statements are equivalent:
 \begin{enumerate}
 \item $f\in D^{1,p}(\X)$ and $g$ is a $p$--weak upper gradient of $f$ in the sense of Definition \ref{def-sob-class}.
 
 \item $f$ is Sobolev along $p$-almost every curve and $g$ is a $p$--weak upper gradient of $f$ in the sense of Definition \ref{sobolev-along}.
 \end{enumerate}
\end{theorem}

We finish the present survey on Sobolev spaces by recalling a result on Lipschitz approximation of functions in $W^{1,p}(\X)$ for $p \in (1, \infty)$.

\begin{proposition} \label{prop:Lip_approx_p}
Let $p \in (1, \infty)$. If $f \in W^{1,p}(\X)$, then there exists a sequence $(f_{k})_{k \in\N}\subset\Lip_{\bs}(\X)$ such that
\begin{equation} \label{eq:Lip_approx_p}
\lim_{k \to + \infty} \int_{\X} |f_k - f|^p + |\Lip_{\a}(f_{k}) - |Df| |^p \, d\mu = 0.
\end{equation}
\end{proposition}
\begin{proof}
Let $f \in W^{1,p}(\X)$. Then, \cite[Theorem 6.1]{acd} implies that $f$ has a $p$--relaxed slope which is equal to its minimal $p$--weak upper gradient $|Df| \in L^p(\X)$. Hence, \cite[Proposition 4.2]{acd} yields the existence of a sequence of functions $(f_{k})_{k \in \N}$ such that $f_k \in \Lip_{\bs}(\X)$ for all $k \in \N$ and \eqref{eq:Lip_approx_p} holds.
\end{proof}

\subsection{The differential structure}\label{sec-diff-str}

Below, we shall briefly discuss the differential machinery that we are going to use for the rest of the paper. As anticipated in the Introduction, this will provide a metric counterpart to the notions of cotangent and tangent bundles. We recall that, while \cite{gi2} provided this structure for the case $p=2$ only, a straightforward generalization of such construction for any exponent $p\in[1,\infty]$ was given in \cite[Chapter 5]{bu}, so here we shall follow the latter as our main reference.

We notice that for a full understanding of this setting, a throughout treatment of the theory of $L^p$--normed modules would be necessary; this task was carried on in detail in \cite{gi2}. Nevertheless, for the reader's convenience, we shall include a brief account on this theory in Appendix \ref{sec-appendix}.

\begin{definition} \label{pcm} We shall call \textit{pre--cotangent module} the set
\[
\textsc{pcm}_p\coloneqq\left\{ \left\{ \left(f_{i},A_{i}\right)\right\} _{i\in\mathbb{N}};\;\left(A_{i}\right)_{i\in\mathbb{N}}\subset\mathfrak{B}(\X),\:f_{i}\in D^{1,p}\left(\X\right)\:\forall\,i\in\mathbb{N},\;\sum_{i\in\mathbb{N}}\||D f_i|\|^p_{L^{p}(A_i)}<\infty\right\} ,
\]
where the $A_{i}$'s form a disjoint partition of
$\X$, $D^{1,p}(\X)$, $p\in[1,\infty)$,
denotes the Sobolev--Dirichlet class of order $p$ as in Definition
\ref{def-sob-class}, and $\mathfrak{B}({\X})$ denotes the set
of all equivalence classes of measurable subsets of ${\X}$, with
$A,B\subset{\X}$ being equivalent whenever 
\begin{equation*}
\mu(A\Delta B)=0, \text{ where } A\Delta B := (A\setminus B)\cup(B\setminus A).
\end{equation*} 

When $p=\infty$, in the corresponding definition of $\textsc{pcm}_\infty$ we shall simply require the $f_i$'s to be in $D^{1,\infty}(A_i)$ for all $i\in\mathbb{N}$ and
\[
 \sup_{i\in\mathbb{N}}\||D f_i|\|_{L^{\infty}(A_i)}<\infty.
\]

\end{definition}

We introduce an equivalence relation $\sim$ on
$\textsc{pcm}_{p}$ by setting $$\left\{ \left(f_{i},A_{i}\right)\right\} _{i\in\mathbb{N}}\sim\left\{ \left(g_{j},B_{j}\right)\right\} _{j\in\mathbb{N}}$$
whenever $\left|D\left(f_{i}-g_{j}\right)\right|=0$ $\mu$--almost
everywhere on $A_{i}\cap B_{j}$ for all $i,j\in\mathbb{N}$.

$\textsc{pcm}_{p}/\!\sim$ turns into a vector space if we define the
sum and the multiplication by scalars as
\begin{align*}[(f_{i},A_{i})_{i}]+[(g_{j},B_{j})_{j}] & = [(f_{i}+g_{j},A_{i}\cap B_{j})_{i,j}],\\
\lambda\left[\left(f_{i},A_{i}\right)_{i}\right] & =\left[\left(\lambda f_{i},A_{i}\right)_{i}\right].
\end{align*}
Let $\text{Sf}(\X,\mu)$ denote the space of simple functions, that is, those of the form 
\[
h=\sum_{j\in\mathbb{N}}\mathds1_{B_{j}}\text{·}a_{j},
\]
where $\left\{B_j\right\}_{j\in \mathbb{N}}$ is a partition of $\X$. If $\left[\left(f_{i},A_{i}\right)_{i}\right]\in\textsc{pcm}_{p}/\!\sim$,
then we define the multiplication with $h \in \text{Sf}(\X,\mu)$ as
\[
h[(f_{i},A_{i})_{i}]\coloneqq[(a_{j}f_{i},A_{i}\cap B_{j})_{i,j}].
\]
This operation gives a bilinear map $\text{Sf}(\mu)\times\textsc{pcm}_{p}/\!\sim\rightarrow\textsc{pcm}_{p}/\!\sim$
such that $$\uno_{\X} \left[\left(f_{i},A_{i}\right)_{i}\right]=\left[\left(f_{i},A_{i}\right)_{i}\right].$$

\begin{definition}\label{point-norm}Consider the map $|\cdot|_{*}:\textsc{pcm}_{p}/\!\sim\rightarrow L^{p}(\X)$
given by
\[
\left|\left[\left(f_{i},A_{i}\right)_{i}\right]\right|_{*}\coloneqq\left|Df_{i}\right|
\]
$\mu$--almost everywhere on $A_{i}$ for all $i\in\mathbb{N}$; this
map, namely the \textit{pointwise norm} on $\textsc{pcm}_{p}/\!\sim$, is well defined thanks
to the above definition of the equivalence relation on $\textsc{pcm}_{p}$.

Since $D^{1,p}(\X)$ is a vector space, one has
the following inequalities for $|\cdot|_{*}$:
\begin{align}\label{eq:point-norm}\begin{split}
|[(f_{i}+g_{j},A_{i}\cap B_{j})_{i,j}]|_{*} & \le|[(f_{i},A_{i})_{i}]|_{*}+|[(g_{j},B_{j})_{j}]|_{*}, \\
\left|\lambda\left[\left(f_{i},A_{i}\right)_{i}\right]\right|_{*} & =\left|\lambda\right|\left|\left[\left(f_{i},A_{i}\right)_{i}\right]\right|_{*}, \\
\left|h\left[\left(f_{i},A_{i}\right)_{i}\right]\right|_{*} & =\left|h\right|\left|\left[\left(f_{i},A_{i}\right)_{i}\right]\right|_{*},
\end{split}\end{align}
valid $\mu$--almost everywhere for every $[(f_{i},A_{i})_{i}],[(g_{j},B_{j})_{j}]\in\textsc{pcm}_{p}/\!\sim$,
$h\in\text{Sf}(\mu)$ and $\lambda\in\mathbb{R}$. \end{definition}

The above arguments, in particular (\ref{eq:point-norm}),
allow us to define a norm on $\textsc{pcm}_{p}/\!\sim$.

\begin{definition} We define $\|\cdot\| _{L^{p}\left(T^{*}\X\right)}:\textsc{pcm}_{p}/\!\sim\rightarrow[0,\infty)$, $p\in[1,\infty)$ by setting
\[
\left\Vert \left[\left(f_{i},A_{i}\right)_{i}\right]\right\Vert _{L^{p}\left(T^{*}\X\right)}^{p}\coloneqq\int_{\X}\left|\left[\left(f_{i},A_{i}\right)_{i}\right]\right|_{*}^{p}\d \mu=\sum_{i\in\mathbb{N}}\int_{A_{i}}\left|Df_{i}\right|^{p}\mathrm{d}\mu=\sum_{i\in\mathbb{N}}\||D f_i|\|_{L^{p}(A_i)}^p.
\]

As in Definition \ref{pcm}, when $p=\infty$, we set $\|\cdot\|_{L^\infty(T^*\X)}$ to be
\[
 \| [(f_i,A_i)_i]\|_{L^\infty(T^*\X)}\coloneqq\sup_{i\in\mathbb{N}}\||D f_i|\|_{L^{\infty}(A_i)}.
\]

The completion of $\textsc{pcm}_{p}/\!\sim$ with respect
to the norm $\| \cdot\| _{L^{p}\left(T^{*}\X\right)}$
will be called \textit{cotangent module} and denoted by $L^{p}\left(T^{*}\X\right)$.
Consequently, its elements will be called $p$-\textit{cotangent vector fields} or, more traditionally, $p$\textit{--integrable 1--forms} when $p<\infty$, \textit{essentially bounded 1--forms} when $p=\infty$.
\end{definition}

We notice that the space $L^{p}\left(T^{*}\X\right)$ is actually an $L^p$--normed module in the sense of Definition \ref{def:L_p_module}.

The choice of the terminology ``cotangent'', as explained in \cite{gi1,gi2}, is due to the fact that $p$--weak upper gradients are defined by means of a $(p,q)$--duality between test plans and speeds of curves, so in some geometric sense they constitute ``cotangent'' objects to the metric space $\X$.

\medskip

Having a notion of ``1--form'' at our disposal, we are entitled to
define the differential of any Sobolev function as follows.

\medskip

\begin{definition}\label{differential}
Given a function $f\in D^{1,p}(\X)$ we define its \textit{differential} $\mathrm{d}f\in L^{p}\left(T^{*}\X \right)$ as
\[
\mathrm{d}f\coloneqq\left[\left(f,\X\right)\right]\in \textsc{pcm}_{p}/\!\sim\, \subset L^{p}\left(T^{*}\X\right).
\]
Here, $\left(f,\X\right)$ stands for $\left(f_{i},A_{i}\right)_{i\in\mathbb{N}}$
with $f_{0}=f$, $A_{0}=\X$ and $f_{i}=0$, $A_{i}=\emptyset$
for every $i>0$. \end{definition}

Of course, the differential is linear by construction (in particular, it is $L^{\infty}$--linear) and the definition
of pointwise norm ensures that 
\[
|\text{d}f|_{*}=|Df|
\]
$\mu$--almost everywhere for all $f\in D^{1,p}({\X})$.
Moreover, $\text{d}$ is local; this means that for any $f,g\in D^{1,p}({\X})$
one has $\text{d}f=\text{d}g$ $\mu$--almost everywhere on $\{f=g\}$,
see \cite[Theorem 2.2.3]{gi2}.

Putting together \cite[Theorem 2.2.6]{gi2} and \cite[Corollary 2.2.8]{gi2} (and the respective generalizations featured in \cite{bu}), we get the usual calculus rules for the differential.

\begin{proposition}  
\label{prop:diff_Leibniz}
Let $p \in [1, \infty]$. The following equations hold $\mu$--almost everywhere:
\begin{enumerate}
       \item $\d(fg)=g\d f+f\d g$ for all $f,g\in D^{1,p}(\X)\cap L^{\infty}(\X)$;
       \item $\d f=0$ on $f^{-1}({\cal N})$ for all $f\in D^{1,p}(\X)$ and every $\mathscr{L}^1$-negligible Borel set ${\cal N}\subset\R$;
       \item $\d(\varphi\circ f)=(\varphi'\circ f)\d f$ for all $f\in D^{1,p}(\X)$ and every $\varphi:I\to\R$ Lipschitz, where $I\subset\R$ open is such that $\mu(f^{-1}(\R\setminus I))=0$ and $\varphi'\circ f$ is defined arbitrarily on $f^{-1}\left (\big\{\text{non--differentiability\:points\:of\:}\varphi\big\}\right)$. 
      \end{enumerate}
\end{proposition}

We recall now another useful property of the differentials of Sobolev-Dirichlet functions, stated in \cite[Proposition 2.2.5]{gi2} for $p = 2$ and generalized in \cite[Proposition 5.4.7]{bu}.

\begin{lemma} \label{lem:L_p_cotangent_generated}
Let $p \in [1, \infty)$. Then $L^{p}(T^* \X)$ is generated in the sense of modules (Definition \ref{def:span}) by the space $\{ \d f : f \in W^{1,p}(\X)\}$.
\end{lemma}

Now, by duality -- in the sense of modules, \cite{gi2} (see also in Appendix \ref{sec-appendix}) -- with the cotangent module $L^{p}\left(T^{*}{\X}\right)$ it
is possible to define a ``tangent module''. To this aim, let $p,q\in[1,\infty]$
be any two conjugate exponents, $\frac{1}{p}+\frac{1}{q}=1$.

\begin{definition}\label{tang-mod} The \textit{tangent module} $L^{q}(T\X)$ is the dual module of $L^{p}\left(T^{*}\X\right)$.
The elements of $L^{q}\left(T\X\right)$ will be called \textit{$q$--vector fields}.
\end{definition}

\begin{remark}\label{duality}
We explicitly point out that, by the theory of $L^p$--normed modules (see Remark \ref{dual-mod-norm}), for all $p\in[1,\infty]$ the module dual of $L^p(T^{*}\X)$ is indeed an $L^q$-normed module, where $\frac{1}{p}+\frac{1}{q}=1$, even in the extreme cases $p =1, q= \infty$ and $p= \infty, q=1$. In addition, thanks to Remark \ref{dual-mod-norm} and Lemma \ref{lem:L_p_cotangent_generated} we can define a pointwise norm $| \cdot|$ on $L^{q}\left (T \X \right)$ by setting 
\begin{equation*}
 \left\vert X \right\vert \coloneqq {\text{ess-}\!\sup}\left \{ \left\vert \d f(X)\right\vert : f\in D^{1,p}(\X),\:|\d f|_{*}\le1 \right \}
\end{equation*}
for every $X \in L^q(T \X)$.
\end{remark}

With the machinery discussed so far, it comes a ``natural'' notion of divergence of a vector field.

\begin{definition}\label{def-div}
For $q \in [1, \infty]$ we define
 \[
  \mathfrak{D}^q(\X) \coloneqq \left\{  X\in L^{q}(T \X);\: \exists f \in L^q (\X): \int_{\X} fg \text{d}\mu = -\int_{\X} \text{d} g(X)\text{d}\mu,\: \forall g\in W^{1,p}(\X)  \right\}.
 \]
 
Clearly, $\mathfrak{D}^q(\X) \subset L^q(T\X)$; the function $f$, which is unique by the density of $W^{1,q}(\X)$ in $L^q(\X)$, will be called the \textit{$q$--divergence} of the vector field $X$, we shall write $\div(X) \coloneqq f$, and the space $\mathfrak{D}^q(\X)$ will be obviously referred to as the \textit{domain of the $q$--divergence}.
\end{definition}

\smallskip

\begin{remark} The linearity of the differential implies that
$\mathfrak{D}^q(\X)$ is a vector space and hence that the divergence is a linear operator. Moreover, the Leibniz rule for differentials immediately yields the same property for the divergence as well: given $X\in \mathfrak{D}^q(\X)$ and $f\in L^{\infty}(\X) \cap D^{1,p}(\X)$ with $|\mathrm{d}f|_{*}\in L^{\infty}(\X)$, one has
\begin{align*}
fX & \in \mathfrak{D}^q(\X),\\
\div(fX) &=\text{d}f(X)+f\text{\ensuremath{\div}(X).}
\end{align*} 
In fact, these hypotheses give $\text{d}f(X)+f\div(X)\in L^{p}(\X)$ and for all $g\in W^{1,q}(\X)$ one has $fg\in W^{1,p}(\X)$, whence
\[
-\int_{\X}gf\div(X) \, \text{d}\mu=\int_{\X}\text{d}(fg)(X) \, \text{d}\mu=\int_{\X} g\text{d}f(X)+\text{d}g(fX) \, \text{d}\mu,
\]
so the claim holds.

\end{remark}

We conclude this section with some important comments on the choice of giving the notions of cotangent and tangent modules for arbitrary exponents $p,q\in[1,\infty]$.
\medskip
\begin{remark}\label{p-grad}
 Note that, until now, we have not made any structural assumptions on the metric measure space $(\X,d,\mu)$. This means that the minimal $p$--weak upper gradient of any Sobolev function $f\in W^{1,p}(\X)$ \textbf{depends} on $p$. Precisely, as the class of $q$--test plans contains the one of $q'$--test plans for $q\le q'$, then $D^{1,p}(\X)\subset D^{1,p'}(\X)$ for $p\ge p'$, so that for any $f\in D^{1,p}(\X)$, one has that if $g$ is a $p$--weak upper gradient of $f$, it is also a $p'$--weak upper gradient. In particular, for the minimal upper gradient it holds $|Df|_p\ge |Df|_{p'}$ $\mu$--almost everywhere in $\X$.
 The paper \cite{ds} contains a throughout discussion of the dependence of $p$--weak upper gradients on $p$ and provides also examples of spaces $\X$ where, for $f\in D^{1,p}(\X)\cap D^{1,p'}(\X)$, one has $|Df|_p \neq |Df|_{p'}$ on a set of positive measure.
 Thus, without any structural assumptions, then also $L^p(T^{*}\X)$ and $L^q(T\X)$ depend on the respective exponents.
 This issue is removed when $(\X,d,\mu)$ is endowed with a doubling measure and supports a Poincar\'e Inequality, \cite{ch}, or when it is an $\rcd$ space, \cite{gh}; then, since in the ensuing discussion we shall work mainly in the $\rcd$ setting, all these problematics will not be relevant to us.
 
 In particular, when the dependence on the exponent is removed, given any $X\in\frd^p(\X) \cap\frd^q(\X)$, from the definition of divergence it follows that there exist $f_p,f_q\in L^p(\X) \cap L^q(\X)$, $f_p=\div_p(X)$ and $f_q=\div_q(X)$, such that for any $g\in W^{1,p}(\X) \cap W^{1,q}(\X)$ one has
 
 \[
  \int_\X \left(f_p-f_q\right)g\text{d}\mu=0,
 \]
 
 meaning that $f_p-f_q=0$ $\mu$--almost everywhere, whence the uniqueness of the divergence of $X$.

\end{remark}

\subsection{$\rcd$ Spaces and the Heat Flow}\label{rcdht}
\bigskip

In this section we recall the notion of an $\rcd$ metric measure space, naively, a space whose Ricci curvature is bounded from below by some $K\in\mathbb{R}$, along with its most notable properties, which we shall use extensively in the upcoming sections. The construction of such spaces relies heavily on the theories of Gradient Flows and of Optimal Transport, see for instance \cite{ags1} and \cite{vi} respectively.

To get started, we give the following basic definition.

\begin{definition}\label{ch-dir-energy} Let $p\in [1,\infty)$.
We define the \textit{Cheeger--Dirichlet Energy} ${\bf E}_{p}:L^{2}({\X})\rightarrow[0,\infty]$
as the functional given by
\[
{\bf E}_{p}(f)\coloneqq\begin{cases}
{\displaystyle \frac{1}{p}\int_{{\X}}|Df|^{p}\text{d}\mu,} & f\in D^{1,p}(\X) \cap L^{2}({\X})\\
\\
+\infty, & \mathrm{otherwise}.
\end{cases}
\]
\end{definition}

\begin{remark}
 In particular, when $p=2$ the Cheeger--Dirichlet energy ${\bf E}_2$ allows us to characterize the \textit{domain of the Laplace operator} -- or, more simply of the {\it Laplacian} -- as
 \[
  D(\Delta)\coloneqq \left\{ f\in L^2(\X):\; \partial^{-}\mathbf{E}_2(f) \neq \emptyset\right\},
 \]
where $\partial^{-}\mathbf{E}_2(f)$ denotes the {\it sub--differential} of $\bf{E}_2$ at $f$, namely
\[
 \partial^{-}\mathbf{E}_2(f)\coloneqq\left\{v\in L^2(\X);\;\mathbf{E}_2(f)+\int_\X vg \, \text{d}\mu\le\mathbf{E}_2(f+g)\;\forall g\in L^2(\X)\right\},
\]
with the convention $\partial^{-}\mathbf{E}_2(f)=\emptyset$ whenever $\mathbf{E}_2(f)=+\infty$. Thus said, whenever $f\in D(\Delta)$ we define its {\it Laplacian} $\Delta f\in L^2(\X)$ as \[\Delta f\coloneqq -v,\] where $v$ is the element with minimal norm in $\partial^{-}\mathbf{E}_2(f)$.
\end{remark}

By the properties of the minimal weak upper gradient, ${\bf E}_{p}$
is convex and lower semi-continuous; moreover, its domain is dense
in $L^{2}({\X})$.

A key--tool in the definition of $\rcd$ spaces is the concept of ``Infinitesimal Hilbertianity''
of the metric measure space, first introduced in \cite{gi1}.

\begin{definition}\label{inf-Hilb} $\left(\X,d,\mu\right)$
will be said {\textit{infinitesimally Hilbertian}} whenever $W^{1,2}({\X})$
is a Hilbert space. This is equivalent to ask that
the semi--norm $\|\cdot\|_{D^{1,2}(\X)}$
satisfies the parallelogram rule, and that the 2--energy ${\bf E}_{2}$
is a Dirichlet form. \end{definition}

Besides infinitesimal Hilbertianity, in order to come to the main
definition we also need the concept of {\textit{entropy}} of a probability
measure.

\begin{definition}\label{entropy} The \textit{relative entropy} is defined as the functional $\mathscr{E}_{\mu}:\mathscr{P}({\X})\rightarrow\mathbb{R}\cup\{+\infty\}$
given by

\[
\mathscr{E}_{\mu}\left(\mathfrak{m}\right)\coloneqq\begin{cases}
{\displaystyle \int_{{\X}}\rho\mathrm{log}(\rho)\mathrm{d}\mu} & \mathrm{if}\:\mathfrak{m}=\rho\mu\:\mathrm{and}\:\left(\rho\mathrm{log}(\rho)\right)^-\in L^{1}(\X)\\
\\
+\infty, & \mathrm{otherwise.}
\end{cases}
\]
\end{definition}

Last but not least, a fundamental role in the description of $\rcd$ spaces is played by the Wasserstein space of probability measures on $\X$.

\begin{definition}
We denote by $\left(\mathscr{P}_2(\X),W_2\right)$ the {\it Wasserstein space} of probability measures on $\X$ with finite second moment, i.e.\[
\mathscr{P}_2(\X)\coloneqq\left\{\mathfrak{m}\in\mathscr{P}(\X):\;\int_\X d^2\left(x,x_0\right)\text{d}\mathfrak{m}<\infty\;\forall\, x_0\in\X\right\}, 
\]
endowed with the \textit{Wasserstein distance} $W_{2}$ given by
\begin{equation}\label{wass}
 W_{2}^{2}\left(\mathfrak{m}_{1},\mathfrak{m}_{2}\right)\coloneqq \inf \left \{\int_{{\X}\times{\X}}d^{2}(x,y)\mathrm{d}\bm{\gamma}(x,y):\; \bm{\gamma}\in\Gamma(\mathfrak{m}_{1},\mathfrak{m}_{2})\right \},
\end{equation}
where 
\begin{equation*}
\Gamma(\mathfrak{m}_{1},\mathfrak{m}_{2}) := \left \{ \gamma \in \mathscr{P}\left({\X}\times{\X}\right)\text{ such that } \pi_{\#}^{i}(\bm{\gamma})=\mathfrak{m}_{i} \text{ for } i=1,2 \right \}
\end{equation*}
and $\pi^{i}$ denotes the canonical projection over
the $i$--th component, for $i =1, 2$.
\end{definition}

We are now able to give the definition of $\rcd$ spaces.

\begin{definition}\label{RCD} A complete and separable metric measure space $({\X},d,\mu)$ endowed with a non--negative Radon measure $\mu$ will be called an $\rcd$ {\normalfont{space}}, for some $K\in\mathbb{R}$,
if it is infinitesimally Hilbertian and, for every $\lambda,\nu\in\mathscr{P}_{2}({\X})$
with finite relative entropy, there exists a $W_{2}$-geodesic $\mathfrak{m}_{t}$
with $\mathfrak{m}_0=\lambda$
and $\mathfrak{m}_1=\nu$
and such that, for every $t\in[0,1]$,
\[
\mathscr{E}_{\mu}\left(\mathfrak{m}_{t}\right)\le(1-t)\mathscr{E}_{\mu}\left(\lambda\right)+t\mathscr{E}_{\mu}\left(\nu\right)-\frac{K}{2}t(1-t)W_{2}^{2}\left(\lambda,\nu\right).
\]

\end{definition}

Another tool which will be of great use to us is the ``heat flow'', {whose analysis has been discussed extensively in \cite{ags3}}.

\medskip

\begin{definition}\label{heatflow} The \textit{heat flow}
$\h_{t}$, $t\ge0$, is the $L^2$--gradient flow of the Cheeger--Dirichlet
2--energy ${\bf E}_{2}$.\end{definition}

As observed in {\cite{ags3,gi2}}, the theory of gradient flows ensures the
existence and uniqueness of the heat flow as a 1--parameter semigroup
$\left(\text{h}_{t}\right)_{t\ge0}$, $\text{h}_{t}:L^{2}(\X)\rightarrow L^{2}(\X)$,
such that for every $f\in L^{2}(\X)$ the curve $t\mapsto\text{h}_{t}(f)$
is continuous on $[0,\infty)$, absolutely continuous on $(0,\infty)$
and moreover fulfills the differential equation

\[
\frac{\text{d}}{\text{d}t}\text{h}_{t}(f)=\Delta\left(\text{h}_t f\right)
\]

for almost every $t>0$, which means $\text{h}_{t}(f)\in D(\Delta)$
for every $f\in L^{2}(\X)$ and for every $t>0$. In addition, one has
$\left\Vert \h_{t}f-f\right\Vert _{L^{2}(\X)}\rightarrow0$
as $t\rightarrow0$ for any $f\in L^{2}(\X)$.

\sloppy

The infinitesimal Hilbertianity of $\mathsf{RCD}(K,\infty)$ spaces ensures that, in our setting, $\left(\text{h}_{t}\right)_{t\ge0}$
defines a semigroup of linear and self-adjoint operators.

Also, from the analysis carried on in \cite{ags3}, we have that for
every $p\in[1,\infty]$ it holds
\begin{equation}
\left\Vert \text{h}_{t}f\right\Vert _{L^{p}(\X)}\le\left\Vert f\right\Vert _{L^{p}(\X)}\label{eq:ht-contraction}
\end{equation}
for every $t\ge0$ and for every $f\in L^{2}(\X)\cap L^{p}(\X)$.
Then, by a density argument we can uniquely extend the heat flow to
a family of linear and continuous operators $\text{h}_{t}:L^{p}(\X)\rightarrow L^{p}(\X)$
of norm bounded by 1 for every $p\in[1,\infty]$, as the contraction
results proved in \cite{agmr} and \cite{ags2} showed.

In regard to our discussion, the most important property of the heat
flow is the \textit{Bakry--Émery contraction estimate}
\begin{equation}\label{BE-first}
\left|D\text{h}_{t}f\right|^{2}\le e^{-2Kt}\text{h}_{t}\left(\left|D f\right|^{2}\right),
\end{equation}
$\mu$--almost everywhere for every $t\ge0$ and
for every $f\in W^{1,2}(\X)$, see \cite{ags2}, \cite{gko} and the seminal paper \cite{ba}.

Actually, we shall use the Bakry--Émery estimate in its ``self--improved''
version established by \cite{sa} in the $\rcd$ setting:
\begin{equation}
\left|D\text{h}_{t}f\right|\le e^{-Kt}\text{h}_{t}\left(|D f|\right) \label{eq:BE}
\end{equation}
$\mu$--almost everywhere for every $t\ge0$ and for every $f\in W^{1,2}(\X)$.

It is important to notice that, as explained in \cite[Section 6]{ags2}, \cite[Section 4]{ags5}, and then ultimately proved in \cite[Section 4]{sa}, the Bakry--\'Emery estimate and the $\rcd$ condition are actually equivalent requirements.

We also recall that $\h_t$ maps contractively $L^\infty(\X)$ in $C_b(\supp(\mu))$, by \cite[Theorem 6.1]{ags2} and \cite[Theorem 7.1]{agmr}, and that 
\begin{equation}\label{htf-le-htg}
 \forall\,f,g\in L^2(\X),\:f\le g+c\:\mathrm{for\:some}\;c\in\mathbb{R}\;\mathrm{implies}\;\h_t f\le \h_t g+c,
\end{equation}
by \cite[Theorem 4.16]{ags3}. 

The lower curvature bound implies also that
\begin{equation}\label{ht-uno}
 \h_t 1 = 1.
\end{equation}
Indeed, by \cite[Theorem 4.24]{st2}, when the metric {measure} space $(\X,d,\mu)$ has curvature bounded from below by some $K\in\mathbb{R}$, then  for every $ x\in\X$ and every $\rho>0$ there is $c>0$ such that 
\begin{equation}\label{ball-exp}\mu\left(B_{\rho}(x)\right)\le ce^{c\rho^2},\end{equation}
which in turn yields \eqref{ht-uno} by the remarks after \cite[Theorem 4]{st1}.

Lastly, another property of the heat flow which will be useful to us is the
following:
\begin{equation}
c(K,t)\left|D\text{h}_{t}f\right|^{2}\le\text{h}_{t}f^{2}-\left(\text{h}_{t}f\right)^{2},\label{eq:BE-1}
\end{equation}
valid $\mu$--almost everywhere in ${\X}$, for every $t\ge0$ and
for every $f\in L^{2}({\X})$ for some positive constant $c=c(K,t)$ (see \cite{ags5,sa}).

\section{Characterization of $BV$ Functions}\label{BV}

In the rest of the paper we shall make an explicit use of $BV$ functions. To this aim, for the first part of this section we shall recall the derivation approach of \cite{di}, and then we will see that the differential machinery discussed in Section \ref{sec-diff-str} entitles us to give an equivalent definition of $BV$ functions in terms of $\frd^\infty(\X)$ vector fields.

We recall that the derivation approach of \cite{di} is equivalent to the relaxation procedure performed by \cite{miranda} and also to the ``weak $BV$'' space described in \cite{ad}; we refer in particular to \cite[Section 7]{ad} and \cite[Section 7.3]{di} for an extensive discussion on the equivalences of $BV$ spaces.

Let $(\X,d,\mu)$ be a complete and separable metric measure space endowed with a non--negative Radon measure $\mu$ which is finite on bounded sets.

\subsection{The derivation approach}

\begin{definition}\label{der}
We say that a linear map $\delta:\Lip_\bs(\X)\to L^0(\X)$ is a \textit{Lipschitz derivation}, and we write $\delta \in \der(\X)$, if it satisfies the following properties:
\sloppy
 \begin{enumerate}
 \item Leibniz rule: $\delta(fg)=f\delta(g)+g\delta(f)$ for every $f,g\in\Lip_\bs(\X)$;
 \item Weak locality: there exists some function $g\in L^0(\X)$ such that $$|\delta(f)|(x)\le g(x)\, \Lip_{\text{a}}(f)(x)$$ for $\mu$--almost every $x\in\X$ and for all $f\in\Lip_\bs(\X)$, where $\Lip_{\text{a}}(f)(x)$ denotes the asymptotic Lipschitz constant of $f$ at $x$, as defined in \eqref{lip-a}.
 \end{enumerate}
The smallest function $g$ satisfying \textit{ii)} will be denoted by $|\delta|$. In the case $|\delta|\in L^p(\X)$, we shall write $\delta\in L^p(\X)$.
\end{definition}

We notice that, thanks to the weak locality {condition in} the definition of derivations, any $\delta \in \der(\X)$ can be extended to $\Lip_{\rm loc}(\X)$.

We define the support of a derivation $\delta \in \der(\X)$ in the usual distributional sense. Indeed, we say that $\delta = 0$ on a open set $V$ if $\delta(f) = 0$ for any $f \in \Lip_{\bs}(\X)$ such that $\supp(f) \subset V$, and then we set
\begin{equation*}
\supp(\delta) := \X \setminus \bigcup_{\substack{V \text{ open},\\ \delta = 0 \text{ on } V}} V.
\end{equation*}

It is possible to define the divergence of a derivation by requiring that the usual integration by parts formula holds.

\begin{definition}
 Given $\delta\in\der(\X)$ with $\delta\in L^1_{\text{loc}}(\X)$, we define its \textit{divergence} as the operator $\div(\delta):\Lip_\bs(\X)\to\mathbb{R}$ such that
 \[
  f\mapsto-\int_\X\delta(f)\d\mu.
 \]
 We say that $\div(\delta)\in L^p(\X)$ if this operator admits an integral representation via a unique $L^p(\X)$ function $h=\div(\delta)$:
 \[
  \int_\X \delta(f)\d\mu=-\int_\X hf\d\mu.
 \]

\end{definition}

\medskip

For all $p,q\in[1,\infty]$ we shall set
\[
 \der^p(\X)\coloneqq \left\{\delta\in\der(\X);\;\delta\in L^p(\X)\right\}
\]
and
\[
 \der^{p,q}(\X)\coloneqq\left\{\delta\in\der(\X);\;\delta\in L^p(\X),\;\div(\delta)\in L^q(\X)\right\}.
\]

When $p=\infty=q$, we write $\der_b(\X)$ instead of $\der^{\infty,\infty}(\X)$.

The domain of the divergence is then characterized as
\[
 D(\div)\coloneqq\left\{\delta\in\der(\X);\;|\delta|,\div(\delta)\in L^1_\text{loc}(\X)\right\},
\]
which contains $\der^{p,q}(\X)$ for all $p,q\in[1,\infty]$.

With these tools available, we may proceed as in \cite{di} to give the definition of $BV$ functions.

\begin{definition}\label{bv-der}
 Let $u\in L^1(\X)$. We say that $u$ is of \textit{bounded variation} in $\X$, and we write $u \in BV(\X)$, if there is a linear and continuous map $L_u :\der_b(\X)\to\mathbf{M}(\X)$ such that
 \begin{equation}\label{eq-bv-der}
  \int_\X \d L_u(\delta)=-\int_\X u\div(\delta)\d\mu
 \end{equation}
for all $\delta\in\der_b(\X)$ and satisfying $L_u(h\delta)=hL_u(\delta)$ for any $h\in\Lip_b(\X)$. 
 
We say that a measurable set $E$ has {\em finite perimeter} in $\X$ if $\uno_{E} \in BV(\X)$.
\end{definition}

As observed in \cite{di}, the above characterization is well-posed since, if we take any two maps $L_u,\tilde{L}_u$ as in Definition \ref{bv-der}, the Lipschitz-linearity of derivations ensures that $L_u(\delta)=\tilde{L}_u(\delta)$ for all $\delta\in\der_b(\X)$. 
In other words, when $u\in BV(\X)$ the measure-valued map $L_u$ is uniquely determined and we shall write $Du(\delta) := L_u(\delta)$.

We combine now the statements of \cite[Theorems 7.3.3 and 7.3.4]{di}, which allow us to recover the usual properties of $BV$ functions, including the representation formula for the total variation as the supremum of 
\begin{equation*}
\int_\Omega u\div(\delta)\d\mu
\end{equation*}
taken over all derivations $\delta$ such that $\supp(\delta) \Subset \Omega$. 

\begin{theorem}{\normalfont{\cite[Theorems 7.3.3 - 7.3.4]{di}}}\label{char-bv}
 Assume $u\in BV(\X)$. There exists a non--negative, finite Radon measure $\nu\in\mathbf{M}(\X)$ such that for every Borel set $B\subset\X$ one has
 
 \begin{equation}\label{Du}
  \int_B\d Du(\delta)\le\int_B |\delta|^*\d\nu\quad \forall \delta\in\der_b(\X),
 \end{equation}

 where $|\delta|^*$  denotes the upper-semicontinuous envelope of $|\delta|$. The least measure $\nu$ satisfying \eqref{Du} will be denoted by $\|Du\|$, the total variation of $u$. Moreover, 
 \[
  \|Du\|(\X)=\sup \left\{|Du(\delta)(\X)|;\;\delta\in\der_b(\X),\;|\delta|\le1\right\}.
 \] 
 Finally, the classical representation formula for $\|Du\|$ holds, in the sense that, if $\Omega\subset\X$ is any open set, then 
\begin{equation} \label{eq:tot_var_duality}
 \|Du\|(\Omega)=\sup\left\{\int_\Omega u\div(\delta)\d\mu;\;\delta\in\der_b(\X),\;\supp(\delta)\Subset\Omega,\;|\delta|\le1\right\}.
 \end{equation}
\end{theorem}

\smallskip
\begin{remark} \label{rem:prop_set_fin_per}
As a consequence of \eqref{eq:tot_var_duality} and of the Leibniz rule for derivations, we see that, if $E$ is a set of finite perimeter in $\X$, then
\begin{equation} \label{eq:concentration_perimeter_measure}
\|D \uno_{E}\|(\X \setminus \partial E) = 0.
\end{equation}
Hence, if $\mu(\partial E) = 0$, then the measures $\mu$ and $\|D \uno_{E}\|$ are mutually singular.
In addition, we notice that, if $E$ is a measurable set,
\begin{equation*}
\int_\Omega \uno_{E} \div(\delta)\d\mu = - \int_\Omega \uno_{E^{c}} \div(\delta)\d\mu
\end{equation*}
for any open set $\Omega$ and $\delta\in\der_b(\X),\;\supp(\delta)\Subset\Omega$. Hence, if we have $\uno_{E} \in BV(\X)$ or $\uno_{E^c} \in BV(\X)$, then the Radon measures $\|D \uno_{E}\|$ and $\|D \uno_{E^c}\|$ are well defined and coincide.
\end{remark}

We recall now an approximation result for $BV$ functions analogous to Proposition \ref{prop:Lip_approx_p} for Sobolev spaces.

\begin{proposition} \label{prop:Lip_approx_BV}
If $f \in BV(\X)$, then there exists a sequence $(f_{k})_{k \in \N}$ in $\Lip_{\bs}(\X)$ such that $f_{k} \to f$ in $L^1(\X)$ and $\Lip_{\a}(f_k)\mu \weakto \|Df\|$ in $\M(\X)$.
\end{proposition}
\begin{proof}
This result follows by combining \cite[Proposition 4.5.6 and Theorem 7.3.7]{di}.
\end{proof}

\subsection{The approach via vector fields. Proof of the equivalence}

Before discussing our equivalent characterization of the $BV$ space, we need to introduce the notion of Sobolev derivations following the definition given in \cite{gi2}.

\begin{definition}\label{sob-der}
Let $p,q\in[1,\infty]$ be two conjugate exponents. A linear map $L: D^{1,p}(\X)\to L^1(\X)$ such that
 \begin{equation} \label{eq:sob-der}
  |L(f)|\le\ell |Df|
 \end{equation}
 $\mu$--almost everywhere for every $f\in D^{1,p}(\X)$ and for some $\ell\in L^q(\X)$ will be called a $q$--\textit{Sobolev derivation}. The set of $q$--Sobolev derivations $L:D^{1,p}(\X)\to L^1(\X)$ will be denoted by $\sder_q(\X)$.
\end{definition}

Of course, the usual calculus rules hold for this definition as well. It is interesting to notice that, roughly speaking, the space of $q$--Sobolev derivations contains the space of $q$--summable Lipschitz derivations with divergence in $L^{q}(\X)$ for $q > 1$. In order to make this statement rigorous, we start with the following technical lemma. 

\begin{lemma}\label{lem:inclusion_Sob_der}
Let $q \in (1, \infty]$ and $p \in [1, \infty)$ be conjugate exponents. Any $\delta \in \der^{q, q}(\X)$ may be uniquely extended to a linear continuous map $\tilde{\delta} : D^{1,p}(\X)\cap (L^p(\X)\cup L^{\infty}(\X)) \to L^{1}(\X)$ satisfying  
\begin{equation} \label{eq:action_delta_ext_lemma}
\int_{\X} h \, \tilde{\delta}(f) \d \mu = - \int_{\X} f \left ( h \div(\delta) + \delta(h) \right ) \d \mu
\end{equation}
for all $f \in D^{1,p}(\X)\cap (L^p(\X)\cup L^{\infty}(\X)), h \in \Lip_{\bs}(\X)$, and
\begin{equation} \label{eq:sob-der_ext}
  |\tilde{\delta}(f)|\le\ell |Df|
 \end{equation}
$\mu$--almost everywhere for every $f\in D^{1,p}(\X)\cap (L^p(\X)\cup L^{\infty}(\X))$, where 
\begin{equation*}
\ell = \begin{cases} |\delta| & \text{ if } q \in (1, \infty),  \\ 
\|\delta\|_{L^{\infty}(\X)} & \text{ if } q = \infty.
\end{cases}
\end{equation*}
\end{lemma}
\allowdisplaybreaks
\begin{proof}
Let at first $q \in (1, \infty)$ and let $\delta \in \der^{q, q}(\X)$. Then, for any $f \in \Lip_\bs(\X)$ we have
\begin{equation} \label{eq:point_bound_Lip}
|\delta(f)| \le |\delta|\cdot\Lip_{\text{a}}(f), \text{ with } |\delta| \in L^{q}(\X),
\end{equation}
and 
\begin{equation} \label{eq:IBP_div_der_q}
\int_{\X} \delta(f) \d \mu = - \int_{\X} \div(\delta) \, f \d \mu, \text{ with } \div(\delta) \in L^{q}(\X).
\end{equation}
We start by showing that we can extend $\delta$ to functions in $W^{1, p}(\X) := D^{1, p}(\X)\cap L^p(\X)$, where $p$ is the conjugate exponent of $q$, in such a way that \eqref{eq:sob-der_ext} holds with $\ell = |\delta|$. Hence, let $f \in W^{1, p}(\X)$. Thanks to Proposition \ref{prop:Lip_approx_p}, we know that there exists a sequence $(f_{k})_{k \in \N} \subset \Lip_{\bs}(\X)$, such that
\begin{equation*}
f_{k} \to f \text{ in } L^{p}(\X) \text{ and } \Lip_{\text{a}}(f_{k}) \to |D f| \text{ in } L^{p}(\X).
\end{equation*}
Hence, thanks to \eqref{eq:point_bound_Lip}, it is easy to see that $(\delta(f_{k}))_{k \in \N}$ is a bounded sequence in $L^{1}(\X)$, so that $(\delta(f_{k}) \mu)_{k \in \N}$ is uniformly bounded in $\M(\X)$. In addition, this sequence is also uniformly tight. Indeed, since $|\delta| \in L^q(\X)$, for all $\eps > 0$ there exists a compact set $K_{\eps}$ such that $$\int_{\X \setminus K_{\eps}} |\delta|^q \,d \mu \le \left ( \frac{\eps}{C_f} \right )^{q},$$
where $$C_f := \sup_{k} \|\Lip_{\a}(f_k)\|_{L^{p}(\X)} < \infty.$$
Thanks to \eqref{eq:point_bound_Lip} and H\"older's inequality, we see that
\begin{align*}
\|\delta(f_{k}) \mu\|(\X \setminus K_{\eps}) & = \int_{\X \setminus K_{\eps}} |\delta(f_{k})| \, d \mu \le \int_{\X \setminus K_{\eps}} |\delta| \Lip_{\text{a}}(f_k) \, d \mu \\
& \le \left (\int_{\X \setminus K_{\eps}} |\delta|^q \,d \mu\right )^{\frac{1}{q}} \|\Lip_{\a}(f_k)\|_{L^{p}(\X)} \le \eps.
\end{align*}
Therefore, Theorem \ref{thm:bog_weak_conv} yields the existence of a subsequence $(\delta(f_{k_{j}}) \mu)_{j \in \N}$ weakly converging to a Radon measure, which we denote by $\delta_{f}$. By the Leibniz rule (point 1 in Definition \ref{der}) and \eqref{eq:IBP_div_der_q}, the measure $\delta_f$ acts on $h \in \Lip_{\bs}(\X)$ in the following way:
\begin{align*}
\int_{\X} h \d \delta_{f} & = \lim_{j \to + \infty} \int_{\X} h \delta(f_{k_{j}}) \d \mu = - \lim_{j \to + \infty} \int_{\X} f_{k_{j}} \left ( h \div(\delta) + \delta(h) \right ) \d \mu \\
& = - \int_{\X} f \left ( h \div(\delta) + \delta(h) \right ) \d \mu. 
\end{align*}
This shows that $\delta_{f}$ does not depend on the converging subsequence and is indeed unique.
In particular, for all $h \in \Lip_{\bs}(\X)$, we have
\begin{equation*}
\left |\int_{\X} h \, \d \delta_{f} \right |= \lim_{k \to + \infty} \left | \int_{\X} h \, \delta(f_{k}) \d \mu \right | \le \liminf_{k \to + \infty} \int_{\X} |h| \, |\delta| \Lip_{\text{a}}(f_{k}) \d \mu = \int_{\X} |h| \, |\delta| \, |D f| \d \mu,
\end{equation*}
from which we deduce that $\|\delta_{f}\| \le |\delta| |Df| \mu$ in the sense of Radon measures. Therefore, $\delta_{f}$ is indeed absolutely continuous with respect to $\mu$, and we denote by $\tilde{\delta}(f)$ its $L^1$ density. Thus, it is clear that $\tilde{\delta}(f)$ is well defined for all $f \in W^{1, p}(\X)$ and it is a continuous linear operator from $W^{1,p}(\X)$ to $L^1(\X)$ satisfying \eqref{eq:action_delta_ext_lemma} and \eqref{eq:sob-der_ext} with $\ell = |\delta|$. In addition, we have $\tilde{\delta}(f) = \delta(f)$ for all $f \in \Lip_{\bs}(\X)$, since, thanks to \eqref{eq:action_delta_ext_lemma}, the Leibniz rule (point 1 in Definition \ref{der}) and \eqref{eq:IBP_div_der_q}, for all $h \in \Lip_{\bs}(\X)$ we have
\begin{equation*}
\int_{\X} h \tilde{\delta}(f) \d \mu = - \int_{\X} f \left ( h \div(\delta) + \delta(h) \right ) \d \mu = \int_{\X} h \delta(f) \d \mu.
\end{equation*}
Thus, $\tilde{\delta}$ is the unique extension of $\delta$ to $W^{1,p}(\X)$. In the case $f \in D^{1,p}(\X) \cap L^{\infty}(\X)$, we may not have $f \in L^p(\X)$, and so we need to proceed by approximation with Lipschitz cutoff functions. For some fixed $x_0 \in \X$ and all $R > 0$, we define
\begin{equation*}
\eta_{R}(x) := \uno_{B_R(x_0)}(x) + \left ( 2 - \frac{d(x, x_0)}{R}\right ) \uno_{B_{2R}(x_0) \setminus B_R(x_0)}(x).
\end{equation*}
It is easy to see that $\eta_R \in \Lip_{\bs}(\X) \subset D^{1,p}(\X) \cap L^{\infty}(\X)$ and it satisfies
\begin{equation*} 
\eta_R(x) \equiv 1 \text{ on }\overline{B_R(x_0)}, \eta_R(x) \equiv 0 \text{ on }\X \setminus B_{2R}(x_0) \text{ and }|D \eta_{R}| \le \frac{1}{R} \uno_{B_{2R}(x_{0}) \setminus B_R(x_0)}.
\end{equation*}
Therefore, we have 
\begin{equation} \label{eq:Leibniz_ineq}
|D(f \eta_R)| \le \eta_R |Df| +  |f| |D \eta_R|
\end{equation}
by the weak Leibniz rule for the minimal $p$-weak upper gradient, see point 2 of Remark \ref{rmk-prop-upgrad}. This implies $f \eta_R \in W^{1,p}(\X)$, so that $\tilde{\delta}(f \eta_R)$ is well defined for all $R > 0$. In addition, we notice that the sequence $(\tilde{\delta}(f \eta_R))_{R > 0}$ is Cauchy with respect to $L^1_{\rm loc}$--convergence: indeed, for any $R, r > 0$, by \eqref{eq:sob-der_ext} and \eqref{eq:Leibniz_ineq} we have
\begin{equation*}
|\tilde{\delta}(f (\eta_R - \eta_r))| \le |\delta|\left (|\eta_R - \eta_r| |Df| + |f| ( |D \eta_R| + |D \eta_r|)\right),
\end{equation*}
so that for any bounded set $B \subset \X$ we obtain
\begin{equation*}
\int_{B} |\tilde{\delta}(f (\eta_R - \eta_r))| \, d \mu = 0
\end{equation*}
for all $R > r > 0$ such that $B \subset B_r(x_0)$. Hence, by considering a covering of $\X$ of balls $\{B_k(x_0)\}_{k \in \N}$ and a gluing argument, we may deduce the existence of a subsequence $(\tilde{\delta}(f \eta_{R_{j}}))_{j \in \N}$ converging to a function $\overline{\delta}(f)$ in $L^1_{\rm loc}(\X)$, a priori depending on the subsequence.
In addition, since $\tilde{\delta}$ satisfies \eqref{eq:sob-der_ext} on $W^{1,p}(\X)$, we obtain the estimate 
\begin{equation*}
|\tilde{\delta}(f \eta_{R_j})| \le |\delta| \left ( \eta_{R_j} |Df| +  |f| |D \eta_{R_j}| \right ),
\end{equation*}
which easily implies $|\overline{\delta}(f)| \le |\delta| |Df|$ $\mu$--almost everywhere, so that we also deduce that $\overline{\delta}(f) \in L^1(\X)$. In addition, thanks to \eqref{eq:action_delta_ext_lemma}, for all $h \in \Lip_{\bs}(\X)$ we have
\begin{align*}
\int_{\X} h \overline{\delta}(f) \, d \mu & = \lim_{j \to + \infty} \int_{\X} h \tilde{\delta}(f \eta_{R_j}) \, d \mu = - \lim_{j \to + \infty} \int_{\X} h f \eta_{R_j} \div(\delta) + f \eta_{R_j} \delta(h) \, d \mu \\
& = - \int_{X} f \left ( h \div(\delta) + \delta(h) \right ) \, d\mu. 
\end{align*} 
This shows that $\overline{\delta}(f)$ does not depend on the subsequence and is unique. In this way, we extend $\tilde{\delta}$ to a linear operator from $D^{1,p}(\X) \cap L^{\infty}(\X)$ into $L^1(\X)$ satisfying \eqref{eq:action_delta_ext_lemma} and \eqref{eq:sob-der_ext} with $\ell = |\delta|$, which we still denote by $\tilde{\delta}$, with a little abuse of notation.

We now consider the case $q = \infty$ and $p = 1$, and we start by assuming $f \in W^{1, 1}(\X) := D^{1,1}(\X) \cap L^{1}(\X)$. We notice that $f \in BV(\X)$, with $\|D f\| = |D f| \mu$. Thanks to Proposition \ref{prop:Lip_approx_BV}, we know that there exists a sequence $(f_{k})_{k \in \N}\subset\Lip_{\bs}(\X)$ such that
\begin{equation*}
f_{k} \to f \text{ in } L^{1}(\X) \text{ and } \Lip_{\text{a}}(f_{k})\mu \weakto \|D f\|.
\end{equation*}
Therefore, \eqref{eq:point_bound_Lip} implies
\begin{equation*}
|\delta(f_k)| \le |\delta| \Lip_{\text{a}}(f_{k}) \le \|\delta\|_{L^{\infty}(\X)} \Lip_{\text{a}}(f_{k})
\end{equation*} 
$\mu$--almost everywhere. Hence, Theorem \ref{thm:bog_weak_conv} implies that the sequence $(\Lip_{\text{a}}(f_{k})\mu)_{k \in \N}$ is uniformly bounded and tight, since it is weakly converging, and thus we conclude that the 
sequence $(\delta(f_k))_{k \in \N}$ is uniformly bounded and tight, again by Theorem \ref{thm:bog_weak_conv}. Therefore, we can proceed as we did in the case $q \in (1, \infty)$ in order to deduce the existence of a Radon measure $\delta_{f} \in \M(\X)$ which is the weak limit of a subsequence $(\delta(f_{k_j}))_{j \in \N}$. Then, we can employ the Leibniz rule to show that $\delta_{f}$ does not depend on the converging subsequence and that it is unique. In addition, we notice that for all $h \in \Lip_{\bs}(\X)$ \eqref{eq:point_bound_Lip} implies
\begin{align*}
\left |\int_{\X} h \, \d \delta_{f} \right | & = \lim_{k \to + \infty} \left | \int_{\X} h \, \delta(f_{k}) \d \mu \right | \le \liminf_{k \to + \infty} \int_{\X} |h| \, |\delta| \Lip_{\text{a}}(f_{k}) \d \mu \\
& \le \|\delta\|_{L^{\infty}(\X)} \liminf_{k \to + \infty} \int_{\X} |h| \Lip_{\text{a}}(f_{k}) \d \mu = \|\delta\|_{L^{\infty}(\X)} \int_{\X} |h| \, |D f| \d \mu,
\end{align*}
from which we deduce that $\|\delta_{f}\| \le \|\delta\|_{L^{\infty}(\X)} |Df| \mu$ in the sense of Radon measures. Therefore, $\delta_{f}$ is indeed absolutely continuous with respect to $\mu$, and we denote by $\tilde{\delta}(f)$ its $L^1$ density. Thus, it is clear that $\tilde{\delta}$ is a linear continuous operator from $W^{1,1}(\X)$ to $L^1(\X)$ satisfying \eqref{eq:action_delta_ext_lemma} and \eqref{eq:sob-der_ext} with $\ell = \|\delta\|_{L^{\infty}(\X)}$. In addition, we can show that $\tilde{\delta}$ is an extension of $\delta$, arguing by approximation and by the Leibniz rule as in the case $q \in (1, \infty)$. Finally, we can extend $\tilde{\delta}$ to a continuous linear operator from $D^{1,1}(\X) \cap L^{\infty}(\X)$ to $L^1(\X)$ arguing by approximation with cutoff functions as we did in the case $q \in (1, \infty)$, with the only difference that now in \eqref{eq:sob-der_ext} we have $\ell = \|\delta\|_{L^{\infty}(\X)}$.
\end{proof}

\begin{remark}
We wish to underline the fact that the approximation argument of Lemma \ref{lem:inclusion_Sob_der} does not work in the case $q = 1$ and $p = \infty$, since in general it is not possible to approximate functions in $W^{1, \infty}(\X)$ with sequences of Lipschitz functions with bounded support. In addition, unless $(\X, d, \mu)$ is the space $\R^n$ endowed with the standard Euclidean distance and Lebesgue measure $\Leb{n}$, we cannot prove that each function in $W^{1, \infty}(\X)$ admits a representative in $\Lip_{b}(\X)$, given that it is enough to choose as $\X$ a suitable subset of $\R^n$, for $n \ge 2$, to obtain counterexamples (see \cite[Section 2.3]{afp}). Hence, not even a different approximation approach based only on the use of a cutoff sequence seems to be working, without some other assumptions on $(\X, d, \mu)$.
\end{remark}

We proceed now to show, in the case $q > 1$, the existence of an extension of a derivation in $\der^{q, q}(\X)$ to the whole Sobolev-Dirichlet class $D^{1,p}(\X)$, where $p$ is the conjugate exponent to $q$; in this way proving that each $\delta \in \der^{q, q}(\X)$ admits an extension to $\sder_q(\X)$.

\begin{lemma} \label{lem:extension_Sob_der}
Let $q \in (1, \infty]$. Any $\delta \in \der^{q, q}(\X)$ may be extended to a $q$-Sobolev derivation $\overline{\delta} \in \sder_q(\X)$ satisfying \eqref{eq:sob-der} with 
\begin{equation} \label{eq:ell_definition}
\ell := \begin{cases} |\delta| & \text{ if } q \in (1, \infty),  \\ 
\|\delta\|_{L^{\infty}(\X)} & \text{ if } q = \infty.
\end{cases}
\end{equation}
\end{lemma}
\begin{proof}
Let $p \in [1, \infty)$ be the conjugate exponent to $q$ and $f \in D^{1,p}(\X)$. For all $k \in \N$ we set
\begin{equation*}
T_{k}(f) := \begin{cases} k & \text{ if } f \ge k, \\
f & \text{ if } - k < f < k, \\
- k & \text{ if } f \le - k.
\end{cases}
\end{equation*}
Clearly, $T_{k}(f) \in D^{1,p}(\X)\cap L^{\infty}(\X)$, and, thanks to point 4 of Remark \ref{rmk-prop-upgrad}, we have $|DT_{k}(f)| =|D f| \uno_{\{|f| < k\}}$ $\mu$--almost everywhere. Hence, the unique extension $\tilde{\delta}$ of $\delta$ given by Lemma \ref{lem:inclusion_Sob_der} can be applied to $T_k(f)$, and \eqref{eq:sob-der_ext} yields
\begin{equation} \label{eq:k_truncation_estimate}
|\tilde{\delta}(T_k(f))| \le \ell |D T_k(f)| = \ell |D f| \uno_{\{|f| < k\}} \le \ell |D f|,
\end{equation}
\sloppy
where $\ell$ is defined as above. Since $\ell |D f| \in L^1(\X)$, we deduce that the sequence $( \tilde{\delta}(T_k(f)) \mu)_{k \in \N}$ is uniformly bounded in $\M(\X)$. In addition, \eqref{eq:k_truncation_estimate} easily implies also the uniform tightness, since, for all $\eps > 0$, there exists a compact set $K_{\eps}$ such that 
\begin{equation*}
\int_{\X \setminus K_{\eps}} \ell |D f| \d \mu < \eps, 
\end{equation*}
so that
\begin{equation*}
\int_{\X \setminus K_{\eps}} |\tilde{\delta}(T_k(f))| \d \mu \le \int_{\X \setminus K_{\eps}} \ell |D f| \d \mu < \eps.
\end{equation*}
Therefore, Theorem \ref{thm:bog_weak_conv} yields the existence of a subsequence $( \tilde{\delta}(T_{k_j}(f)) \mu)_{j \in \N}$ weakly converging to a Radon measure, which we denote by $\overline{\delta}_{f}$. Thanks to \eqref{eq:k_truncation_estimate}, we deduce that $\|\overline{\delta}_{f}\| \le \ell |D f| \mu$, since, for all $h \in C_{b}(\X)$, we have
\begin{equation*}
\left | \int_{\X} h \d \overline{\delta}_{f} \right | = \lim_{j \to +\infty} \left | \int_{\X} h \tilde{\delta}(T_{k_j}(f)) \d \mu \right | \le \liminf_{j \to +\infty} \int_{\X} |h| \ell |D f| \uno_{\{|f| < k_{j}\}} \d \mu \le \int_{\X} |h| \ell |D f| \d \mu.
\end{equation*}
Hence, there exists a function $\overline{\delta}(f) \in L^1(\X)$ such that $\overline{\delta}_{f} = \overline{\delta}(f) \mu$ and $|\overline{\delta}(f)| \le \ell |D f|$. Thus, we constructed a linear continuous map $\overline{\delta} : D^{1,p}(\X) \to L^1(\X)$ satisfying \eqref{eq:sob-der}. Finally, it is clear that, if $f \in D^{1,p}(\X) \cap L^{\infty}(\X)$, then $\tilde{\delta}(T_k(f)) = \tilde{\delta}(f)$ for $k$ large enough, so that $\overline{\delta} = \tilde{\delta}$ on $D^{1,p}(\X) \cap L^{\infty}(\X)$, and this shows that $\overline{\delta}$ is an extension of the derivation $\delta$, since $\tilde{\delta}$ is the unique extension of $\delta$ to $D^{1,p}(\X) \cap L^{\infty}(\X)$, thanks to Lemma \ref{lem:inclusion_Sob_der}.
\end{proof}

We wish to stress the fact that the extension $\overline{\delta}$ given by Lemma \ref{lem:extension_Sob_der} is a priori not unique. However, we will prove that it is indeed unique, thanks to the one-to-one correspondence between the derivations in $\der^{q,q}(\X)$ and the fields in $\frd^q(\X)$, which is treated in the following statements.

\begin{lemma}\label{isomorphism_1}
Let $q\in[1,\infty]$. Any vector field $X\in\frd^q(\X)$ induces a unique derivation $\delta_X\in\der^{q,q}(\X)$ given by $\delta_X = X \circ \d$, and
 \begin{equation}\label{est-rmk-37}
  |\delta_X(\varphi)|\le \Lip_{\text{a}}(\varphi) |X|
 \end{equation}
$\mu$--almost everywhere for every $\varphi\in\Lip_{\mathrm{bs}}(\X)$, so that $|\delta_X| \le |X|$. In addition, we have $\div(\delta_{X}) = \div(X)$ and $\delta_X \in \sder^q(\X)$, satisfying \eqref{eq:sob-der} with $\ell = |X|$.
\end{lemma}
\begin{proof}
It is easy to see that $\delta_X$ defined as above satisfies the Leibniz rule in Definition \ref{der}, thanks to Proposition \ref{prop:diff_Leibniz}, and that for all $\varphi\in\Lip_\bs(\X)$ we have $\delta_X(\varphi) = \d\varphi(X)$, so that
 \begin{equation} \label{est-rmk-37_1}
|\d\varphi(X)|\le|\d\varphi|_{*}|X|=|D\varphi||X| \le \Lip_{\text{a}}(\varphi) |X|
 \end{equation}
$\mu$--almost everywhere, since $\Lip_{\text{a}}(\varphi)$ is an upper gradient of $\varphi$. This gives point 2 of Definition \ref{der} with $g = |X|$. Then, it follows that $|\delta_X| \le |X| \in L^{q}(\X)$. In addition, we notice that $\delta_X = X \circ \d$ can be naturally extended to $D^{1,p}(\X)$, thanks to the properties of the differential (Definition \ref{differential}); and, analogously to \eqref{est-rmk-37_1}, for all $f \in D^{1,p}(\X)$ we get 
 \begin{equation*}
|\delta_X(f)| = |\d f(X)|\le|\d f|_{*}|X|=|D f||X| 
 \end{equation*}
$\mu$--almost everywhere, which clearly implies \eqref{eq:sob-der} with $\ell = |X|$. Finally, we notice that, for any $\psi\in\Lip_\bs(\X)$, we have
\begin{equation*}
\int_\X\psi\div(\delta_{X})\d\mu = - \int_\X\delta_{X}(\psi)\d\mu = -\int_\X\d\psi(X)\d\mu = \int_\X\psi\div(X)\d\mu.
\end{equation*}
This means that $\div(\delta_{X}) = \div(X) \in L^{q}(\X)$, thus proving our claim. 
\end{proof} 

We show now the converse statement; that is, the fact that every derivation in $\der^{q,q}(\X)$ induces a vector field $X_\delta\in \frd^q(\X)$, in the case $q > 1$.

\begin{lemma} \label{isomorphism_2}
Let $q \in (1, \infty]$ and $\delta\in\der^{q,q}(\X)$. Then there exists a unique vector field $X_\delta\in \frd^q(\X)$ such that $\delta = X_{\delta} \circ \d$. In addition, we have 
\begin{equation} \label{eq:X_delta_estimate} 
|X_{\delta}| \le \begin{cases} |\delta| & \text{ if } q \in (1, \infty),  \\ 
\|\delta\|_{L^{\infty}(\X)} & \text{ if } q = \infty.
\end{cases} 
\end{equation}
and $\div(X_{\delta}) = \div(\delta)$.
\end{lemma} 
\begin{proof}
Thanks to Lemma \ref{lem:extension_Sob_der}, we may extend $\delta$ to a $q$--Sobolev derivation $\overline{\delta} \in \sder_q(\X)$, a priori not unique. By the theory of $L^p$--normed modules, see \cite{gi2} and Appendix \ref{sec-appendix}, $$L^q(T\X)=\textsc{Hom}\left(L^p\left(T^*\X\right),L^1(\X)\right),$$where $\frac{1}{p}+\frac{1}{q}=1$. As a consequence, using \cite[Theorem 5.5.4]{bu}\footnote{Notice that \cite[Theorem 5.5.4]{bu} is a generalization to any exponent $q \in (1, \infty]$ of \cite[Theorem 2.3.3]{gi2}, which was proved only for the case $q=2$.} and the fact that $\overline{\delta} \in \sder_q(\X)$, we deduce that there exists $X_{\overline{\delta}}\in L^q(T\X)$ such that $\overline{\delta} = X_{\overline{\delta}} \circ \d$, which means that $\delta(\varphi) = \d \varphi(X_{\overline{\delta}})$ for all $\varphi \in \Lip_{\bs}(\X)$, since $\overline{\delta} = \delta$ on $\Lip_{\bs}(\X)$. Moreover, Lemma \ref{lem:extension_Sob_der} implies that, for all $f \in D^{1,p}(\X)$,
\begin{equation*}
|\d f(X_{\overline{\delta}})|= |\overline{\delta}(f)| \le \ell |D f| = \ell |\d f|_* 
\end{equation*}
$\mu$--almost everywhere, where $\ell$ is defined as in \eqref{eq:ell_definition}. Hence, by Remark \ref{dual-mod-norm} and Lemma \ref{lem:L_p_cotangent_generated}, $\mu$--almost everywhere we have
\begin{align*}
 \left\vert X_{\overline{\delta}}\right\vert & \coloneqq {\text{ess-}\!\sup}\left \{ \left\vert L(X_{\overline{\delta}})\right\vert : L = \d f \text{ for } f\in D^{1,p}(\X),\:|L|_{*}\le1 \right \} \\
 & = {\text{ess-}\!\sup}\left \{ \left\vert \d f(X_{\overline{\delta}})\right\vert : f\in D^{1,p}(\X),\:|\d f|_{*}\le1 \right \} \le\ell,
\end{align*}
which implies \eqref{eq:X_delta_estimate}, thanks to \eqref{eq:ell_definition}.

We notice that $X_{\overline{\delta}}$ depends a priori on the choice of the extension $\overline{\delta}$ of $\delta$ to $D^{1,p}(\X)$. However, this dependence is actually illusory. Indeed, assume that there exist two such extensions $\overline{\delta}_1, \overline{\delta}_2$. Then, there exists two fields $X_{\overline{\delta}_1}, X_{\overline{\delta}_2} \in L^q(T\X)$ such that $\overline{\delta}_i = X_{\overline{\delta}_i} \circ \d$ for $i = 1, 2$. However, these extensions $\overline{\delta}_1, \overline{\delta}_2$ coincide with the unique extension $\tilde{\delta}$ of $\delta$ to $W^{1,p}(\X)$ given by Lemma \ref{lem:inclusion_Sob_der}, so that, for all $f \in W^{1,p}(\X)$, we have
\begin{equation*}
\d f(X_{\overline{\delta}_1}) = \overline{\delta}_1(f) = \tilde{\delta}(f) = \overline{\delta}_2(f) = \d f(X_{\overline{\delta}_2}),
\end{equation*}
which clearly implies $\d f(X_{\overline{\delta}_1} - X_{\overline{\delta}_2}) = 0$ for all $f \in W^{1,p}(\X)$. Thus, we conclude that it must be $X_{\overline{\delta}_1} = X_{\overline{\delta}_2}$, by the linearity of the pairing and thanks to Lemma \ref{lem:L_p_cotangent_generated}. This proves the uniqueness of the vector field, which we shall denote simply by $X_{\delta}$.
We conclude the proof by showing that $\div(X_{\delta}) = \div(\delta)$. Let $\psi\in\Lip_\bs(\X)$ and notice that
\begin{equation*}
 \int_\X\psi\div(X_{\delta})\d\mu=-\int_\X\d\psi(X_{\delta})\d\mu=-\int_\X\delta(\psi)\d\mu=\int_\X\psi\div(\delta)\d\mu.
\end{equation*}
Since $\psi$ is arbitrary, we conclude that $L^q(\X)\ni\div(\delta)=\div(X_{\delta})$, which means $X_{\delta}\in\frd^q(\X)$.
\end{proof}

\begin{corollary}
Let $q \in (1, \infty]$. Any $\delta \in \der^{q, q}(\X)$ may be uniquely extended to a $q$--Sobolev derivation $\overline{\delta} \in \sder_q(\X)$ satisfying \eqref{eq:sob-der} with 
\begin{equation*}
\ell = \begin{cases} |\delta| & \text{ if } q \in (1, \infty),  \\ 
\|\delta\|_{L^{\infty}(\X)} & \text{ if } q = \infty.
\end{cases}
\end{equation*}
\end{corollary}
\begin{proof}
Lemma \ref{lem:extension_Sob_der} provides the existence of such extension. The uniqueness follows by exploiting Lemma \ref{isomorphism_2}: indeed, we can associate to $\delta$ a unique vector $X_{\delta} \in \frd^q(\X)$ such that $\delta = X_{\delta} \circ \d$. This means that $\delta$ can be naturally extended to $D^{1,p}(\X)$, where $p$ is the conjugate exponent to $q$, by setting $\overline{\delta}(f) := \d f(X_{\delta})$ for all $f \in D^{1,p}(\X)$, which is well posed, since $\d f \in L^{p}(T^*\X)$ and $X_{\delta} \in L^{q}(T\X)$. In order to show that this extension is indeed unique, let $T_k(f)$ be the truncation of $f$ defined as in the proof of Lemma \ref{lem:extension_Sob_der} and let $\tilde{\delta}$ be the unique extension of $\delta$ to $D^{1,p}(\X)\cap L^{\infty}(\X)$. Since $\tilde{\delta}$ is unique, it must coincide with $X_{\delta} \circ \d$ on $D^{1,p}(\X) \cap L^{\infty}(\X)$. Thanks to point 3 of Proposition \ref{prop:diff_Leibniz}, we have
\begin{equation*}
\d T_k(f) = \uno_{\{ |f| < k \}} \d f,
\end{equation*}
so that
\begin{equation*}
\tilde{\delta}(T_k(f)) = \d T_k(f)(X_\delta) = \uno_{\{ |f| < k \}} \d f(X_\delta).
\end{equation*}
Thanks to the estimates on $|X_\delta|$ given in Lemma \ref{isomorphism_2}, we clearly have $\d f(X_\delta) \in L^1(\X)$. Therefore, Lebesgue's dominated convergence theorem implies that 
\begin{equation*}
\tilde{\delta}(T_k(f)) \to \d f(X_\delta) \text{ in } L^1(\X) \text{ as } k \to + \infty,
\end{equation*}
which means that $\overline{\delta} \coloneqq X_{\delta}\circ\d$ is the unique extension of $\delta$ to $D^{1,p}(\X)$. Finally, the estimates on $\ell$ follow immediately from those on $|X_\delta|$ given in Lemma \ref{isomorphism_2}.
\end{proof}

All in all, we have just proved the following result.

\begin{theorem}\label{der-iso-d}
Let $(\X,d,\mu)$ be a complete and separable metric measure space endowed with a non--negative Radon measure $\mu$ which is finite on bounded sets and let $q \in (1, \infty]$. Then, any vector field $X\in\frd^{q}(\X)$ induces a unique derivation $\delta_X\in\der^{q,q}(\X)$ satisfying $\delta_X = X \circ \d$, $|\delta_{X}| \le |X|$ and $\div(\delta_{X}) = \div(X)$. Vice versa, any derivation $\delta\in\der^{q,q}(\X)$ induces a unique vector field $X_\delta\in\frd^q(\X)$ satisfying $\delta = X_{\delta} \circ d$,
\begin{equation*}
|X_{\delta}| \le \begin{cases} |\delta| & \text{ if } q \in (1, \infty),  \\ 
\|\delta\|_{L^{\infty}(\X)} & \text{ if } q = \infty.
\end{cases}
\end{equation*}
and $\div(X_{\delta}) = \div(\delta)$.
\end{theorem}

In conclusion, Lemma \ref{isomorphism_1}, Lemma \ref{isomorphism_2} and Theorem \ref{der-iso-d} entitle us to characterize the space of $BV$ functions in the following alternative way.

\begin{definition}\label{def-bv-rcd}
 Let $u\in L^1(\X)$. We say that $u\in BV_{\frd}(\X)$ if there exists a continuous linear map $L_u:\frd^\infty(\X)\to\mathbf{M}(\X)$ such that
 \begin{equation}\label{eq-bv-vect}
  \int_\X\d L_u(X)=-\int_\X u\, \div(X) \, \d\mu
 \end{equation}
for all vector fields $X\in\frd^\infty(\X)$. 
\end{definition}

In the same fashion as for Definition \ref{bv-der}, we observe that also the above characterization is well--posed, in the sense that, when $u\in BV_{\frd}(\X)$, the map $L_u$ is uniquely determined. Therefore, we shall write, with no ambiguity, $Du(X)$ instead of $L_u(X)$.

It is worth noticing that, if $q=\infty$ -- which is the case of interest for $BV$ functions --, Theorem \ref{der-iso-d} obviously gives a bijection between $\der_b(\X)$ and $\frd^{\infty}(\X)$. This fact has an immediate yet notable consequence.

\begin{theorem}\label{equal-BV} We have the equivalence
\begin{equation}\label{eq-equal-BV}
BV(\X) = BV_{\frd}(\X).
\end{equation}
In particular, if $\delta \in \der_b(\X)$ and $X \in \frd^{\infty}(\X)$ are associated by the bijection given by Theorem \ref{der-iso-d}, then the signed measures $Du(\delta)$ and $Du(X)$ coincide, and the same is true for their respective total variations.
\end{theorem}
\begin{proof}
Given $u\in L^1(\X)$, for all $X\in\frd^\infty(\X)$ we have
\begin{equation*}\int_\X u\div(X)\d\mu = \int_\X u\div(\delta_X)\d\mu,
\end{equation*}
where $\delta_X\in\der_b(\X)$ is the unique derivation associated to $X$ given by Lemma \ref{isomorphism_1}. On the other hand, for all $\delta \in \der_b(\X)$ we have
\begin{equation*}
\int_\X u\, \div(\delta) \, \d\mu = \int_{\X} u\, \div(X_{\delta}) \, \d\mu, 
\end{equation*} 
where $X_{\delta} \in \frd^{\infty}(\X)$ is the unique vector field associated to $\delta$ given by Lemma \ref{isomorphism_2}. Therefore, a straightforward application of Theorem \ref{der-iso-d}  -- with $\delta$ and $X$ associated by the bijection -- allows us to conclude. \end{proof}

The above result motivates the following definition.

\begin{definition}\label{BV-grad}Let $\vinf(\X)$ denote any of the spaces $\der_b(\X)$ or $\frd^\infty(\X)$. If $u\in BV(\X)$, we say that the measure--valued map $Du:\vinf(\X)\to\mathbf{M}(\X)$ satisfying \eqref{eq-bv-der} or \eqref{eq-bv-vect} respectively, is the {\it gradient} of $u$.\end{definition}

We observe that this notation is chosen in order to emphasize the fact that the measure $Du(X)$ -- or, $Du(\delta)$ -- can be seen as a pairing between the vector field $X$ -- or, the derivation $\delta$ -- coherently with the definition of pairing measure given in Section \ref{leib-rules}.

\begin{remark}\label{remark-BV}
We explicitly point out that by virtue of Theorem \ref{equal-BV} and \cite[Theorem 7.3.7]{di}, $BV_\frd(\X)$ is also equivalent to the ``relaxed'' $BV$ space treated in \cite{miranda}.

Moreover, Theorem \ref{char-bv} continues to hold also for the definition via vector fields, so that for any open set $\Omega\subset\X$ we have an analogous representation formula for the total variation, given by
\begin{equation*}\label{T-Var}
\|Du\|(\Omega)=\sup\left\{\int_\X u\div(X)\d\mu;\;X\in\frd^\infty(\X),\;\supp(X)\Subset\Omega,\;|X|\le1\right\},
\end{equation*}
which is obviously lower semicontinuous with respect to the $L^1$--convergence.
Clearly, the observations in Remark \ref{rem:prop_set_fin_per} apply also to Definition \ref{def-bv-rcd} as well.
\end{remark}

In the following sections we shall always consider $BV$ functions in the sense of Definition \ref{def-bv-rcd}, in order to be coherent with our choice of using the differential machinery of \cite{gi2}.

\section{Divergence-measure fields and Gauss--Green formulas on regular domains}\label{reg-dom}

In this section we introduce the main object of our discussion, namely the essentially bounded divergence--measure vector fields $\DM^\infty(\X)$, in terms of which we shall undertake the task of obtaining integration by parts formulas featuring their (interior and exterior) normal traces.
We start with the main definitions and properties in a general metric measure space $(\X,d,\mu)$ satisfying no specific structural assumptions. Then, we consider a particular class of integration domains on which we achieve integration by part formulas; to this purpose, however, as the Riesz Representation Theorem will be often needed to achieve our claims, we shall require in addition the space $\X$ to be locally compact. 

As we already mentioned in the Introduction, this part is based on \cite{bu}, where the issue was attacked in a geodesic space, and which in turn started from the analysis made by the authors of \cite{mms1} in the context of a doubling metric measure space supporting a Poincar\'e inequality and equipped with the Cheeger differential structure (introduced in \cite{ch}).

\bigskip

Let $\left({\X},d,\mu\right)$ be a complete and separable metric
measure space equipped with a non--negative Radon measure $\mu$, finite on bounded sets. Recall that the differential operator satisfies the Leibniz rule, namely
\begin{equation}
\text{d}(fg)=f\text{d}g+g\text{d}f.\label{eq:Leibniz-d}
\end{equation}
for any $f,g\in\Lip(\X)$.

\begin{definition} We say that $X\in L^{p}(T\X)$, $1\le p\le\infty$, is a
$p$-\textit{summable divergence--measure field},
and we write $X\in{\cal DM}^{p}(\X)$, if its distributional
divergence, which we continue to denote as $\div(X)$, is a finite Radon measure; that is, $\div(X)$
is a measure satisfying 
\begin{equation}
-\int_\X g\mathrm{d}\div(X)=\int_\X\mathrm{d}g(X)\mathrm{d}\text{\ensuremath{\mu}}\label{eq:div-meas}
\end{equation}
for every $g\in \Lip_{\bs}(\X)$. If $p = \infty$, we say that $X \in \DM^{\infty}(\X)$ is an {\em essentially bounded divergence--measure field}.
\end{definition}

In the following, we shall be concerned mostly with the case $p = \infty$.

We prove now a simple version of a Leibniz rule for fields $X \in \DM^{p}(\X)$ and bounded Lipschitz scalar functions $f$ with $\d f \in L^{q}(T^{*} \X)$, where $p, q \in [1, \infty]$ are conjugate exponents. This condition is equivalent to ask that the minimal weak upper gradient $|Df|$ belongs to $L^{q}(\X)$. To this purpose, we notice that for any $g\in\Lip_{\bs}({\X})$ and $X\in L^{p}(T \X)$,
the $L^{\infty}$--linearity of the differential ensures that
\begin{equation}
\text{d}g(fX)=f\text{d}g(X)\label{eq:d-Linearity}
\end{equation}
for every $f\in\Lip_{b}({\X})$.

\begin{lemma} \label{dmlipleibniz} 
Let $p, q \in [1, \infty]$ be such that $\frac{1}{p} + \frac{1}{q} = 1$. Let $X\in{\cal DM}^{p}({\X})$ and $f\in{\rm Lip}_b({\X})$ with $\d f \in L^{q}(T^{*} \X)$.
Then, $fX\in{\cal DM}^{p}({\X})$ and 
\begin{equation}
\div(fX)=f\div(X)+ \d f(X) \mu.\label{eq:Leibniz-Lip}
\end{equation}
\end{lemma} 

\begin{proof} 
It is clear that $f X \in L^{p}(T \X)$. Let $g\in\text{Lip}_{\bs}({\X})$. By applying (\ref{eq:Leibniz-d})--(\ref{eq:d-Linearity})
we get
\begin{align*}
\int_{{\X}}\text{d}g(fX) \, \text{d}\mu & =\int_{{\X}}f\text{d}g(X) \, \text{d}\mu=\int_{{\X}}\text{d}(fg)(X) \, \text{d}\mu-\int_{{\X}}g\text{d}f(X) \, \text{d}\mu\\
 & =-\int_{{\X}}fg \, \text{d}\div(X)-\int_{{\X}}g\text{d}f(X) \, \d \mu
\end{align*}
Hence, $\div(f X) \in \M(\X)$ and \eqref{eq:Leibniz-Lip} holds. 
\end{proof}

We show now a result on the absolute continuity of the measure $\|\div(X)\|$ with respect to the $q$-capacity, which extends analogous properties of the divergence--measure fields in the Euclidean framework (for which we refer to \cite{comi2017locally, PT}).
To this aim, let us first recall the notion of $q$-capacity of any set $A \subset \X$.

\medskip

\begin{definition}
Let $\Omega \subset \X$ be an open set and $A \subset \Omega$. We define the {\em relative $q$--capacity}, $1\le q < \infty$, of $A$ in $\Omega$ as the (possibly infinite) quantity
\begin{equation} \label{eq:cap_def}
\mathrm{Cap}_q(A, \Omega) \coloneqq \inf \left\{\|u\|_{W^{1,q}(\X)}^q;\; u\in \mathcal{W}^{1,q}(\X) \ \text{such that} \ 0\le u\le \chi_{\Omega} \ \text{and} \ u|_A=1 \right\}.
\end{equation}
If $\Omega = \X$, we set $\mathrm{Cap}_q(A) \coloneqq \mathrm{Cap}_q(A, \X)$.
\end{definition}

\begin{remark} \label{rem:density_Lip_Cap}
We notice that, thanks to the density of Lipschitz functions inside Sobolev spaces \cite[Theorem 7.3]{ags4}, the infimum in \eqref{eq:cap_def} may be taken over Lipschitz functions in $W^{1, q}(\X)$. 
\end{remark}

We prove now a technical lemma on compact sets with zero $q$--capacity (an analogous result in the Euclidean setting has been proved in \cite[Lemma 2.8]{comi2017locally}).

\begin{lemma} \label{ref:Cap_Lip_sequence}
Let $q \in [1, \infty)$ and $K \subset \X$ be a compact set such that $\mathrm{Cap}_{q}(K) = 0$. Then there exists a sequence of functions $\varphi_k \in \Lip_\bs(\X)$ such that 
\begin{enumerate}
\item $0\le\varphi_k \le1$, 
\item $\varphi_{k} \equiv 1$ on $K$ for any $k$, 
\item $\left \Vert\varphi_k\right\Vert_{W^{1,q}(\X)}\to 0$ as $k \to + \infty$,
\item $\varphi_k\to 0$ pointwise in $\X\setminus K$. 
\end{enumerate}
\end{lemma}
\begin{proof}
We start by noticing that $\mathrm{Cap}_{q}(K, \Omega) = 0$ for any open set $\Omega \supset K$. Indeed, if $K \subset \Omega$ and we choose $\eta \in \Lip_{\bs}(\Omega)$ satisfying $\eta \equiv 1$  on $K$, then, for any $u \in \mathcal{W}^{1, q}(\X)$ such that $0 \le u \le 1$ and $u|_{K} = 1$, we have 
\begin{equation*}
\mathrm{Cap}_{q}(K, \Omega) \le \int_{\X} \left(|\eta u|^{q} + |D (\eta u)|^{q}\right)\d \mu \le C_{\eta, q} \int_{\X} \left(|u|^{q} + |D u|^{q}\right)\d \mu,
\end{equation*} 
for some $C_{\eta, q} > 0$. Hence, since $u$ was arbitrary, we conclude that $\mathrm{Cap}_{q}(K, \Omega) = 0$. Actually, this argument, combined with Remark \ref{rem:density_Lip_Cap}, shows also that the infimum may be taken on $\Lip_{\bs}(\Omega)$.
Let now $\Omega_{k}$ be a sequence of bounded open sets satisfying $\Omega_{k} \supset K$ and $\bigcap_{k = 1}^{\infty} \Omega_{k} = K$. Let also $\eps_{k}$ a non--negative vanishing sequence. 
Then, by the above argument, for any $k$ there exists $\varphi_{k} \in \Lip_\bs(\Omega_{k}) \cap \mathcal{W}^{1, q}(\X)$ such that $0 \le \varphi_{k} \le 1$, $\varphi_{k} \equiv 1$ on $K$ and $\| \varphi_{k} \|_{W^{1, q}(\X)} < \eps_{k}$. This ends the proof.\\
\end{proof}

 The following result extends \cite[Proposition 7.1.8]{bu} to any pair of conjugate exponents $p, q$ and uses the same arguments based on the Hahn decomposition as in \cite[Lemma 3.6]{mms1}, and in \cite[Theorem 2.15]{comi2017locally} for the Euclidean case.

\begin{proposition}\label{div-cap}
Let $1 < p \le\infty$ and $1\le q < \infty$ be such that $\frac{1}{p}+\frac{1}{q}=1$, and let $X\in\DM^{p}(\X)$. Then, for any Borel set $B$ such that $\mathrm{Cap}_q(B)=0$, we have $\|\div(X)\|(B)=0$.
\end{proposition}
\begin{proof}
Let $B$ be a Borel set such that $\mathrm{Cap}_{q}(B) = 0$.
Since $\div(X)$ is a signed Radon measure, we are entitled to consider its positive and negative parts, namely $\div(X)^+$ and $\div(X)^-$ respectively. By the Hahn decomposition theorem, there exist disjoint Borel sets $B_{\pm} \subset B$ with $B_{+} \cup B_{-} = B$ such that $\pm \div(X) \res B_{\pm} \ge 0$.
Hence, we need to prove that $\div(X)(B_{\pm}) = 0$, and, in order to do so, it suffices to show that $\div(X)(K) = 0$ for any compact subset $K$ of $B_{\pm}$, thanks to the interior regularity of $\|\div(X)\|$.

Without loss of generality, let $K \subset B_{+}$, since the calculations in the case $K \subset B_{-}$ are the same. Since $\mathrm{Cap}_{q}(K) = 0$, by Lemma \ref{ref:Cap_Lip_sequence} there exists a sequence $\left(\varphi_k\right)_{k\in\mathbb{N}}\subset\text{Lip}_\bs(\X)$ such that $0\le\varphi_k \le1$, $\varphi_{k} = 1$ on $K$, $\left\Vert\varphi_k\right\Vert_{W^{1,q}(\X)}\to 0$ as $k \to + \infty$, and $\varphi_k\to 0$ pointwise in $\X\setminus K$. Then, it follows that
\begin{equation} \label{eq:cap_abs_cont_proof}
\left\vert\int_\X \varphi_k\text{d}\div(X)\right\vert=\left\vert\int_\X \text{d}\varphi_k(X)\text{d}\mu\right\vert\le\left\Vert\left\vert D\varphi_k\right\vert\right\Vert_{L^q(\X)}\left\Vert |X| \right\Vert_{L^p(\X)}\le\left\Vert\varphi_k\right\Vert_{W^{1,q}(\X)}\left\Vert |X| \right\Vert_{L^p(\X)} \to 0 
\end{equation}
as $k \to + \infty$.
On the other hand, since $0 \le \varphi_k\le1$ for any $k\in\mathbb{N}$ and $\|\div(X)\|(\X)<\infty$, thanks to the Lebesgue's Dominated Convergence Theorem, we get
\begin{equation*}
\div(X)(K) = \lim_{k \to + \infty} \int_\X \varphi_k\text{d}\div(X),
\end{equation*}  
which, combined with \eqref{eq:cap_abs_cont_proof}, implies $\div(X)(K) = 0$.\\
\end{proof}

We give now a definition of regularity on domains that will allow us to establish some integration by parts formula. As a starting point we recall the notion of upper inner Minkowski content of a set: if we set
\[
\Omega_{t}\coloneqq\left\{ x\in\Omega;\:\text{dist}\left(x,\Omega^{c}\right)\ge t\right\} ,
\]
for $t>0$, then we define
\[
\mathfrak{M}_{i}^*(\partial\Omega)\coloneqq\limsup_{t\to0}\frac{\mu\left(\Omega\setminus\Omega_{t}\right)}{t}.
\]
\\
\begin{definition}\label{regdom} An open set $\Omega\subset{\X}$
is said to be a \textit{regular domain} if it has
finite perimeter in ${\X}$, namely $\mathds1_{\Omega}\in BV({\X})$,
and if
\[
\left\Vert D\mathds1_{\Omega}\right\Vert ({\X})=\mathfrak{M}_{i}^{*}(\partial \Omega).
\]
\end{definition}

In the following, we shall make use of the Coarea Formula for Lipschitz functions. We recall that this result was proved for $BV$ functions on doubling metric measure spaces supporting a Poincar\'e inequality in \cite[Proposition 4.2]{miranda}; however, the same property holds as well for Lipschitz functions also in our more general setting by the remarks at the beginning of \cite[Section 4]{adg}.

\begin{proposition}
 Let $f\in \Lip_{\mathrm{loc}}(\Omega)$, $\Omega\subset\X$ open set and denote $$E_t\coloneqq\{x\in\Omega;\;f(x)>t\},$$$t\in\R$. Then, for any open set $A\subset\Omega$ one has
 \begin{equation}\label{coarea}
  \int_{-\infty}^{+\infty}\|D\uno_{E_t}\|(A)\d t=\|Df\|(A).
 \end{equation}
 \end{proposition}
 
\begin{remark}[Examples of regular domains.]
As anticipated in the Introduction, the class of regular domains was introduced in \cite{bu} as a generalization of the so--called {\it regular balls} previously considered in \cite{mms1}. In the latter work the authors considered a doubling, geodesic metric measure space supporting a Poincar\'e inequality: in such setting it is actually possible to see that almost every ball is a regular domain, in the sense that for any $x_0\in\X$ and for $\mathscr{L}^1$--almost every $\rho>0$, $\Omega=B_\rho(x_0)$ is a regular domain. Let us briefly discuss this claim\footnote{Actually, it is not strictly necessary to ask that $(\X,d,\mu)$ is also geodesic. Indeed, by the doubling condition and the Poincar\'e inequality, as shown in \cite[Chapter 8]{hkst}, the metric space turns out to be quasiconvex, and this fact allows for $(\X,d,\mu)$ to be made geodesic by a suitable bi--Lipschitz transformation of the metric, meaning that we can choose a geodesic metric which is bi--Lipschitz equivalent to the original one (and which preserves the doubling property of the measure $\mu$). So, in choosing $d_{x_0}(x)$ we may definitely work with such a geodesic bi--Lipschitz metric change in order to exploit \eqref{dgrad_one}.}.\\
We recall that, under the geodesicity assumption, by the results of \cite{ch} we have that the distance function $$d_{x_0}(x)\coloneqq d(x,x_0)\quad\forall x_0\in\X$$ (which is obviously a Lipschitz function) is such that its minimal weak upper gradient satisfies \begin{equation}\label{dgrad_one}|Dd_{x_0}|(x)=1\end{equation} for $\mu$--almost every $x\in\X$.
Thus, to see that $\Omega=B_\rho(x)$ is a regular domain for almost every $\rho>0$, we can invoke the Coarea Formula \eqref{coarea} together with the condition \eqref{dgrad_one} to find
\[
\frac{\mu\left(B_{\rho}\left(x_0\right)\setminus B_{\rho-h}\left(x_0\right)\right)}{h}=
\frac{1}{h}
\int_{B_{\rho}\left(x_0\right)\setminus B_{\rho-h}\left(x_0\right)}|D d_{x_0}(x)|\text{d}\mu(x)=\frac{1}{h}\int_{\rho-h}^\rho\|D\uno_{B_t\left(x_0\right)}\|(\X)\text{d}t.
\]
So, the claim follows by considering the Lebesgue points of the map $t\mapsto\|D\uno_{B_{t}\left(x_0\right)}\|(\X)$.
\end{remark}

\begin{definition}\label{def-seq}
Given a regular domain $\Omega$, for $\eps>0$ we set
 \begin{equation}
\varphi_{\varepsilon}^{\Omega}(x)\coloneqq\begin{cases}
0 & x\in\Omega^{c},\\
\dfrac{\text{dist}\left(x,\Omega^{c}\right)}{\varepsilon} & x\in\Omega\backslash\Omega_{\varepsilon},\\
1 & x\in\Omega_{\varepsilon}.
\end{cases}\label{eq:def-sequence}
\end{equation}
We call $(\varphi_{\eps}^{\Omega})_{\eps>0}$ the \textit{{defining sequence}} of the regular domain $\Omega$. In the following, we shall simply write $\varphi_{\eps}$, when the relation to $\Omega$ is unambiguous.
\end{definition}

Of course, $(\varphi_{\eps})_{\eps>0} \subset \Lip_{b}(\X)$ by construction.

As we shall prove in the following statement, thanks to the properties of regular domains, the measures $\left|D\varphi_{\varepsilon}\right|\mu$ converge weakly to $\left\Vert D\mathds1_{\Omega}\right\Vert $. To this aim - as we already pointed out - we shall need to assume the space $\X$ to be locally compact, in order to exploit useful consequences of the Riesz Representation Theorem.

\begin{proposition}\label{w-conv-per}
Let $\X$ be locally compact, $\Omega\subset{\X}$ be a regular domain and $(\varphi_{\eps})_{\eps>0}$ be its defining sequence. Set 
\begin{equation*}
\mu_{\eps}(E) := \frac{\mu((\Omega \setminus \Omega_{\eps}) \cap E)}{\eps}
\end{equation*} 
for any Borel set $E$.
Then, we have 
\begin{equation} \label{eq:weak_conv_mink}
\mu_{\eps} \weakto \|D \uno_{\Omega}\|
\end{equation}
and
\begin{equation} \label{eq:weak_conv_nabla_phi}
\left|D \varphi_{\varepsilon}\right|\mu \weakto \|D \uno_{\Omega}\|.
\end{equation} 
\end{proposition}
\begin{proof}
Thanks to \cite[Proposition 1.80]{afp}, it is sufficient to prove that, for any open set $A \subset \X$, we have
\begin{equation} \label{eq:weak_conv_mink_proof}
\lim_{\eps \to 0} \mu_{\eps}(\X)=\left\Vert D\mathds1_{\Omega}\right\Vert (\X) \ \text{and} \ \liminf_{\eps \to 0} \mu_{\eps}(A) \ge \|D \uno_{\Omega}\|(A),
\end{equation}
and
\begin{equation} \label{eq:weak_conv_nabla_phi_proof}
\lim_{\eps \to 0} \int_{\X} \left|D \varphi_{\varepsilon}\right| \, \d \mu = \|D \uno_{\Omega}\|(\X) \ \text{and} \ \liminf_{\eps \to 0} \int_{A} \left|D \varphi_{\varepsilon}\right| \, \d \mu \ge \|D \uno_{\Omega}\|(A).
\end{equation}
Since $\Omega$ is a regular domain, we have
\begin{equation*}
\limsup_{\eps \to 0} \mu_{\eps}(\X)=\left\Vert D\mathds1_{\Omega}\right\Vert (\X).
\end{equation*}
Then, by exploiting the Coarea Formula \eqref{coarea}, the fact that the function $g(x)={\rm dist}\left(x,\Omega^{c}\right)$
is Lipschitz with $|D g(x)| \le 1$, and $\Omega_{\eps}=\{g>\eps\}$,
we obtain
\begin{align*}
\frac{\mu\left(\left(\Omega\setminus\Omega_{\eps}\right)\cap A\right)}{\eps} & =\frac{1}{\eps}\int_{\left(\Omega\backslash\Omega_{\eps}\right)\cap A}\text{d}\mu \ge \frac{1}{\eps}\int_{\left(\Omega\backslash\Omega_{\eps}\right)\cap A}|D g|\text{d}\mu \\
 & =\frac{1}{\eps}\int_{\mathbb{R}}\left\Vert D\mathds1_{\left\{ g>s\right\} }\right\Vert \left(\left(\Omega\backslash\Omega_{\eps}\right)\cap A\right)\text{d}s\\
 & =\frac{1}{\eps}\int_{0}^{\eps}\left\Vert D\mathds1_{\Omega_{s}}\right\Vert (A)\text{d}s  =\int_{0}^{1}\left\Vert D\mathds1_{\Omega_{\eps s}}\right\Vert (A)\text{d}s.
\end{align*}
Hence, by Fatou's lemma and the lower semicontinuity of the perimeter
measure (see Remark \ref{remark-BV}), using the fact that $\Omega_{\eps s}$ converges to $\Omega$
in measure for any $s \in [0, 1]$, we get
\begin{align*}
\liminf_{\eps \to0}\frac{\mu\left(\left(\Omega\backslash\Omega_{\eps}\right)\cap A\right)}{\eps} & \ge \liminf_{\eps \to 0} \int_{0}^{1} \|D \uno_{\Omega_{\eps s}}\|(A) \, \d s \\
& \ge \int_{0}^{1} \liminf_{\eps \to 0} \|D \uno_{\Omega_{\eps s}}\|(A) \, \d s \geq\left\Vert D\mathds1_{\Omega}\right\Vert (A),
\end{align*}
for any open set $A \subset \X$. Then, if $A = \X$, we get
\begin{equation*}
\liminf_{\eps \to 0} \mu_{\eps}(\X) \ge \left\Vert D\mathds1_{\Omega}\right\Vert (\X),
\end{equation*}
and this proves \eqref{eq:weak_conv_mink_proof}, which implies \eqref{eq:weak_conv_mink}.
We notice now that
\[
\int_{\X}\left|D\varphi_{\varepsilon}\right| \text{d}\mu = \int_{\Omega \setminus \Omega_{\eps}} \frac{|D g|}{\eps} \, \d \mu  \le \frac{\mu\left(\Omega\backslash\Omega_{\varepsilon}\right)}{\varepsilon}.
\]
Hence, we immediately get
\begin{equation} \label{eq:limsup_X_nabla_phi_eps}
\limsup_{\eps \to 0} \int_{\X}\left|D \varphi_{\varepsilon}\right| \text{d}\mu \le \lim_{\eps \to 0} \frac{\mu\left(\Omega\backslash\Omega_{\varepsilon}\right)}{\varepsilon} = \|D \uno_{\Omega}\|(\X).
\end{equation}
Then, arguing with the Coarea Formula as above, for any open set $A$ we obtain
\begin{equation*}
\int_{A}\left|D \varphi_{\varepsilon}\right| \text{d}\mu = \int_{(\Omega \setminus \Omega_{\eps}) \cap A} \frac{|D g|}{\eps} \, \d \mu = \int_{0}^{1} \|D \uno_{\Omega_{\eps s}}\|(A) \, \d s.
\end{equation*}
Thus, Fatou's lemma, the fact that $\Omega_{\eps s}$ converges to $\Omega$ in measure and the lower semicontinuity of the perimeter measure imply that
\begin{equation*}
\liminf_{\eps \to 0} \int_{A}\left|D \varphi_{\varepsilon}\right| \text{d}\mu \ge \|D \uno_{\Omega}\|(A).
\end{equation*}
This proves the second part of \eqref{eq:weak_conv_nabla_phi_proof}, while we obtain the first by taking $A = \X$ and employing \eqref{eq:limsup_X_nabla_phi_eps}. Thus, \eqref{eq:weak_conv_nabla_phi} follows and this ends the proof.\\
\end{proof}

\begin{remark}
An interesting byproduct of the proof of Proposition \ref{w-conv-per} is that, if $\Omega$ is a regular domain, the measures $\displaystyle \nu_{\eps} := \int_{0}^{1} \|D \uno_{\Omega_{\eps s}}\| \, \d s$ weakly converge to $\|D \uno_{\Omega}\|$. We leave the details to the interested reader.
\end{remark}

We are now able to prove an integration by parts formula for regular domains and vector fields in $\DM ^{\infty} (\X )$, following in part the footsteps of the proof already given in \cite[Theorem 5.7]{mms1}; this result can be actually seen as an improvement of \cite[Theorem 7.1.7]{bu}.

\begin{theorem}
\label{gauss-geod}
Let $\X$ be locally compact, $X\in\DM^{\infty}({\X})$
and $\Omega\subset{\X}$ be a regular domain. Then there
exists a function $(X\cdot\nu_\Omega)_{\partial\Omega}^{-}\in L^{\infty}\left(\partial\Omega,\left\Vert D\mathds1_{\Omega}\right\Vert \right)$
such that 
\begin{equation}
\int_{\Omega}f \, \mathrm{d}\div(X)+\int_{\Omega}\mathrm{d}f(X) \, \mathrm{d}\mu = - \int_{\partial \Omega}f \, (X \cdot {\nu}_{\Omega})_{\partial\Omega}^{-} \, \mathrm{d}\left\Vert D\mathds1_{\Omega}\right\Vert,\label{eq:-36}
\end{equation}
for every $f\in\mathrm{Lip}_{b}({\X})$ such that $\supp(f \uno_{\Omega})$ is a bounded set. In addition, we have the following estimate:
\begin{equation} \label{eq:norm_trace_int_bound_reg}
\|(X\cdot\nu_\Omega)_{\partial\Omega}^{-}\|_{L^{\infty}(\partial \Omega, \|D \uno_{\Omega}\|)} \le \| |X| \|_{L^{\infty}(\Omega)}.
\end{equation}
\end{theorem}

As it is done in the literature on weak integration by parts formulas, for instance in the Euclidean space, we shall call the function $(X\cdot\nu_\Omega)_{\partial\Omega}^{-}$ the \textit{{interior normal trace}} of $X$ on $\partial\Omega$.

\begin{proof}
Let $\left(\varphi_{\varepsilon}\right)_{\varepsilon>0}$ be the defining sequence 
of $\Omega$ as in \eqref{eq:def-sequence}. It is clear that $f \varphi_{\eps} \in \Lip_{\bs}(\X)$, so that, by \eqref{eq:Leibniz-Lip}, we have
\begin{align*}
\int_{{\X}}f \text{d}\varphi_{\varepsilon}(X) \, \text{d}\mu & =\int_{{\X}}\text{d}\left(f\varphi_{\varepsilon}\right)(X) \, \text{d}\mu - \int_{{\X}}\varphi_{\varepsilon}\text{d}f(X) \, \text{d}\mu \\
 & =-\int_{{\X}}\varphi_{\varepsilon} f \text{d}\div(X) -\int_{{\X}}\varphi_{\varepsilon} \text{d}f(X) \text{d}\mu.
\end{align*}
Since $\varphi_{\varepsilon}$ converges to $\mathds1_{\Omega}$
everywhere and $\supp{(f \varphi_{\eps})} \subset \supp{(f \uno_{\Omega})}$ is a bounded set, the Lebesgue's Dominated Convergence Theorem entails the existence of
\begin{equation} \label{eq:approx_def_sequence_IBP}
\lim_{\varepsilon\to0}\int_{{\X}}f \text{d}\varphi_{\varepsilon}(X) \, \text{d}\mu=-\int_{\Omega}f \, \text{d}\div(X) -\int_{\Omega}\text{d}f(X) \, \text{d}\mu.
\end{equation}
Let us now define the map
\[
T_{X}(f):=\lim_{\varepsilon\to0}T_{X}^{\varepsilon}(f),
\]
for any $f\in{\rm Lip}_{b}({\X})$, where
\[
T_{X}^{\varepsilon}(f):=\int_{{\X}}f \text{d}\varphi_{\varepsilon}(X) \, \text{d}\mu.
\]
It is clear that we have the estimate
\[
\left|T_{X}^{\varepsilon}(f)\right|\leq \||X|\|_{L^{\infty}(\Omega)}\int_{{\X}} |f| \left|D \varphi_{\varepsilon}\right|\, \text{d}\mu.
\]
Hence, by \eqref{eq:weak_conv_nabla_phi}, we deduce
\begin{equation} \label{eq:trace_functional_bound}
\left|T_{X}(f)\right|\leq \||X|\|_{L^{\infty}(\Omega)} \int_{\partial \Omega} |f| \, \d \left\Vert D\mathds1_{\Omega}\right\Vert \le \||X|\|_{L^{\infty}(\Omega)} \|f\|_{L^{1}(\X, \|D \uno_{\Omega}\|)}.
\end{equation}
We notice that, since $\Lip_{\bs}(\X)$ is dense in $C_{\bs}(\X)$ and $C_{\bs}(\X)$ is dense in $L^{1}(\X, \|D \uno_{\Omega}\|)$, $T_{X}$ may be extended to a continuous linear functional in the dual of $L^{1}(\X, \|D \uno_{\Omega}\|)$.
Therefore, by the Riesz Representation Theorem and the fact that $\|D \uno_{\Omega}\|(\X \setminus \partial \Omega) = 0$ (see Remark \ref{rem:prop_set_fin_per}), there exists a function $(X\cdot\nu)_{\partial\Omega}^{-}\in L^{\infty}(\partial \Omega,\left\Vert D\mathds1_{\Omega}\right\Vert )$
such that 
\begin{equation} \label{eq:trace_def_reg}
T_{X}(f) = \int_{\partial \Omega} f (X \cdot \nu_{\Omega})_{\partial \Omega}^{-} \, \d \|D \uno_{\Omega}\|,
\end{equation}
for any $f \in \Lip_{\bs}(\X)$. 
Since
\begin{equation*}
T_{X}(f) = -\int_{\Omega}f \, \text{d}\div(X) -\int_{\Omega}\text{d}f(X) \, \text{d}\mu,
\end{equation*}
by \eqref{eq:approx_def_sequence_IBP}, we immediately obtain \eqref{eq:-36} for any $f \in \Lip_{\bs}(\X)$.
Let us now consider $f \in \Lip_{b}(\X)$ such that $\supp(f \uno_{\Omega})$ is bounded. It is easy to notice that, for any cutoff $\eta \in \Lip_{\bs}(\X)$ such that $\eta \equiv 1$ on $\supp(f \uno_{\Omega})$, we have 
\begin{align*}
\int_{\Omega}f \, \text{d}\div(X) + \int_{\Omega}\text{d}f(X) \, \d \mu &= \int_{\Omega} \eta f \, \text{d}\div(X) + \int_{\Omega}\text{d}(\eta f)(X) = - T_{X}(\eta f) \\
& = - \int_{\partial \Omega} \eta f (X \cdot \nu_{\Omega})_{\partial \Omega}^{-} \, \d \|D \uno_{\Omega}\| \\
& = - \int_{\partial \Omega} f (X \cdot \nu_{\Omega})_{\partial \Omega}^{-} \, \d \|D \uno_{\Omega}\|.
\end{align*}
Hence, we get \eqref{eq:-36} in this general case.
Finally, the estimate \eqref{eq:norm_trace_int_bound_reg} follows easily by \eqref{eq:trace_functional_bound} and \eqref{eq:trace_def_reg}.
\end{proof}

Thanks to this result, we immediately obtain a Gauss--Green formula for regular domains.

\begin{corollary}
Let $\X$ be locally compact, $X\in\DM^{\infty}({\X})$ and $\Omega\subset{\X}$ be a bounded regular domain. Then we have
\begin{equation} \label{eq:GG_int_reg}
\div(X)(\Omega) = - \int_{\partial \Omega} (X \cdot \nu)_{\partial \Omega}^{-} \, \d \|D \uno_{\Omega}\|.
\end{equation}
\end{corollary}
\begin{proof}
Take $f = 1$ in Theorem \ref{gauss-geod}.
\end{proof}

We prove now an integration by parts formula on the closure of a given open set $\Omega$, under the assumption that $\overline{\Omega}^{c}$ is a regular domain.

\begin{theorem} \label{thm:IBP_ext_reg}
Let $\X$ be locally compact, $X\in\DM^{\infty}({\X})$ and $\Omega \subset{\X}$ be an open set such that $\overline{\Omega}^c$ is a regular domain. Then there
exists a function $(X\cdot\nu_\Omega)_{\partial\Omega}^{+}\in L^{\infty}\left(\partial \Omega,\left\Vert D\mathds1_{\overline{\Omega}}\right\Vert \right)$
such that 
\begin{equation}
\int_{\overline{\Omega}}f \, \mathrm{d}\div(X)+\int_{\overline{\Omega}}\mathrm{d}f(X) \, \mathrm{d}\mu = - \int_{\partial \Omega}f \, (X \cdot {\nu}_{\Omega})_{\partial\Omega}^{+} \, \mathrm{d}\left\Vert D\mathds1_{\overline{\Omega}}\right\Vert,\label{eq:IBP_ext_reg}
\end{equation}
for every $f\in\mathrm{Lip}_{b}({\X})$ such that $\supp(f \uno_{\Omega})$ is a bounded set. In addition, we have the following estimate:
\begin{equation} \label{eq:norm_trace_ext_bound_reg}
\|(X\cdot\nu_\Omega)_{\partial\Omega}^{+}\|_{L^{\infty}(\partial \Omega, \|D \uno_{\overline{\Omega}}\|)} \le \| |X| \|_{L^{\infty}(\overline{\Omega}^{c})}.
\end{equation}
If $\Omega$ is also bounded, then we have
\begin{equation} \label{eq:GG_ext_reg}
\div(X)(\overline{\Omega}) = - \int_{\partial \Omega} (X \cdot {\nu}_{\Omega})_{\partial\Omega}^{+} \, \mathrm{d}\left\Vert D\mathds1_{\overline{\Omega}}\right\Vert.
\end{equation}
\end{theorem}

Analogously as before, we shall call the function $(X\cdot\nu_\Omega)_{\partial\Omega}^{+}$ the \textit{{exterior normal trace}} of $X$ on $\partial\Omega$.

\begin{proof}
Let $\eta \in \Lip_{\bs}(\X)$ be a cutoff function such that $\eta \equiv 1$ on $\supp(f \uno_{\Omega})$. Then, $\eta f \in \Lip_{\bs}(\X)$, and so \eqref{eq:div-meas} implies
\begin{equation*}
\int_{\X} \eta f \, \d \div(X) + \int_{\X} \d(\eta f)(X) \, \d \mu = 0.
\end{equation*}
By \eqref{eq:-36}, we have
\begin{align*}
0 & = \int_{\overline{\Omega}} \eta f \, \d \div(X) + \int_{\overline{\Omega}} \d(\eta f)(X) \, \d \mu + \int_{\overline{\Omega}^c} \eta f \, \d \div(X) + \int_{\overline{\Omega}^c} \d(\eta f)(X) \, \d \mu \\
& = \int_{\overline{\Omega}} \eta f \, \d \div(X) + \int_{\overline{\Omega}} \d(\eta f)(X) \, \d \mu - \int_{\partial \Omega} \eta f \, (X \cdot \nu_{\overline{\Omega}^{c}})_{\partial \Omega}^{-} \, \d \|D \uno_{\overline{\Omega}}\| \\
& = \int_{\overline{\Omega}} f \, \d \div(X) + \int_{\overline{\Omega}} \d f(X) \, \d \mu - \int_{\partial \Omega} f \, (X \cdot \nu_{\overline{\Omega}^{c}})_{\partial \Omega}^{-} \, \d \|D \uno_{\overline{\Omega}}\|.
\end{align*}
Thus, if we set $(X\cdot\nu_\Omega)_{\partial\Omega}^{+} := - (X \cdot \nu_{\overline{\Omega}^{c}})_{\partial \Omega}^{-}$, we get \eqref{eq:IBP_ext_reg} and \eqref{eq:norm_trace_ext_bound_reg}, by Theorem \ref{gauss-geod}. Finally, if $\overline{\Omega}$ is bounded, we can take $f = 1$ and obtain \eqref{eq:GG_ext_reg}.
\end{proof}

\begin{remark}
It is not difficult to see that one can argue in an analogous way as in the proof of Theorem \ref{thm:IBP_ext_reg} to achieve an integration by parts formula on any closed set $C$ such that $\mu(C) > 0$ and $C^{c}$ is a regular domain. The condition $\mu(C) > 0$ is necessary so that $\uno_{C}$ is not identically zero, as a $BV$ function. Under this and the previous assumptions, we obtain the existence of $(X\cdot\nu_C)_{\partial C}^{+}\in L^{\infty}\left(\partial C,\left\Vert D\mathds1_{C}\right\Vert \right)$
such that 
\begin{equation*}
\int_{C}f \, \mathrm{d}\div(X)+\int_{C}\mathrm{d}f(X) \, \mathrm{d}\mu = - \int_{\partial C}f \, (X \cdot {\nu}_{C})_{\partial C}^{+} \, \mathrm{d}\left\Vert D\mathds1_{C}\right\Vert,
\end{equation*}
for every $f\in\mathrm{Lip}_{b}({\X})$ such that $\supp(f \uno_{C})$ is a bounded set. In addition, we have the following estimate:
\begin{equation*} 
\|(X\cdot\nu_C)_{\partial C}^{+}\|_{L^{\infty}(\partial C, \|D \uno_{C}\|)} \le \| |X| \|_{L^{\infty}(C^{c})}.
\end{equation*}
If then $C$ is also bounded, we obtain
\begin{equation*}
\div(X)(C) = - \int_{\partial C} (X \cdot {\nu}_{C})_{\partial C}^{+} \, \d \|D\mathds1_{C}\|.
\end{equation*}
\end{remark}

We consider now a particular subfamily of regular domains, in order to obtain both integration by parts formulas, from the interior and the exterior, with the same perimeter measure.

\begin{corollary} \label{cor:mu_negl_boundary_Omega}
Let $\X$ be locally compact, $X\in\DM^{\infty}({\X})$ and $\Omega \subset{\X}$ be a regular domain such that $\overline{\Omega}^c$ is a regular domain and $\mu(\partial \Omega) = 0$. Then $(X \cdot {\nu}_{\Omega})_{\partial\Omega}^{+} \in L^{\infty}(\partial \Omega, \|D\uno_{\Omega}\|)$, and we have \eqref{eq:-36}, 
\begin{equation}
\int_{\overline{\Omega}}f \, \mathrm{d}\div(X)+\int_{\Omega}\mathrm{d}f(X) \, \mathrm{d}\mu = - \int_{\partial \Omega}f \, (X \cdot {\nu}_{\Omega})_{\partial\Omega}^{+} \, \mathrm{d}\left\Vert D\mathds1_{\Omega}\right\Vert,\label{eq:IBP_ext_reg_mod}
\end{equation}
and
\begin{equation}
\int_{\partial \Omega} f \, \mathrm{d}\div(X) = - \int_{\partial \Omega}f \, \left ( (X \cdot \nu_{\Omega})_{\partial\Omega}^{+} - (X \cdot \nu_{\Omega})_{\partial \Omega}^{-} \right ) \, \mathrm{d}\left\Vert D\mathds1_{\Omega}\right\Vert,\label{eq:IBP_boundary_reg}
\end{equation}
for every $f\in\mathrm{Lip}_{b}({\X})$ such that $\supp(f \uno_{\Omega})$ is a bounded set. If $\Omega$ is also bounded, then we have
\begin{equation} \label{eq:GG_boundary_reg}
\div(X)(\partial \Omega) = - \int_{\partial \Omega} \left ( (X \cdot \nu_{\Omega})_{\partial\Omega}^{+} - (X \cdot \nu_{\Omega})_{\partial \Omega}^{-} \right ) \, \mathrm{d}\left\Vert D\mathds1_{\Omega}\right\Vert.
\end{equation}
\end{corollary}
\begin{proof}
Since $\mu(\partial \Omega) = 0$, then $\|D \uno_{\Omega}\| = \|D \uno_{\overline{\Omega}}\|$. Therefore, it is enough to apply Theorems \ref{gauss-geod} and \ref{thm:IBP_ext_reg}. Then, by setting $f = 1$ in the case $\Omega$ is bounded, we immediately get \eqref{eq:GG_boundary_reg}. 
\end{proof}

\begin{remark}
Under the assumptions of Corollary \ref{cor:mu_negl_boundary_Omega} and arguing as in the proof of Theorem \ref{thm:IBP_ext_reg}, we obtain the following relations:
\begin{equation*} 
(X \cdot \nu_{\overline{\Omega}^{c}})_{\partial \Omega}^{-} = - (X \cdot \nu_{\Omega})_{\partial \Omega}^{+} \ \text{and} \ (X \cdot \nu_{\overline{\Omega}^c})_{\partial \Omega}^{+} = - (X \cdot \nu_{\Omega})_{\partial \Omega}^{-}. 
\end{equation*}
\end{remark}

We conclude this section with a refined result in the special case in which the measure $\div(X)$ is absolutely continuous with respect to $\mu$, which includes the case $X \in \frd^{\infty}(\X)$.

\begin{corollary} \label{cor:dom_reg_trace_abs_cont}
Let $\X$ be locally compact, $X\in\DM^{\infty}({\X})$ such that $\|\div(X)\| \ll \mu$, and let $\Omega \subset{\X}$ be a regular domain such that $\overline{\Omega}^c$ is a regular domain and $\mu(\partial \Omega) = 0$. Then we have
\begin{equation} \label{eq:traces_equal_reg_abs_cont}
(X \cdot \nu_{\Omega})_{\partial \Omega}^{-} = (X \cdot \nu_{\Omega})_{\partial \Omega}^{+}  \  \ \|D \uno_{\Omega}\|\text{-a.e. in} \ \partial \Omega.
\end{equation}
Thus, there exists a unique normal trace, which we denote by $(X \cdot \nu_{\Omega})_{\partial \Omega}$, satisfying
\begin{equation} \label{eq:norm_trace_int_bound_reg_abs_cont}
\|(X\cdot\nu_\Omega)_{\partial\Omega}\|_{L^{\infty}(\partial \Omega, \|D \uno_{\Omega}\|)} \le \| |X| \|_{L^{\infty}(\Omega)}
\end{equation}
and
\begin{equation}
\int_{\Omega}f \, \mathrm{d}\div(X)+\int_{\Omega}\mathrm{d}f(X) \, \mathrm{d}\mu = - \int_{\partial \Omega}f \, (X \cdot {\nu}_{\Omega})_{\partial\Omega} \, \mathrm{d}\left\Vert D\mathds1_{\Omega}\right\Vert \label{eq:IBP_reg_mod_abs_cont}
\end{equation}
for every $f\in\mathrm{Lip}_{b}({\X})$ such that $\supp(f \uno_{\Omega})$ is a bounded set. If $\Omega$ is also bounded, then we have
\begin{equation}  \label{eq:GG_reg_mod_abs_cont}
\div(X)(\Omega) = - \int_{\partial \Omega} (X \cdot {\nu}_{\Omega})_{\partial\Omega} \, \mathrm{d}\left\Vert D\mathds1_{\overline{\Omega}}\right\Vert.
\end{equation}
\end{corollary}
\begin{proof}
By the absolute continuity property of $\div(X)$, we have $\|\div(X)\|(\partial \Omega) = 0$, so that \eqref{eq:IBP_boundary_reg} reduces to 
\begin{equation*}
\int_{\partial \Omega}f \, \left ( (X \cdot \nu_{\Omega})_{\partial\Omega}^{+} - (X \cdot \nu_{\Omega})_{\partial \Omega}^{-} \right ) \, \mathrm{d}\left\Vert D\mathds1_{\Omega}\right\Vert = 0,
\end{equation*}
for any $f \in \Lip_{b}(\X)$ such that $\supp(\uno_{\Omega} f)$ is bounded. 
Hence, we obtain \eqref{eq:traces_equal_reg_abs_cont}, and so we are allowed to set $(X \cdot\nu_{\Omega})_{\partial \Omega} := (X \cdot \nu_{\Omega})_{\partial \Omega}^{-}$. Finally, \eqref{eq:norm_trace_int_bound_reg_abs_cont}, \eqref{eq:IBP_reg_mod_abs_cont} and \eqref{eq:GG_reg_mod_abs_cont} follow immediately from \eqref{eq:norm_trace_int_bound_reg}, \eqref{eq:-36} and \eqref{eq:GG_int_reg}.
\end{proof}

\section{Leibniz rules}\label{leib-rules}

In this section $\left({\X},d,\text{\ensuremath{\mu}}\right)$
will be assumed to be a locally compact $\mathsf{RCD}(K,\infty)$ metric measure space. It is important to notice that such assumption is not too restrictive: indeed, there is a variety of results holding in locally compact (or even compact) $\rcd$ spaces (see for instance \cite[Remark 3.5]{gh}). In addition, there are examples of locally compact $\rcd$ spaces which cannot be $\mathsf{RCD}(K,N)$ for any finite $N$: as an example, we may consider $(\R^n, d, \gamma)$, where $d$ is the usual Euclidean distance, and $\gamma$ is the Gaussian measure $e^{- K |x|^2/2} \Leb{n}$.

We observe that by \eqref{ball-exp}, the reference measure $\mu$ is automatically finite on bounded sets.
Recall also the definition
of heat flow $\text{h}_{t}$ along with its properties given above
in Section \ref{rcdht}.

\smallskip

In order to obtain a generalization of Lemma \ref{dmlipleibniz} to scalar functions of bounded variation, we need two preliminary technical results concerning the regularizing properties of the heat flow.

\medskip

\begin{lemma} \label{lem:heat_flow_lip_functions}
Let $f\in L^\infty(\X)\cap BV(\X)$. Then, we have $\h_{t} f \in \Lip_b(\X)$ and $|D\h_{t} f| \in L^{1}(\X)$ for all $t > 0$. In addition, we get $\left|D\text{h}_{t}f\right| \mu \rightharpoonup\|Df\|$ as $t \to 0^+$ and
\begin{equation}
\int_{{\X}}\varphi\left|D \h_{t}f\right| \, \d \mu\le e^{-Kt}\int_{{\X}} \h_{t}\varphi \, \d \|Df\| \label{eq:BE-BV-dual}
\end{equation}
for all $t > 0$ and non--negative $\varphi \in C_{b}(\X)$. If $\supp(\mu) = \X$, then \eqref{eq:BE-BV-dual} holds also for any non--negative $\varphi \in L^{\infty}(\X)$.
\end{lemma}
\begin{proof}
Since $f \in L^{1}(\X) \cap L^{\infty}(\X)$, it is clear that $f \in L^{2}(\X)$, so that $\h_{t} f$ is well defined. The $L^{\infty}$--estimate on $\text{h}_{t}f$ comes from (\ref{eq:ht-contraction}). As a consequence, \eqref{eq:BE-1} together with \cite[Theorem 6.2]{ags2} implies $\text{h}_{t}f\in\text{Lip}_{b}({\X})$.

By (\ref{eq:BE}) and the semigroup property, for any $s,t>0$
we have $\left|D\text{h}_{t+s}f\right|\le e^{-Kt}\text{h}_{t}\left(\left|D\text{h}_{s}f\right|\right)$.
Hence, for any non--negative $\varphi\in C_b({\X})$, we get
\begin{equation}
\int_{{\X}}\varphi\left|D\text{h}_{t+s}f\right|\text{d}\mu\le e^{-Kt}\int_{{\X}}\varphi\text{h}_{t}\left|D\text{h}_{s}f\right|\text{d}\mu=e^{-Kt}\int_{{\X}}\left(\text{h}_{t}\varphi\right)\left|D\text{h}_{s}f\right|\text{d}\mu,\label{eq:BE+semigroup}
\end{equation}
where we explicitly used the self--adjointness of the heat flow.

By \cite[Proposition 5.2]{mms2}, if $f\in BV({\X})$ then 
\begin{equation*}
\int_{\X} \left|D\text{h}_{s}f\right| \, \d \mu \rightarrow \|Df\|({\X})
\end{equation*}
as $s\rightarrow0$. Moreover, the lower semicontinuity of the total
variation yields
\[
\underset{s\rightarrow0}{\lim\inf} \int_{A} \left|D\text{h}_{s}f\right| \, \d \mu \ge\|Df\|(A)
\]
for any open set $A\subset{\X}$. Now, \cite[Proposition 1.80]{afp}
entails the convergence $\left|D\text{h}_{s}f\right| \mu \rightharpoonup\|Df\|$
as $s\rightarrow0$. Therefore, we can pass to the limit as $s\rightarrow0$
in the inequality (\ref{eq:BE+semigroup}) to get \eqref{eq:BE-BV-dual} for any non--negative $\varphi \in C_b({\X})$.

Now, we can take the supremum over $\varphi\in C_b({\X})$,
$0\le\varphi\le1$, to get
\[
\left\Vert D\text{h}_{t}f\right\Vert _{L^{1}(\X)}\le e^{-Kt}\|Df\|({\X}),
\]
which in turn gives $\left|D\text{h}_{t}f\right|\in L^{1}(\X)$. Finally, in the case $\supp(\mu) = \X$, let us consider a non--negative $\varphi \in L^{\infty}(\X)$: for any $\eps > 0$, we have $\h_{\eps} \varphi \in C_b(\X)$, by the $L^{\infty}$--contractivity property, and so we get
\begin{equation*}
\int_{{\X}}\h_{\eps} \varphi \left|D \h_{t}f\right| \, \d \mu\le e^{-Kt}\int_{{\X}} \h_{t+ \eps}\varphi  \, \d \|Df\| 
\end{equation*}
for any $t > 0$. Hence, by taking the limit as $\eps \to 0$, we get \eqref{eq:BE-BV-dual}.
\end{proof}

We notice that, in the case $\varphi = 1$, \eqref{eq:BE-BV-dual} was already proved in \cite[Proposition 1.6.3]{ambhonda} under the more general assumption that $f$ is only a function of bounded variation.

\begin{lemma} \label{lem:pairing_X_f}
Let $X \in L^{\infty}(T\X)$ and $f\in L^{\infty}(\X)\cap BV(\X)$. Then the family of measures $(\d (\h_{t}f)(X) \mu)_{t > 0}$ satisfies
\begin{equation} \label{eq:pairing_X_f_bound}
\int_{{\X}} \left | \d \left(\h_{t}f\right)(X) \right| \, \d \mu \le \left\Vert |X| \right\Vert _{L^{\infty}(\X)} e^{-Kt}\|Df\|({\X}),
\end{equation}
and any of its weak limit point $\bm{D}f(X) \in \mathbf{M}(\X)$ satisfies 
\begin{equation}
\left|\bm{D}f(X)\right|\le\left\Vert |X|\right\Vert _{L^{\infty}(\X)}\|Df\|.\label{eq:abscont-pairing}
\end{equation}
\end{lemma}

\begin{proof}
By \eqref{eq:BE-BV-dual}, for any $\varphi \in C_b(\X)$ and $t > 0$, we get
\begin{align*}
\left|\int_{{\X}} \varphi \, \text{d}\left(\text{h}_{t}f\right)(X)\text{d}\mu\right| & \le\left\Vert |X|\right\Vert _{L^{\infty}(\X)}\int_{{\X}}|\varphi|\left|D\text{h}_{t}f\right|\text{d}\mu\\
 & \le\left\Vert \varphi \right\Vert _{L^{\infty}(\X)}\left\Vert |X|\right\Vert _{L^{\infty}(\X)}e^{-Kt}\|Df\|({\X}),
\end{align*}
which implies \eqref{eq:pairing_X_f_bound}, by passing to the supremum on $\varphi \in C_{b}(\X)$ with $\|\varphi\|_{L^{\infty}(\X)} \le 1$. This means that the family of measures $\text{d}\left(\text{h}_{t}f\right)(X) \mu$
is uniformly bounded in ${\bf M}({\X})$. In addition, this family of measures is uniformly tight: indeed, arguing as above we can see that $\|\text{d}\left(\text{h}_{t}f\right)(X) \mu\| \le \left\Vert |X|\right\Vert _{L^{\infty}(\X)} \| (D h_{t} f) \mu\|$, and the family $(\|(D h_{t} f) \mu\|)_{t > 0}$ is uniformly tight by Theorem \ref{thm:bog_weak_conv}, since $|(D h_{t} f)| \mu \rightharpoonup \|Df\|$ as $t \to 0^+$, by Lemma \ref{lem:heat_flow_lip_functions}. Hence, by Theorem \ref{thm:bog_weak_conv} we conclude that there exists a measure $\bm{D}f(X)$ and a positive vanishing sequence $\left(t_{j}\right)_{j\in\mathbb{N}}$ such that
\[
\text{d}\text{h}_{t_{j}}f(X) \mu \rightharpoonup \bm{D}f(X)\;\text{in}\;{\bf M}({\X}).
\]
In order to prove (\ref{eq:abscont-pairing}), we choose
a non--negative $\varphi\in C_b({\X})$ and we employ (\ref{eq:BE-BV-dual})
once more to get
\begin{align*}
\left|\int_{{\X}}\varphi\,\text{d}\bm{D}f(X)\right| & =\underset{j\rightarrow\infty}{\lim}\left|\int_{{\X}}\varphi\,\text{d}\left(\text{h}_{t_{j}}f\right)(X)\text{d}\mu\right|\\
 & \le\underset{j\rightarrow\infty}{\lim\inf}\int_{{\X}}\varphi\left|\text{d}\left(\text{h}_{t_{j}}f\right)(X)\right|\text{d}\mu\\
 & \le\underset{j\rightarrow\infty}{\lim\inf}\,e^{-Kt_{j}}\left\Vert |X|\right\Vert _{L^{\infty}(\X)}\int_{{\X}}\text{h}_{t_{j}}\left(\varphi\right)\text{d}\|Df\|\\
 & =\left\Vert |X|\right\Vert _{L^{\infty}(\X)}\int_{{\X}}\varphi\,\text{d}\|Df\|.
\end{align*}
\end{proof}

\begin{theorem} \label{thm:Leibniz_DM_BV} Let $X\in{\cal DM}^{\infty}({\X})$
and $f\in L^{\infty}(\X)\cap BV(\X)$, then $fX\in \DM^{\infty}({\X})$.
In addition, there exists a non--negative sequence $\left(t_{j}\right)_{j\in\mathbb{N}}$
with $t_{j}\searrow0$ such that 
\begin{equation*}
\text{h}_{t_{j}}f\overset{*}{\rightharpoonup}\tildef{f} \text{ in } L^{\infty}({\X},\|\div(X)\|),
\end{equation*} 
\begin{equation} \label{eq:pairing_w_conv}
{\rm d}{\rm h}_{t_{j}}f(X) \mu \rightharpoonup \bm{D}f(X) \text{ in } {\bf M}({\X}),
\end{equation} 
and it holds 
\begin{equation}
\div(fX)=\tildef{f}\div(X)+\bm{D}f(X), \label{eq:LeibnizDM-BV}
\end{equation}
where the measure $\bm{D}f(X)$ satisfies \eqref{eq:abscont-pairing}. In addition, if we assume $\|\div(X)\| \ll \mu$, then we have 
\begin{equation} \label{eq:tilde_f_abs_cont}
\tildef{f}(x) = f(x)  \text{ for } \|\div(X)\|\text{-a.e. } x \in \X,
\end{equation} 
and there exists a unique $\bm{D}f(X) \in \M(\X)$ satisfying 
\begin{equation*}
{\rm d}{\rm h}_{t}f(X) \mu \rightharpoonup \bm{D}f(X) \text{ as } t \searrow 0 \text{ in } {\bf M}({\X})
\end{equation*}
and \eqref{eq:abscont-pairing}. In this case, it holds
\begin{equation}
\div(fX)= f \div(X)+\bm{D}f(X). \label{eq:LeibnizDM-BV_abs_cont}
\end{equation}
\end{theorem}

\begin{proof} 
We notice again that, since $f \in L^{1}(\X) \cap L^{\infty}(\X)$, it is clear that $f \in L^{2}(\X)$. Then, we start by approximating $f$ via $\text{h}_{t}f$. 
Thanks to Lemma \ref{lem:heat_flow_lip_functions}, we know that $\h_{t} f \in \Lip_{b}(\X)$ and $|D \h_{t} f| \in L^{1}(\X)$, so that we can employ \eqref{eq:Leibniz-Lip} to obtain 
\begin{align}
-\int_{{\X}}\text{d}g\left(\left(\text{h}_{t}f\right)X\right)\text{d}\mu & =\int_{{\X}}g \, \text{d} \div\left(\left(\text{h}_{t}f\right)X\right)\nonumber \\
 & =\int_{{\X}}\left(\text{h}_{t}f\right)g \, \text{d} \div(X) +\int_{{\X}} g \d (\text{h}_{t}f)(X) \, \text{d}\mu \label{eq:Leib-dual-approx}
\end{align}
for any $g \in \Lip_{\bs}(\X)$.
By the $L^{2}$--convergence $\text{h}_{t}f\rightarrow f$ as $t \searrow 0$
and by (\ref{eq:d-Linearity}), we get
\[
\int_{{\X}}\text{d}g\left(\left(\text{h}_{t}f\right)X\right)\text{d}\mu=\int_{{\X}}\left(\text{h}_{t}f\right)\text{d}g(X)\text{d}\mu\rightarrow\int_{{\X}}f\text{d}g(X)\text{d}\mu=\int_{{\X}}\text{d}g(fX)\text{d}\mu.
\]

Using (\ref{eq:ht-contraction}), we have $\left|\text{h}_{t}f(x)\right| \le\left\Vert f\right\Vert _{L^{\infty}(\X)}$ for any $x \in \X$.
In particular, $\left(\text{h}_{t}f\right)_{t\ge0}$ is uniformly
bounded in $L^{\infty}({\X},\|\div(X)\|)$. Hence, there exist $\tildef{f}\in L^{\infty}({\X},\|\div(X)\|)$
and a positive sequence $t_{j}\searrow0$ such that $\text{h}_{t_{j}}f\overset{*}{\rightharpoonup}\tildef{f}$
in $L^{\infty}({\X},\|\div(X)\|)$. This yields
\[
\int_{{\X}}\left(\text{h}_{t_{j}}f\right)g \, \text{d} \div(X) \rightarrow\int_{{\X}}\tildef{f}g \, \text{d}\div(X)
\]
as $j\rightarrow\infty$. Finally, Lemma \ref{lem:pairing_X_f} implies that there exists a finite Radon measure $\bm{D}f(X)$ satisfying \eqref{eq:abscont-pairing} and such that, up to extracting a further subsequence, 
\begin{equation*}
\int_{{\X}}g\, \text{d}\left(\text{h}_{t_{j}}f\right)(X)\text{d}\mu \to \int_{\X} g \, \d\bm{D}f(X).
\end{equation*}
All in all, we obtain $Xf\in{\cal DM}^{\infty}({\X})$ and \eqref{eq:LeibnizDM-BV}. 
Finally, in the case $\|\div(X)\| \ll \mu$, we immediately get
\begin{equation*}
\int_{\X} g \, \text{h}_{t}f \d \div(X) \to \int_{\X} g \, f \d \div(X) \text{ for all } g \in \Lip_{\bs}(\X), 
\end{equation*}
thanks to the $L^{2}$-convergence $\text{h}_{t}f\rightarrow f$ as $t \searrow 0$. This proves \eqref{eq:tilde_f_abs_cont}. Then, \eqref{eq:Leib-dual-approx} implies that, for any sequence $(t_{j})_{j \in \N}$ for which \eqref{eq:pairing_w_conv} holds, we get
\begin{align*}
 \int_{\X} g \, \d\bm{D}f(X) & = \lim_{j \to + \infty} \int_{{\X}} g \d (\text{h}_{t_{j}}f)(X) \, \text{d}\mu \\
 & = \lim_{j \to + \infty} \int_{{\X}}\left(\text{h}_{t_{j}}f\right)g \, \text{d} \div(X) +\int_{{\X}}\text{d}g\left(\left(\text{h}_{t_{j}}f\right)X\right)\text{d}\mu \\
& = \lim_{t \searrow 0} \int_{{\X}}\left(\text{h}_{t}f\right)g \, \text{d} \div(X) +\int_{{\X}}\text{d}g\left(\left(\text{h}_{t}f\right)X\right)\text{d}\mu \\
& = \int_{{\X}} g \, f \text{d} \div(X) +\int_{{\X}}\text{d}g\left(f X\right)\text{d}\mu = \int_{{\X}} g \, f \text{d} \div(X) - \int_{{\X}} g \, \d \div (f X ).
 \end{align*}
This immediately implies that $\bm{D}f(X)$ is unique and \eqref{eq:LeibnizDM-BV_abs_cont} holds. In addition, we get 
\begin{equation*}
{\rm d}{\rm h}_{t}f(X) \mu = \div(h_{t}(f) X) - \text{h}_{t}f \div(X) \rightharpoonup \div(f X) - f \div(X) = \bm{D}f(X) \text{ as } t \searrow 0,
\end{equation*}
and this ends the proof.
\end{proof}

\section{Gauss--Green formulas on sets of finite perimeter in $\rcd$ spaces} \label{GGformulas}

From this point onwards, we shall assume $(\X,d,\mu)$ to be a locally compact $\rcd$ space, unless otherwise stated.

In this section we formally introduce the notion of interior and exterior distributional normal traces on the boundary of sets of finite perimeter (which may be in general different from those defined in Theorems \ref{gauss-geod} and \ref{thm:IBP_ext_reg}) for a divergence--measure field, in order to achieve the Gauss--Green and integration by parts formulas. Following the Euclidean approach (see for instance \cite[Theorem 3.2]{comi2017locally}), we define these distributional normal traces as the densities of suitable pairings involving the given field and the characteristic function of the set with respect to the perimeter measure of the set itself. Even though in the general case these traces are not uniquely determined, since we do not have the uniqueness of the pairing term in the Leibniz rule (Theorem \ref{thm:Leibniz_DM_BV}), we shall also consider some additional assumptions which allow us to obtain a unique distributional normal trace.

\subsection{Distributional normal traces and Leibniz rules for characteristic functions} \label{sec:refined_Leibniz_rule}

\medskip

Let $X \in \DM^{\infty}(\X)$ and let $E \subset \X$ be a measurable set such that $\uno_{E} \in BV(\X)$ or $\uno_{E^{c}} \in BV(\X)$. By Remark \ref{rem:prop_set_fin_per}, we know that $\|D \uno_{E}\|$ is well defined and equal to $\|D \uno_{E^c}\|$. Let $\left(\h_{t}\right)_{t\ge0}$ denote as usual the heat flow: by Lemma \ref{lem:pairing_X_f}, there exists a sequence $t_{j} \to 0$ and two measures $\bm{D} \uno_{E}(\uno_{E} X), \bm{D} \uno_{E}(\uno_{E^{c}} X)$ such that
\begin{equation} \label{eq:pairing_X_E_conv_int_ext}
\begin{array}{c} 
\d (\h_{t_{j}} \uno_{E})(X_E) \mu \weakto \bm{D} \uno_{E}(X_E) \\
\d (\h_{t_{j}} \uno_{E})(X_{E^c}) \mu \weakto \bm{D} \uno_{E}(X_{E^c}) 
\end{array}\quad \text{in $\; \mathbf{M}(\X)$,}
\end{equation}

where we have set $X_E\coloneqq\uno_E X$ and $X_{E^c}\coloneqq\uno_{E^c}X$, a notation which we shall keep throughout the remaining sections of the paper.

Again by Lemma \ref{lem:pairing_X_f}, these weak limit measures satisfy the following estimates:
\begin{equation}\label{eq:pairing_X_E_bound}
|\bm{D} \uno_{E}(X_E)| \le \||X|\|_{L^{\infty}(E)} \|D \uno_{E}\|, \ \ \text{and} \ \ |\bm{D} \uno_{E}(X_{E^c})| \le  \||X|\|_{L^{\infty}(E^{c})} \|D \uno_{E}\|.
\end{equation}

We observe that, since the pairing measures $\bm{D}\mathds1_{E}(\cdot)$
are absolutely continuous with respect to the perimeter measure, thanks to \eqref{eq:pairing_X_E_bound}, we
are entitled to consider their densities. Thus, we define the \textit{interior
and exterior distributional normal traces} of $X$ on $\partial E$ as the functions $\ban{X, \nu_{E}}_{\partial E}^{-}$ and $\ban{X, \nu_{E}}_{\partial E}^{+}$ in $L^{\infty}(\partial E, \|D \uno_{E}\|)$
satisfying
\begin{equation}
2\bm{D}\mathds1_{E}\left(X_E\right)= \ban{X, \nu_{E}}_{\partial E}^{-}\|D\mathds1_{E}\|\label{eq:int-trace-per}
\end{equation}
and
\begin{equation}
2\bm{D}\mathds1_{E}\left(X_{E^c}\right)=\ban{X, \nu_{E}}_{\partial E}^{+}\|D\mathds1_{E}\|.\label{eq:ext-trace-per}
\end{equation}

In addition, since $\h_{t_{j}}\uno_{E}$ is uniformly bounded in $L^{\infty}(\X, \|\div(X)\|)$, up to extracting a subsequence, we can assume that there exists $\tildef{\uno_{E}} \in L^{\infty}(\X, \|\div(X)\|)$ such that 
\begin{equation} \label{eq:tildef_uno_E_weak_def}
\h_{t_{j}}\uno_{E} \weakstarto \tildef{\uno_{E}} \ \ \text{in} \ \ L^{\infty}(\X, \|\div(X)\|).
\end{equation}
We notice that, since $\tildef{\uno_{E}}$ is $\|\div(X)\|$--measurable, then it coincides with a Borel measurable function outside of a $\|\div(X)\|$--negligible set. Therefore, without loss of generality, we shall always choose $\tildef{\uno_{E}}$ to be Borel measurable.
Since $\tildef{\mathds1_{E}}$ is a representative of $\mathds1_{E}$,
it is relevant to consider the level sets of this function; that is,
\begin{equation} \label{def:E_s}
\tildef{E^{s}}:=\left\{ \tildef{\uno_{E}}= s \right\} .
\end{equation}
It is clear that $\tildef{E^{s}}$ is a Borel set for any $s \in [0, 1]$. In particular, $\tildef{E^{1}}$ and $\tildef{E^{0}}$ can be seen
as weaker ``versions'' of the measure--theoretic interior and exterior
of $E$. It seems then natural to define a measure theoretic boundary related to $\tildef{\uno_{E}}$ as the Borel set $$\tildef{\partial^{*} E} := \X \setminus (\tildef{E^{1}} \cup \tildef{E^{0}}).$$

We stress the fact that the notions of distributional normal traces of $X$ on the boundary of $E$ and representative of $E$ introduced above are heavily dependent on the choice of the sequence $(t_{j})_{j \in \N}$. In the following, we will always consider the sequence $(t_{j})$, or suitable subsequences, along which \eqref{eq:pairing_X_E_conv_int_ext} and \eqref{eq:tildef_uno_E_weak_def} hold.

In the following remark we show some easy relations between $\tildef{\uno_{E}}$ and $\tildef{\uno_{E^{c}}}$, and the distributional normal traces of an essentially bounded divergence--measure field on the boundary of $E$ and $E^{c}$.

\begin{remark} \label{rem:normal_traces_compl} Let $X \in \DM^{\infty}(\X)$ and $E \subset \X$ be a measurable set such that $\uno_{E} \in BV(\X)$ or $\uno_{E^{c}} \in BV(\X)$. Let $t_{j} \to 0$ be a sequence satisfying $\h_{t_{j}} \uno_{E} \weakstarto \tildef{\uno_{E}}$ and $\h_{t_{j}} \uno_{E^{c}} \weakstarto \tildef{\uno_{E^{c}}}$ in $L^{\infty}(\X, \|\div(X)\|)$ (up to extracting a subsequence, this always holds true). Then, we observe that
\begin{equation*}
\h_{t_{j}} \uno_{E} = 1 - \h_{t_{j}} \uno_{E^{c}}.
\end{equation*}
It is then clear that 
\begin{equation} \label{eq:E_tilde_compl}
\tildef{\uno_{E}} = 1 - \tildef{\uno_{E^{c}}},
\end{equation}
which implies
\begin{equation} \label{eq:E_E_compl_tilde}
\tildef{E^{1}} = \tildef{(E^{c})^{0}}, \ \ \tildef{E^{0}} = \tildef{(E^{c})^{1}}, \ \ \tildef{E^{1/2}} = \tildef{(E^{c})^{1/2}}, \ \ \tildef{\partial^{*} E} = \tildef{\partial^{*} E^c}.
\end{equation}
In addition, since 
\begin{equation*} 
\text{d}\left(\text{h}_{t_{j}}\mathds1_{E^{c}}\right)\left(X\right) = - \d(\h_{t_{j}} \uno_{E})(X),
\end{equation*}
it follows that
\begin{equation} \label{eq:E_compl_pairing_X}
\bm{D}\uno_{E^{c}}(X) = - \bm{D} \uno_{E}(X) 
\end{equation}
in the sense of Radon measures. Arguing analogously, we also obtain
\begin{equation*}
\bm{D} \uno_{E}(X_E) = - \bm{D} \uno_{E^{c}}(X_E) \quad \text{and} \quad \bm{D} \uno_{E}( X_{E^c}) = - \bm{D} \uno_{E^{c}}(X_{E^c}).
\end{equation*}
This easily implies the following relations between the distributional normal traces on $E$ and $E^c$:
\begin{equation*}
\ban{X, \nu_{E}}^{-}_{\partial E} = - \ban{X, \nu_{E^{c}}}^{+}_{\partial E^{c}} \ \ \text{and} \ \  \ban{X, \nu_{E}}^{+}_{\partial E} = - \ban{X, \nu_{E^{c}}}^{-}_{\partial E^{c}}.
\end{equation*}
We also notice that the linearity property of the pairing $\bm{D}\uno_{E}(\cdot)$ implies
\begin{equation} \label{eq:linearity_pairing_uno_E}
\bm{D}\mathds1_{E}(X) =\bm{D}\mathds1_{E}\left(\left(\mathds1_{E}+\mathds1_{E^{c}}\right)X\right)  =\bm{D}\uno_{E}\left(X_E\right)+\bm{D}\uno_{E}\left(X_{E^c}\right),
\end{equation}
from which, thanks to \eqref{eq:int-trace-per} and \eqref{eq:ext-trace-per}, we get
\begin{equation} \label{eq:linearity_pairing_uno_E_traces}
\bm{D}\mathds1_{E}(X) = \frac{\ban{X, \nu_{E}}^{-}_{\partial E} + \ban{X, \nu_{E}}^{+}_{\partial E}}{2} \| D \uno_{E} \|.
\end{equation}
\end{remark}

In the case of a Borel measurable set $E$ it is actually possible to characterize the sets $\tildef{E^{1}}$ and $\tildef{E^{0}}$ in terms of $E$ and its complementary.

\begin{remark}\label{w-lim} Let $E$ be a Borel measurable set such that $\uno_{E} \in BV(\X)$ or $\uno_{E^{c}} \in BV(\X)$.  Since it obviously holds $\h_t\uno_{E^c}=1-\h_t\uno_E$, without loss of generality we can assume $\uno_E\in BV(\X)$, which implies $\h_t\uno_E\in L^2(\X)$ and $\h_t\uno_E\to\uno_E$ in $L^2(\X)$  as $t\to 0$.

If we now take $X\in\mathcal{DM}^\infty(\X)$  and a positive vanishing sequence $\left(t_j\right)_{j \in \N}$ such that $\h_{t_j}\uno_E\overset{*}{\rightharpoonup}\widetilde{\uno_E}$ in $L^\infty\left(\X,\|\div(X)\|\right)$, by the $L^2$--convergence we have, up to some subsequence, the pointwise convergence $\h_{t_j}\uno_E(x)\to\uno_E(x)$ for $\mu$--almost every $x\in\X$.

Hence, there exists a Borel set $\mathcal{N}_{E}$ such that $\mu(\mathcal{N}_{E}) = 0$ and $\h_{t_{j}} \uno_{E}(x) \to \uno_{E}(x)$ for any $x \notin \mathcal{N}_{E}$. Therefore, for any $\psi \in L^{1}(\X, \|\div(X)\|)$ we have
\begin{equation*}
\int_{\X} \psi \, (\h_{t_{j}}\uno_{E}) \, \d \|\div(X)\| \to \int_{\X} \psi \, \tildef{\uno_{E}} \, \d \|\div(X)\|
\end{equation*}
and
\begin{equation*}
\int_{\X \setminus \mathcal{N}_{E}} \psi \, (\h_{t_{j}}\uno_{E}) \, \d \|\div(X)\| \to \int_{\X \setminus \mathcal{N}_{E}} \psi \, \uno_{E} \, \d \|\div(X)\|, 
\end{equation*}
by Lebesgue's Dominated Convergence Theorem. We stress the fact that $\uno_{E}$ is measurable with respect to $\|\div(X)\|$, since $E$ is a Borel set. This means that 
\begin{equation*}
\tildef{\uno_{E}}(x) = \uno_{E}(x) \ \text{for} \ \|\div(X)\|\text{--a.e.} \ x \notin \mathcal{N}_{E},
\end{equation*}
and so we get that $E \setminus \mathcal{N}_{E} = \tildef{E^{1}} \setminus \mathcal{N}_{E}$, up to a $\|\div(X)\|$--negligible set. Thanks to \eqref{eq:E_E_compl_tilde}, we can argue similarly with $E^{c}$ and $\tildef{E^{0}}$, so that we obtain 
\begin{equation*}
\|\div(X)\|\left ( (E \Delta \tildef{E^{1}}) \setminus \mathcal{N}_{E} \right ) = 0 \ \text{and} \ \|\div(X)\|\left ( (E^{c} \Delta \tildef{E^{0}}) \setminus \mathcal{N}_{E} \right ) = 0.
\end{equation*}
In addition, 
\begin{equation*}
\tildef{\partial^{*} E} \setminus \mathcal{N}_{E} = \X \setminus ( \tildef{E^{1}} \cup \tildef{E^{0}} \cup \mathcal{N}_{E}) \subset \left ( E \setminus (\tildef{E^{1}} \cup \mathcal{N}_{E}) \right ) \cup \left ( E^{c} \setminus (\tildef{E^{0}} \cup \mathcal{N}_{E}) \right ),
\end{equation*}
and so we get
\begin{equation*} 
\|\div(X)\|\left(\tildef{\partial^{*}E} \setminus \mathcal{N}_{E} \right) = 0.
\end{equation*}
Thus, we conclude that $\tildef{E^{1}}$ and $\tildef{E^{0}}$ are the representatives of the sets $E$ and $E^{c}$ with respect to the measure $\|\div(X)\| \res \X \setminus \mathcal{N}_{E}$, respectively.
\end{remark}

We employ now these remarks to obtain a refinement of the Leibniz rule for characteristic functions of Caccioppoli sets. 

As a preliminary, we observe that $\h_{t_{j}} \uno_{E}$ is uniformly bounded in $L^{\infty}(\partial E, \|D \uno_{E}\|)$. Hence, there exists a function $\hatE \in L^{\infty}(\partial E, \|D \uno_{E}\|)$ such that, up to a subsequence, 
\begin{equation} \label{weak_star_conv_perimeter_uno_E}
\h_{t_{j}} \uno_{E} \weakstarto \hatE \ \ \text{in} \ L^{\infty}(\partial E, \|D \uno_{E}\|). 
\end{equation}
As in the case of $\tildef{\uno_{E}}$, there exists a Borel representative of $\hatE$ which coincides with it up to a $\|D\uno_{E}\|$-negligible set, and, unless otherwise stated, we shall always consider this representative.

\begin{theorem} \label{thm:Leibniz_rule_E_1} 
Let $X\in{\cal DM}^{\infty}({\X})$ and let $E\subset\X$
be a measurable set such that $\uno_{E} \in BV(\X)$ or $\uno_{E^{c}} \in BV(\X)$. Then, we have $\mathds1_{E}X\in{\cal DM}^{\infty}({\X})$ and the following formulas hold
\begin{equation}
\div\left(X_E\right)=\tildef{\mathds1_{E}}\div(X)+\bm{D}\mathds1_{E}(X),\label{eq:Leib1}
\end{equation}
\begin{equation}
\div\left(X_E\right)=\left(\tildef{\mathds1_{E}}\right)^{2}\div(X)+ \hatE\bm{D}\mathds1_{E}(X)+\bm{D}\mathds1_{E}\left(X_E\right),\label{eq:Leib2}
\end{equation}
\begin{equation}
\tildef{\mathds1_{E}}\left(1-\tildef{\mathds1_{E}}\right)\div(X)= \hatE \bm{D}\mathds1_{E}\left(X_E\right) - (1 - \hatE) \bm{D}\mathds1_{E}\left(X_{E^c}\right).\label{eq:div-boundary}
\end{equation}
\end{theorem}
\begin{proof} Let first $\uno_{E} \in BV(\X)$. By (\ref{eq:LeibnizDM-BV}) we immediately get \eqref{eq:Leib1}, by extracting a suitable subsequence of a sequence $(t_{j})_{j \in \N}$ such that \eqref{eq:pairing_X_E_conv_int_ext} and \eqref{eq:tildef_uno_E_weak_def} hold.
In order to prove to (\ref{eq:Leib2}), we approximate $\mathds1_{E}$ via $\text{h}_{t}\mathds1_{E}$
and, employing (\ref{eq:Leibniz-Lip}) and (\ref{eq:Leib1}), we get
\begin{align}
\div\left(\left(\text{h}_{t}\mathds1_{E}\right) X_E\right) & =\text{h}_{t}\mathds1_{E}\div\left(X_E\right)+\text{d}\left(\text{h}_{t}\mathds1_{E}\right)\left(X_E\right)\mu\nonumber \\
 & =\left(\text{h}_{t}\mathds1_{E}\right)\tildef{\mathds1_{E}}\div(X)+\left(\text{h}_{t}\mathds1_{E}\right)\bm{D}\mathds1_{E}(X)+\text{d}\left(\text{h}_{t}\mathds1_{E}\right)\left(X_E\right)\mu.\label{eq:Leib-approx}
\end{align}
We select now a non--negative sequence $\left(t_{j}\right)_{j \in \N}$,
$t_{j}\searrow0$, such that \eqref{eq:pairing_X_E_conv_int_ext}, \eqref{eq:tildef_uno_E_weak_def} and \eqref{weak_star_conv_perimeter_uno_E} hold, and we easily get \eqref{eq:Leib2}.

Now, let $E$ be such that $\uno_{E^{c}} \in BV(\X)$. For any $g \in \Lip_{\bs}(\X)$, we have
\begin{align*}
\int_{\X} g \, \d \div(X_{E^c}) & = - \int_{\X} \d g(X_{E^c}) \, \d \mu = - \int_{\X} \uno_{E^{c}} \d g(X) \, \d \mu \\
& = - \int_{\X} \d g(X) \, \d \mu + \int_{\X} \uno_{E} \d g(X) \, \d \mu = \int_{\X} g \, \div(X) + \int_{\X} \uno_{E} \d g(X) \, \d \mu, 
\end{align*}
from which it follows that
\begin{equation*}
\int_{\X} \uno_{E} \d g(X) \, \d \mu = \int_{\X} g \, \d (\div(X_{E^c}) - \div(X)),
\end{equation*}
which easily implies that $X_E \in \DM^{\infty}(\X)$ and 
\begin{equation} \label{eq:complementary_E_div_X}
\div(X_E) = \div(X) - \div(X_{E^c}).
\end{equation}
Hence, by applying \eqref{eq:Leib1} to $\uno_{E^{c}}$ and \eqref{eq:complementary_E_div_X}, we obtain
\begin{equation} \label{eq:Leib_E_div_X_compl_1}
\div(X_E) = \div(X) - \div(X_{E^c}) = (1 - \tildef{\uno_{E^{c}}}) \div(X) - \bm{D}\uno_{E^{c}}(X).
\end{equation}
Therefore, thanks to \eqref{eq:E_tilde_compl}, \eqref{eq:Leib_E_div_X_compl_1} and \eqref{eq:E_compl_pairing_X}, we get \eqref{eq:Leib1} for $E$ such that $\uno_{E^{c}} \in BV(\X)$. Arguing analogously, we obtain also \eqref{eq:Leib2}. 

Finally, if we subtract \eqref{eq:Leib2} from \eqref{eq:Leib1},
we get
\begin{equation*}
\tildef{\mathds1_{E}}\left(1-\tildef{\mathds1_{E}}\right)\div(X) + (1 - \hatE) \bm{D}\mathds1_{E}(X)-\bm{D}\mathds1_{E}\left(X_E\right)=0.
\end{equation*}
Thus, \eqref{eq:linearity_pairing_uno_E} directly entails \eqref{eq:div-boundary}.
\end{proof}

\begin{remark}
We notice that, if we consider the Euclidean space $(\R^{n}, |\cdot|, \Leb{n})$, then the functions $\tildef{\mathds1_{E}}$ and $\hatE$ coincide with the precise representative of $\uno_E$ almost everywhere with respect to the $(n-1)$--dimensional Hausdorff measure, thanks to De Giorgi's blow-up theorem \cite[Theorem 3.59]{afp}. Roughly speaking, the richer structure of the Euclidean space makes the finer analysis that we just performed in our more abstract setting unnecessary. However, already in the framework of stratified groups it is necessary to consider an object similar to $\tildef{\mathds1_{E}}$ (\cite[Section 5]{cm}).
\end{remark}

In the following proposition we summarize some elementary properties
of distributional normal traces.

\begin{proposition} \label{prop:elementary_prop_normal_traces}
Let $X\in{\cal DM}^{\infty}({\X})$ and $E\subset\X$
be a measurable set such that $\uno_{E} \in BV(\X)$ or $\uno_{E^{c}} \in BV(\X)$. Then, $\ban{X, \nu_{E}}_{\partial\Omega}^{-}, \ban{X, \nu_{E}}_{\partial E}^{+}\in L^{\infty}\left(\partial E,\|D\mathds1_{E}\|\right)$, with the estimates
\begin{equation} \label{eq:sup_norm_traces_rough_estimate}
\|\ban{X, \nu_{E}}_{\partial E}^{-}\|_{L^{\infty}(\partial E, \|D \uno_{E}\|)} \le 2 \||X|\|_{L^{\infty}(E)} \ \ \text{and} \ \ \|\ban{X, \nu_{E}}_{\partial E}^{+}\|_{L^{\infty}(\partial E, \|D \uno_{E}\|)} \le 2 \||X|\|_{L^{\infty}(E^{c})}.
\end{equation}
In addition, it holds
\begin{equation}
\tildef{\uno_{E}}\left(1-\tildef{\mathds1_{E}}\right)\div(X)=\frac{\hatE \ban{X, \nu_{E}}_{\partial E}^{-} - (1 - \hatE ) \ban{X, \nu_{E}}_{\partial E}^{+}}{2}\|D\mathds1_{E}\|.\label{eq:div-D1E}
\end{equation}
\end{proposition}

\begin{proof} It is easy to see that \eqref{eq:pairing_X_E_bound} implies \eqref{eq:sup_norm_traces_rough_estimate}.
As for (\ref{eq:div-D1E}), it follows directly from (\ref{eq:div-boundary}),
(\ref{eq:int-trace-per}) and (\ref{eq:ext-trace-per}).
\end{proof}

We employ now Proposition \ref{prop:elementary_prop_normal_traces} to investigate further the relation between the measure $\div(X)$ and the sets $\tildef{E^{s}}$.

\begin{proposition}\label{rel-div-Es}Let $X\in{\cal DM}^{\infty}({\X})$ and let
$E\subset\X$
be a measurable set such that $\uno_{E} \in BV(\X)$ or $\uno_{E^{c}} \in BV(\X)$. Then, we have
\begin{equation} \label{eq:pairing_equality}
\hatE \bm{D}\uno_{E}(X_E) \res \tildef{E^{1}} \cup \tildef{E^{0}} = (1 - \hatE) \bm{D}\uno_{E}(X_{E^c}) \res \tildef{E^{1}} \cup \tildef{E^{0}},
\end{equation}
\begin{equation} \label{eq:trace_uno_E_equality}
\hatE \ban{X, \nu_{E}}_{\partial E}^{-} = (1 - \hatE ) \ban{X, \nu_{E}}_{\partial E}^{+} \ \ \|D\mathds1_{E}\|\text{-a.e. on} \ \tildef{E^{1}} \cup \tildef{E^{0}}
\end{equation}
and
\begin{equation}\label{eq:repr-1E_tilde_bar_div}
\tildef{\mathds1_{E}}=\hatE \ \ \|\div(X)\|\text{--a.e. on} \ \tildef{\partial^{*} E}.
\end{equation}
\end{proposition}

\begin{proof}
It is clear that \eqref{eq:pairing_equality} follows immediately from \eqref{eq:div-boundary}, by restricting the measures to $\tildef{E^{1}} \cup \tildef{E^{0}}$. Then, \eqref{eq:int-trace-per}, \eqref{eq:ext-trace-per} and \eqref{eq:pairing_equality} easily imply \eqref{eq:trace_uno_E_equality}.
In order to prove \eqref{eq:repr-1E_tilde_bar_div}, we start by considering the sequence $\text{h}_{t_{j}}\mathds1_{E}$ satisfying \eqref{eq:tildef_uno_E_weak_def} and \eqref{weak_star_conv_perimeter_uno_E}.
Then, since $\tildef{\mathds1_{E}}\left(1-\tildef{\mathds1_{E}}\right)\in L^{\infty}({\X},\|\div(X)\|)$,
it follows that
\[
\left(\text{h}_{t_{j}}\mathds1_{E}\right)\tildef{\mathds1_{E}}\left(1-\tildef{\mathds1_{E}}\right)\div(X)\rightharpoonup\left(\tildef{\mathds1_{E}}\right)^{2}\left(1-\tildef{\mathds1_{E}}\right)\div(X).
\]
Taking into account \eqref{eq:div-D1E}, the fact that distributional normal traces
are in $L^{\infty}\left({\X},\|D\mathds1_{E}\|\right)$
and \eqref{weak_star_conv_perimeter_uno_E}, we obtain
\begin{align*}
\left(\text{h}_{t_{j}}\mathds1_{E}\right)\tildef{\mathds1_{E}}\left(1-\tildef{\mathds1_{E}}\right)\div(X) & =\left(\text{h}_{t_{j}}\mathds1_{E}\right)\frac{\hatE \ban{X, \nu_{E}}_{\partial E}^{-} - (1 - \hatE ) \ban{X, \nu_{E}}_{\partial E}^{+}}{2}\|D\mathds1_{E}\|\\
 & \rightharpoonup \hatE \frac{\hatE \ban{X, \nu_{E}}_{\partial E}^{-} - (1 - \hatE ) \ban{X, \nu_{E}}_{\partial E}^{+}}{2}\|D\mathds1_{E}\|.
\end{align*}
Employing again (\ref{eq:div-D1E}), we get
\[
\left(\tildef{\mathds1_{E}}\right)^{2}\left(1-\tildef{\mathds1_{E}}\right)\div(X)=\hatE \tildef{\mathds1_{E}}\left(1-\tildef{\mathds1_{E}}\right)\div(X),
\]
which yields
\[
\left(\tildef{\mathds1_{E}}-\hatE\right)\tildef{\mathds1_{E}}\left(1-\tildef{\mathds1_{E}}\right)\div(X)=0
\]
from which we infer the thesis.
\end{proof}

This result underlines the strict relation between $\tildef{\uno_{E}}$ and $\hatE$ on $\tildef{\partial^{*} E}$, and is fundamental in the proof of the general Gauss--Green formulas.

\begin{theorem} \label{thm:Leibniz_rule_E_2}
Let $X\in{\cal DM}^{\infty}({\X})$ and let $E\subset\X$
be a measurable set such that $\uno_{E} \in BV(\X)$ or $\uno_{E^{c}} \in BV(\X)$. Then, we have 
\begin{equation*}
\frac{1}{1 - \hatE} \in L^{1}(\X, \|\bm{D} \uno_{E}(X_E)\|) \quad \text{and} \quad \frac{1}{\hatE} \in L^{1}(\X, \|\bm{D} \uno_{E}(X_{E^c})\|).
\end{equation*}
In addition, the following formulas hold:
\begin{align}
\div\left(X_E\right) & =\mathds1_{\tildef{E^{1}}}\div(X) + \frac{1}{2(1 - \hatE)} \ban{X, \nu_{E}}_{\partial E}^{-}\|D\mathds1_{E}\|,\label{eq:Leib-final1}\\
\div\left(X_E\right) & =\mathds1_{\tildef{E^{1}} \cup \tildef{\partial^{*} E}}\div(X) + \frac{1}{2 \hatE}\ban{X, \nu_{E}}_{\partial E}^{+}\|D\mathds1_{E}\|,\label{eq:Leib-final2}
\end{align}
and 
\begin{equation}
\mathds1_{\tildef{\partial^{*} E}}\div(X)=\frac{\hatE \ban{X, \nu_{E}}_{\partial E}^{-} - (1 - \hatE) \ban{X, \nu_{E}}_{\partial E}^{+}}{2 \hatE(1 - \hatE)}\|D\mathds1_{E}\| \res \tildef{\partial^{*} E}.\label{eq:div-bd-final}
\end{equation}
\end{theorem}

\smallskip

\begin{proof}
Let 
\[
A_{k, E} := \left \{ 1 - \frac{1}{k} < \hatE < 1 - \frac{1}{k + 1} \right \}.
 \]
By \eqref{eq:div-boundary} and \eqref{eq:trace_uno_E_equality}, we get
\begin{equation*}
\frac{\hatE}{1 - \hatE} \bm{D} \uno_{E}(X_E) \res \tildef{\partial^* E} \cap A_{k, E} = \left ( \hatE \, \div(X) + \bm{D} \uno_{E}(X_{E^c}) \right ) \res \tildef{\partial^* E} \cap A_{k, E}
\end{equation*}
for any $k \ge 1$. Hence, we obtain
\begin{align*}
\int_{\tildef{\partial^* E}} \frac{1}{1 - \hatE} \, \d \|\bm{D}\uno_{E}(X_E)\| & = \sum_{k = 1}^{\infty} \int_{\tildef{\partial^* E} \cap A_{k, E}} \frac{1}{1 - \hatE} \, \d \|\bm{D}\uno_{E}(X_E)\| \\
& \le \int_{\tildef{\partial^* E} \cap A_{1, E}} \frac{1}{1 - \hatE} \, \d \|\bm{D}\uno_{E}(X_E)\| +\\
& + \sum_{k = 2}^{\infty} \int_{\tildef{\partial^* E} \cap A_{k, E}} \frac{k}{k - 1} \frac{\hatE}{1 - \hatE} \, \d \|\bm{D}\uno_{E}(X_E)\| \\
& \le 2 \| |X| \|_{L^{\infty}(E)} \|D \uno_{E}\|(\tildef{\partial^{*} E} \cap A_{1, E}) + \\
& + 2 \sum_{k = 2}^{\infty} \left ( \int_{\tildef{\partial^* E} \cap A_{k, E}} \hatE \, \d \|\div(X)\| + \|\bm{D} \uno_{E}(X_{E^c})\|\left (\tildef{\partial^{*}E} \cap A_{k, E} \right ) \right ) \\
& \le 2 ( \| |X| \|_{L^{\infty}(E)} + \| |X| \|_{L^{\infty}(E^{c})}) \|D \uno_{E}\|(\tildef{\partial^{*} E}) + 2 \|\div(X)\|(\tildef{\partial^{*} E}).
\end{align*}

In addition, by \eqref{eq:pairing_equality}, we get 
\begin{equation*}
\frac{\hatE}{1 - \hatE} \bm{D}\uno_{E}(X_E) \res \left ( \tildef{E^{1}} \cup \tildef{E^{0}} \right ) \cap A_{k, E} = \bm{D}\uno_{E}(X_{E^c}) \res \left ( \tildef{E^{1}} \cup \tildef{E^{0}} \right ) \cap A_{k, E},
\end{equation*}
for any $k \ge 1$, and this implies
\begin{align*}
\int_{\tildef{E^{1}} \cup \tildef{E^{0}}} \frac{1}{1 - \hatE} \, \d \|D\uno_{E}(X_E)\| & \le 2 \| |X| \|_{L^{\infty}(E)} \| D\uno_{E}\|((\tildef{E^{1}} \cup \tildef{E^{0}}) \cap A_{1, E}) + \\
& + \sum_{k = 2}^{\infty} \frac{k}{k - 1} \| |X|\|_{L^{\infty}(E^{c})} \|D\uno_{E}\|((\tildef{E^{1}} \cup \tildef{E^{0}}) \cap A_{k, E}) \\
& \le 2 ( \| |X| \|_{L^{\infty}(E)} + \| |X|\|_{L^{\infty}(E^{c})}) \|D\uno_{E}\|(\tildef{E^{1}} \cup \tildef{E^{0}}).
\end{align*}
Therefore, we obtain $\displaystyle \frac{1}{1 - \hatE} \in L^{1}(\X, \|\bm{D}\uno_{E}(X_E)\|)$.

Then, in order to prove $\displaystyle \frac{1}{\hatE} \in L^{1}(\X, \|\bm{D} \uno_{E}(X_{E^c})\|)$, we proceed in a similar way, by employing the sets $B_{k, E} := \left \{ \frac{1}{k + 1} < \hatE < \frac{1}{k} \right \}$.

Hence, we can restrict \eqref{eq:div-D1E} to $\tildef{\partial^{*} E}$, employ \eqref{eq:repr-1E_tilde_bar_div} and divide by $\hatE(1 - \hatE)$ to obtain
\begin{equation*}
\uno_{\tildef{\partial^{*} E}} \div(X) = \left ( \frac{1}{2(1 - \hatE)}  \ban{X, \nu_{E}}_{\partial E}^{-} - \frac{1}{2 \hatE}  \ban{X, \nu_{E}}_{\partial E}^{+} \right ) \|D\mathds1_{E}\| \res \tildef{\partial^{*} E}.
\end{equation*} 
Thus, \eqref{eq:div-bd-final} immediately follows.

If we combine \eqref{eq:Leib1}, \eqref{eq:linearity_pairing_uno_E_traces} and \eqref{eq:div-bd-final}, we obtain
\begin{align*}
\div(\uno_{E} X) & = \uno_{\tildef{E^{1}}} \div(X) + \tildef{\uno_{E}}\uno_{\tildef{\partial^{*}E}} \div(X) + \bm{D}\uno_{E}(X) \\
& = \uno_{\tildef{E^{1}}} \div(X) + \hatE\uno_{\tildef{\partial^{*}E}} \div(X) + \frac{\ban{X, \nu_{E}}^{-}_{\partial E} + \ban{X, \nu_{E}}^{+}_{\partial E}}{2} \| D \uno_{E} \| \\
& = \uno_{\tildef{E^{1}}} \div(X) + \frac{\hatE}{2(1 - \hatE)}  \ban{X, \nu_{E}}_{\partial E}^{-} \|D\mathds1_{E}\| \res \tildef{\partial^{*} E} - \frac{1}{2}  \ban{X, \nu_{E}}_{\partial E}^{+} \|D\mathds1_{E}\| \res \tildef{\partial^{*} E} +\\
& + \frac{\ban{X, \nu_{E}}^{-}_{\partial E} + \ban{X, \nu_{E}}^{+}_{\partial E}}{2} \| D \uno_{E} \| \\
& = \uno_{\tildef{E^{1}}} \div(X) + \frac{1}{2(1 - \hatE)}  \ban{X, \nu_{E}}_{\partial E}^{-} \|D\mathds1_{E}\| \res \tildef{\partial^{*} E} + \\
& + \frac{\ban{X, \nu_{E}}^{-}_{\partial E} + \ban{X, \nu_{E}}^{+}_{\partial E}}{2} \| D \uno_{E} \| \res \tildef{E^{1}} \cup \tildef{E^{0}}.
\end{align*}
Now, \eqref{eq:trace_uno_E_equality} implies
\begin{equation} \label{eq:trace_uno_E_equality_1}
\ban{X, \nu_{E}}^{+}_{\partial E} = \frac{\hatE}{1 - \hatE}  \ban{X, \nu_{E}}^{-}_{\partial E} \quad \| D \uno_{E} \|\text{-a.e. on} \ \tildef{E^{1}} \cup \tildef{E^{0}},
\end{equation}
thanks to the summability of $1/(1 - \hatE)$ with respect to the measure $|(X\cdot \nu_{E})^{-}_{\partial E}| \|D \uno_{E}\| = 2 \|\bm{D} \uno_{E}(X_E)\|$.
Thus, we substitute $\ban{X, \nu_{E}}^{+}_{\partial E}$ and we get \eqref{eq:Leib-final1}.

As for \eqref{eq:Leib-final2}, we employ \eqref{eq:repr-1E_tilde_bar_div} to rewrite \eqref{eq:Leib1} as
\begin{align*}
\div(\uno_{E} X) & = \uno_{\tildef{E^{1}} \cup \tildef{\partial^{*} E}} \div(X) + \uno_{\tildef{\partial^{*} E}} (\tildef{\uno_{E}} - 1) \div(X) + \bm{D}\uno_{E}(X)\\
& = \uno_{\tildef{E^{1}} \cup \tildef{\partial^{*} E}} \div(X) - \uno_{\tildef{\partial^{*} E}} (1 - \hatE) \div(X) + \bm{D}\uno_{E}(X).
\end{align*}
Then, by \eqref{eq:linearity_pairing_uno_E_traces}, \eqref{eq:div-bd-final} and \eqref{eq:trace_uno_E_equality_1}, we get
\begin{align*}
- \uno_{\tildef{\partial^{*} E}} (1 - \hatE) \div(X) + \bm{D}\uno_{E}(X) & = \left (- \frac{1}{2}  \ban{X, \nu_{E}}_{\partial E}^{-} + \frac{1 - \hatE}{2 \hatE}  \ban{X, \nu_{E}}_{\partial E}^{+} \right ) \|D\mathds1_{E}\| \res \tildef{\partial^{*} E} + \\
& + \frac{\ban{X, \nu_{E}}^{-}_{\partial E} + \ban{X, \nu_{E}}^{+}_{\partial E}}{2} \| D \uno_{E} \| \\
& = \frac{1}{2 \hatE}  \ban{X, \nu_{E}}_{\partial E}^{+} \|D\mathds1_{E}\| \res \tildef{\partial^{*} E} + \\
& + \left ( \frac{1 - \hatE}{\hatE} + 1 \right ) \frac{\ban{X, \nu_{E}}_{\partial E}^{+}}{2} \|D\mathds1_{E}\| \res \tildef{E^{1}} \cup \tildef{E^{0}} \\
& = \frac{1}{2 \hatE}  \ban{X, \nu_{E}}_{\partial E}^{+} \|D\mathds1_{E}\|.
\end{align*}
This concludes the proof.
\end{proof}

We end this section with a remark in which we consider the possibility of having two distinct weak* limits of $h_{t} \uno_{E}$ in $L^{\infty}(\X, \|\div(X)\|)$ and we show that they differ only inside $\partial E$.

\begin{remark}
Let $(t_j)$ and $(t_k)$ be two non--negative vanishing sequences such that 
\begin{equation*}
\h_{t_{j}} \uno_{E} \weakstarto \tildef{\uno_{E}}   \ \ \text{and} \ \ \h_{t_{k}} \uno_{E} \weakstarto \tildef{\uno_{E}}' \ \ \text{in} \ L^{\infty}(\X, \|\div(X)\|),
\end{equation*}
with $\tildef{\uno_{E}} \neq \tildef{\uno_{E}}'$. If we set $\tildef{E^{1}}' := \{ \tildef{\uno_{E}}' = 1 \}$, thanks to Theorem \ref{thm:Leibniz_rule_E_2}, it is possible to prove that 
\begin{equation} \label{eq:difference_E_1_representative}
\|\div(X)\| \left ( \left (\tildef{E^{1}} \Delta \tildef{E^1}' \right ) \setminus \partial E \right ) = 0.
\end{equation} 
Indeed, by arguing as we did in this section also for the sequence $(t_{k})$, we obtain a version of \eqref{eq:Leib-final1} for $\tildef{E^{1}}'$; that is, there exists $\hatE'$ and $(\ban{X, \nu_{E}}_{\partial E}^{-})^{'}$ such that
\begin{equation*}
\div\left(X_E\right) =\mathds1_{\tildef{E^{1}}'}\div(X) + \frac{1}{2(1 - \hatE')} (\ban{X, \nu_{E}}_{\partial E}^{-})^{'}\|D\mathds1_{E}\|.
\end{equation*}
By combining this with \eqref{eq:Leib-final1}, we get
\begin{equation*}
\left (\uno_{\tildef{E^{1}}} - \uno_{\tildef{E^1}'} \right ) \div(X) = \left ( \frac{1}{2(1 - \hatE')} (\ban{X, \nu_{E}}_{\partial E}^{-})^{'} - \frac{1}{2(1 - \hatE)} \ban{X, \nu_{E}}_{\partial E}^{-} \right ) \|D \uno_{E}\|.
\end{equation*}
Hence, thanks to \eqref{eq:concentration_perimeter_measure}, \eqref{eq:difference_E_1_representative} immediately follows.
Analogously, it is possible to prove that 
\begin{equation*}
\|\div(X)\|\left (\left (\tildef{E^{0}} \Delta \tildef{E^0}' \right ) \setminus \partial E \right ) = 0.
\end{equation*}
\end{remark}

\subsection{The Gauss--Green and integration by parts formulas} \label{sec:general_Gauss_Green}

Before passing to the Gauss--Green formulas, we need the following
simple result.

\begin{lemma} \label{lem:bounded_supp} Let $X\in{\cal DM}^{\infty}({\X})$
be such that ${\rm supp}(X)$ is bounded. Then, $\div(X)({\X})=0.$\end{lemma}
\begin{proof} Let $g\in\Lip_{\bs}(\X)$ be such that $g\equiv1$ on
$\supp(X)$. Then, we have $gX=X$ and $g\div(X)=\div(X)$, and so,
by \eqref{eq:Leibniz-Lip}, we get 
\[
\div(X)=\div(gX)=g\div(X)+\d g(X)\d\mu=\div(X)+\d g(X)\d\mu,
\]
from which we deduce that $\d g(X)=0$. Therefore, by \eqref{eq:div-meas},
we have 
\begin{equation*}
\div(X)({\X})=\int_{\X}g\,\d\div(X)=-\int_{\X}\d g(X)\,\d\mu=0.
\end{equation*}
\end{proof}

We give now the general version of the Gauss--Green formula in locally compact $\mathsf{RCD}(K,\infty)$ metric measure spaces.

\begin{theorem}[Gauss--Green formulas I] \label{thm:Gauss_Green_formulas}
Let $X\in{\cal DM}^{\infty}({\X})$ and let $E\subset\X$ be a
bounded set of finite perimeter. Then, we have 
\begin{align}
\div(X)(\tildef{E^{1}}) & =-\int_{\partial E} \frac{1}{2(1 - \hatE)} \ban{X, \nu_{E}}_{\partial E}^{-}\,\d\|D\uno_{E}\|,\label{eq:GG_general_int}\\
\div(X)(\tildef{E^{1}}\cup\tildef{\partial^{*} E}) & = - \int_{\partial E} \frac{1}{2 \hatE} \ban{X, \nu_{E}}_{\partial E}^{+}\,\d\|D\uno_{E}\|.\label{eq:GG_general_ext}
\end{align}
\end{theorem} 
\begin{proof} 
It is enough to recall \eqref{eq:Leib-final1}
and \eqref{eq:Leib-final2}, and to apply Lemma \ref{lem:bounded_supp}
in order to get 
\begin{align*}
0=\div\left(X_E\right)(\X) & =\div(X)(\tildef{E^{1}})+\int_{\X} \frac{1}{2(1 - \hatE)} \ban{X, \nu_{E}}_{\partial E}^{-}\,\d\|D\mathds1_{E}\|,\\
0=\div\left(X_E\right)(\X) & =\div(X)(\tildef{E^{1}}\cup\tildef{\partial^*E }) + \int_{\X} \frac{1}{2 \hatE} \ban{X, \nu_{E}}_{\partial E}^{+}\,\d\|D\mathds1_{E}\|.
\end{align*}
Then we employ \eqref{eq:concentration_perimeter_measure} and we get \eqref{eq:GG_general_int} and \eqref{eq:GG_general_ext}.
\end{proof}

The first consequence of Theorem \ref{thm:Gauss_Green_formulas} are the integration by parts formulas.

\begin{theorem}[Integration by parts formulas I]\label{thm:IBP_general}
Let $X \in \DM^{\infty}(\X)$, $E\subset\X$ be a measurable set such that $\uno_{E} \in BV(\X)$ or $\uno_{E^{c}} \in BV(\X)$, and $\varphi \in \Lip_b(\X)$ such that $\supp(\uno_{E} \varphi)$ is bounded. Then, we have 
\begin{align}
\int_{\tildef{E^{1}}} \varphi \, \d \div(X) + \int_{E} \d\varphi(X) \, \d \mu & =-\int_{\partial E} \frac{1}{2(1 - \hatE)} \varphi \ban{X, \nu_{E}}_{\partial E}^{-}\,\d\|D\uno_{E}\|,\label{eq:IBP_general_int}\\
\int_{\tildef{E^{1}}\cup\tildef{\partial^{*} E}} \varphi \, \d \div(X) + \int_{E} \d\varphi(X) \, \d \mu & =-\int_{\partial E} \frac{1}{2 \hatE} \varphi \ban{X, \nu_{E}}_{\partial E}^{+}\,\d\|D\uno_{E}\|.\label{eq:IBP_general_ext}
\end{align}
\end{theorem}
\begin{proof}
It is clear that $\varphi X_E \in L^{\infty}(T \X)$, and that $X_E \in \DM^{\infty}(\X)$, by Theorem \ref{thm:Leibniz_rule_E_1}. Since $\supp(\uno_{E} \varphi)$ is bounded, there exists a cutoff function $\eta \in \Lip_{\bs}(\X)$ such that $\eta \equiv 1$ on $\supp(\uno_E \varphi)$. It is clear that $\eta \varphi \uno_{E} = \varphi \uno_{E}$, and that $\eta \varphi \in \Lip_{\bs}(\X)$. Hence, by Lemma \ref{dmlipleibniz}, we have $\eta \varphi X_E = \varphi X_{E} \in \DM^{\infty}(\X)$ and
\begin{align*}
\div(\varphi X_{E}) & = \div(\eta \varphi X_E) = \eta \varphi \div(X_E) + \d (\eta \varphi)(X_E) \mu \\
& = \eta \varphi \div(X_E) + \uno_{E} \eta \d \varphi(X) \mu + \uno_{E} \varphi \d \eta(X) \mu = \eta \varphi \div(X_E) + \uno_{E} \d \varphi(X) \mu.
\end{align*}
Now, we notice that $\|\div(X_{E})\|(\overline{E}^{c})=0$: indeed, for any $\psi \in \Lip_{\bs}(\overline{E}^c)$, we have
\begin{equation*}
\int_{\X} \psi \, \d \div(X_{E}) = - \int_{\X} \d \psi(X_{E}) \, \d \mu = 0
\end{equation*}
Hence, $\div(X_{E}) = \div(X_{E}) \res \overline{E}$, and this means that 
\begin{equation*}
\eta \varphi \div(X_E) = \varphi \div(X_{E}).
\end{equation*}
Then, by \eqref{eq:Leib-final1} and \eqref{eq:Leib-final2}, we have
\begin{align*}
\div(\varphi X_E) & = \uno_{\tildef{E^{1}}} \varphi \div(X) + \frac{1}{2(1 - \hatE)} \varphi \ban{X, \nu_{E}}^{-}_{\partial E} \|D \uno_{E}\| + \uno_{E} \d \varphi(X) \mu, \\
\div(\varphi X_E) & = \uno_{\tildef{E^{1}} \cup \tildef{\partial^{*} E}} \varphi \div(X) +  \frac{1}{2 \hatE} \varphi \ban{X, \nu_{E}}^{+}_{\partial E} \|D \uno_{E}\| + \uno_{E} \d \varphi(X) \mu. 
\end{align*}
Hence, we evaluate these equations over $\X$ and we employ Lemma \ref{lem:bounded_supp}, in order to obtain $$\div(\varphi X_E)(\X) = 0,$$ since $\supp(\varphi  X_E)$ is bounded. Thus, thanks to \eqref{eq:concentration_perimeter_measure}, we get \eqref{eq:IBP_general_int} and \eqref{eq:IBP_general_ext}.
\end{proof}

\subsection{The case $\|\div(X)\|\ll\mu$}\label{div_abscont_mu}

We study now the Gauss--Green formulas for fields with absolutely continuous divergence--measure, such as, for instance, the elements of $L^\infty(T\X)\cap\frd^{1}(\X)$. This property of the divergence--measure implies that, in particular, the distributional normal traces are unique and do not depend on any approximating sequence.

\begin{proposition} \label{prop:Leibniz_rule_abs_cont} Let $X\in\DM^{\infty}(\X)$
be such that $\|\div(X)\|\ll\mu$ and let $E\subset\X$ be a measurable set such that $\uno_{E} \in BV(\X)$ or $\uno_{E^{c}} \in BV(\X)$. Then, we have 
\begin{equation}
\tildef{\uno_{E}}(x)=\uno_{E}(x)\ \text{for}\ \|\div(X)\|\text{-a.e.}\ x\in\X,\label{eq:repr_uno_E_abs_cont}
\end{equation}
$\|\div(X)\|(\tildef{E^{1}}\Delta E)=0$ and $\|\div(X)\|(\tildef{\partial^{*} E})=0$.
In addition, there exists a unique distributional normal trace $\ban{X,\nu_{E}}_{\partial E}\in L^{\infty}(\partial E,\|D\uno_{E}\|)$,
which satisfies 
\begin{equation}
\div(X_E)=\uno_{E}\div(X)+ \ban{X,\nu_{E}}_{\partial E}\|D\uno_{E}\|.\label{eq:Leibniz_rule_abs_cont}
\end{equation}
\end{proposition} 
\begin{proof} 
It is easy to notice that \eqref{eq:repr_uno_E_abs_cont} is a consequence of \eqref{eq:tilde_f_abs_cont}. Then, the negligibility of the sets $\tildef{E^{1}}\Delta E$ and $\tildef{\partial^{*}E}$ with respect to the measure $\|\div(X)\|$ follows from \eqref{eq:repr_uno_E_abs_cont}. Hence, \eqref{eq:div-D1E}
implies that 
\begin{equation*}
\hatE(x) \ban{X, \nu_{E}}_{\partial E}^{-}(x)=(1 - \hatE(x)) \ban{X, \nu_{E}}_{\partial E}^{+}(x) \quad \text{for} \ \|D\uno_{E}\|\text{-a.e.} \ x\in\X,
\end{equation*}
from which we can define 
\begin{equation*}
\ban{X, \nu_{E}}_{\partial E}:= \frac{1}{2 (1 - \hatE)} \ban{X, \nu_{E}}_{\partial E}^{-} = \frac{1}{2 \hatE} \ban{X, \nu_{E}}_{\partial E}^{+}
\end{equation*}
to be the unique distributional normal trace of $X$ on $\partial E$. All in all,
we get \eqref{eq:Leibniz_rule_abs_cont} as a consequence of \eqref{eq:Leib-final2},
\eqref{eq:repr_uno_E_abs_cont} and the uniqueness of the distributional normal trace.
\end{proof}

The Leibniz rule given in Proposition \ref{prop:Leibniz_rule_abs_cont} allows us to obtain the following Gauss--Green and integration by parts formulas.

\begin{theorem}[Gauss--Green formulas II] \label{thm:GG_abs_cont} 
Let $X\in\DM^{\infty}(\X)$ be such that $\|\div(X)\|\ll\mu$ and let $E$ be a bounded set of finite perimeter. Then, we have 
\begin{equation}
\div(X)(E)=-\int_{\partial E}\ban{X, \nu_{E}}_{\partial E}\,\d\|D\uno_{E}\|.\label{eq:GG_abs_cont}
\end{equation}
\end{theorem} \begin{proof} It is enough to recall \eqref{eq:Leibniz_rule_abs_cont},
and to apply Lemma \ref{lem:bounded_supp} in order to get 
\begin{equation*}
0=\div\left(X_E\right)(\X)=\div(X)(E)+\int_{\X}\ban{X, \nu_{E}}_{\partial E}\,\d\|D\mathds1_{E}\|.
\end{equation*}
Then, \eqref{eq:concentration_perimeter_measure} implies \eqref{eq:GG_abs_cont}.
\end{proof}

\begin{theorem}[Integration by parts formulas II] \label{thm:IBP_abs_cont}
Let $X\in\DM^{\infty}(\X)$ be such that $\|\div(X)\|\ll\mu$, $E\subset\X$ be a measurable set such that $\uno_{E} \in BV(\X)$ or $\uno_{E^{c}} \in BV(\X)$, and $\varphi \in \Lip_b(\X)$ such that $\supp(\uno_{E} \varphi)$ is bounded. Then, we have 
\begin{equation} \label{eq:IBP_abs_cont}
\int_{E} \varphi \, \d \div(X) + \int_{E} \d\varphi(X) \, \d \mu = - \int_{\partial E}\varphi \ban{X, \nu_{E}}_{\partial E} \,\d\|D\uno_{E}\|.
\end{equation}
\end{theorem}
\begin{proof}
Let at first $\varphi \in \Lip_{\bs}(\X)$.
By Lemma \ref{dmlipleibniz} and Proposition \ref{prop:Leibniz_rule_abs_cont}, we have $\varphi  X _E\in \DM^{\infty}(\X)$ and 
\begin{align}
\div(\varphi X_E) & = \varphi \div(X_E) + \d \varphi(X_E) \mu \nonumber \\
 \label{eq:Leib_abs_cont_varphi} & = \uno_{E} \varphi \div(X) + \varphi \ban{X,\nu_{E}}_{\partial E}\|D\uno_{E}\| + \uno_{E} \d \varphi(X) \mu,
\end{align}
thanks to \eqref{eq:Leibniz-Lip} and \eqref{eq:Leibniz_rule_abs_cont}. Since $\supp(\varphi X_E)$ is bounded, Lemma \ref{lem:bounded_supp} implies that $$\div(\varphi X_E)(\X) = 0.$$
By evaluating \eqref{eq:Leib_abs_cont_varphi} over $\X$ and using \eqref{eq:concentration_perimeter_measure}, we get \eqref{eq:IBP_abs_cont} for $\varphi \in \Lip_{\bs}(\X)$. Now, let $\eta \in \Lip_{\bs}(\X)$ be a cutoff function such that $\eta \equiv 1$ on $\supp(\uno_{E} \varphi)$. Since $\eta \varphi \in \Lip_{\bs}(\X)$, we get
\begin{equation*} 
\int_{E} \eta \varphi \, \d \div(X) + \int_{E} \d(\eta \varphi)(X) \, \d \mu = - \int_{\partial E} \eta \varphi \ban{X, \nu_{E}}_{\partial E} \,\d\|D\uno_{E}\|.
\end{equation*}
It is easy to notice that $\eta \equiv 1$ on $\supp(\uno_{\overline{E}} \varphi)$, and that $\d(\eta \varphi) = \d \varphi$ on $E$, and this ends the proof.
\end{proof}

\subsection{The case of regular domains}\label{reg_domain_comparison}

In this section, we compare the two different integration by parts formulas that we obtain on a regular domain $\Omega$, namely \eqref{eq:-36} and \eqref{eq:IBP_general_int}. In particular, we are able to find some characterization for the sets $\tildef{\Omega^1}$ and $\tildef{\Omega^0}$.

\begin{proposition}
Let $X \in \DM^{\infty}(\X)$ and $\Omega$ be a regular domain such that $\overline{\Omega}^{c}$ is a regular domain and $\mu(\partial \Omega) = 0$. Let $\tildef{\Omega^1}, \tildef{\Omega^0}$ be defined as in \eqref{def:E_s}, for some weakly* converging sequence $\h_{t_{j}} \uno_{\Omega} \weakstarto \tildef{\uno_{\Omega}}$ in $L^{\infty}(\X, \|\div(X)\|)$.  Then, we have
\begin{equation} \label{eq:reg_dom_tilde_1_0}
\|\div(X)\|((\Omega \Delta \tildef{\Omega^1}) \setminus \partial \Omega) = 0 \ \text{and} \ \|\div(X)\|((\overline{\Omega}^c \Delta \tildef{\Omega^0}) \setminus \partial \Omega) = 0.
\end{equation}
\end{proposition}
\begin{proof}
If we apply \eqref{eq:IBP_general_int} to $E = \Omega$, and we subtract it from \eqref{eq:-36}, we get
\begin{equation*}
\int_{\X} f \left (\uno_{\Omega} - \uno_{\tildef{\Omega^1}} \right) \, \d \div(X) = - \int_{\partial \Omega} f \left ( (X \cdot \nu)_{\partial \Omega}^{-} - \frac{1}{2 ( 1 - \widehat{\uno_{\Omega}})} \ban{X, \nu_{\Omega}}_{\partial \Omega}^{-} \right) \, \d \|D \uno_{\Omega}\|,
\end{equation*}
for any $f \in \Lip_{b}(\X)$ such that $\supp(f \uno_{\Omega})$ is bounded. If we choose now $f \in \Lip_\bs(\X \setminus \partial \Omega)$, we obtain
\begin{equation*}
\int_{\X} f \left (\uno_{\Omega} - \uno_{\tildef{\Omega^1}} \right) \, \d \div(X)  = 0,
\end{equation*}
from which the first equality in \eqref{eq:reg_dom_tilde_1_0} follows. As for the second, it is enough to argue in a similar way, employing \eqref{eq:IBP_ext_reg} and \eqref{eq:IBP_general_ext}.
\end{proof}

In particular, in the case of absolutely continuous divergence--measure, it follows that the normal trace on regular domains coincide with the distributional one.

\begin{corollary}
Let $X \in \DM^{\infty}(\X)$ such that $\|\div(X)\| \ll \mu$, and let $\Omega$ be a regular domain such that $\overline{\Omega}^{c}$ is a regular domain and $\mu(\partial \Omega) = 0$. Then, we have $(X \cdot \nu_{\Omega})_{\partial \Omega} = \ban{X, \nu_{\Omega}}_{\partial \Omega}$.
\end{corollary}
\begin{proof}
It is an immediate consequence of \eqref{eq:IBP_abs_cont} applied to $E = \Omega$ and of \eqref{eq:IBP_reg_mod_abs_cont}.
\end{proof}

\subsection{The case $\hatE \equiv 1/2$}\label{hatE_onehalf}

In this section, we  show that it is possible to refine the results of Sections \ref{sec:refined_Leibniz_rule} and \ref{sec:general_Gauss_Green} under the additional assumption that, for any measurable set $E$ satisfying $\uno_{E} \in BV(\X)$ or $\uno_{E^{c}} \in BV(\X)$, the weak$^{*}$ limit points $\hatE$ of $\h_{t} \uno_{E}$ in $L^{\infty}(\X, \|D\uno_{E}\|)$ are constant functions. 

We wish to underline the fact that there exist non trivial $\rcd$ spaces enjoying such property. Indeed, in the Euclidean space $(\R^{n}, |\cdot|, \Leb{n})$, if $E$ or $\R^n \setminus E$ is a set of finite perimeter, we have 
\begin{equation*}
\h_t \uno_{E}(x) \to \frac{1}{2}  
\end{equation*} 
for any $x$ in the reduced boundary of $E$, as a simple consequence of the blow-up property. In addition, $\hatE$ is unique and constant in the Wiener spaces $(\X, \|\cdot \|, \gamma)$, where $(\X, \|\cdot\|)$ is a Banach space and $\gamma$ is a Gaussian probability measure. Indeed, Ambrosio and Figalli proved in \cite[Proposition 4.3]{Ambrosio_Figalli} that, if $E$ or $E^c$ is a set of finite perimeter in $(\X, \|\cdot \|, \gamma)$, then
\begin{equation*}
\h_{t} \uno_{E} \weakstarto \frac{1}{2} \quad \text{in} \ L^{\infty}(\X, \|D_{\gamma} \uno_{E}\|).
\end{equation*}
Hence, also in the Wiener spaces we have $\displaystyle \hatE \equiv \frac{1}{2}$. However, it is relevant to underline that the Wiener spaces are not locally compact, and so our theory for the Gauss--Green formulas cannot be applied there, a priori.

Actually, we can prove that, if we assume the constancy of the weak$^{*}$ limit points of $\h_{t} \uno_{E}$ in $L^{\infty}(\X, \|D\uno_{E}\|)$, then the limit is unique and equal to $1/2$.
To prove this result, we show first the following general property of the weak$^{*}$ limit points $\hatE \in L^{\infty}(\X, \|D\uno_{E}\|)$. We stress the fact that such property does not require the local compactness of the $\rcd$ metric measure space, but it is necessary to assume that $\supp(\mu) = \X$.

\begin{theorem} \label{thm:weak_star_conv_set_fin_per} 
Assume $\supp(\mu) = \X$. Let $E\subset \X$ be a measurable set such that either $\uno_{E}\in BV(\X)$ or
$\uno_{E^{c}}\in BV(\X)$. Then, any weak$^*$ limit $\hatE$ of $\h_{t} \uno_{E}$ in $L^{\infty}(\X, \|D \uno_{E}\|)$ satisfies
\begin{equation} \label{eq:weak_star_mean} 
\mean{\X} \hatE \, \d \|D \uno_{E}\| = \frac{1}{2}.
\end{equation}
\end{theorem}
\begin{proof}
Without loss of generality, we can assume $\mathds1_{E}\in BV(\X)$ (since clearly $\h_{t}\uno_{E^{c}} = 1 - \h_{t} \uno_{E}$);
under these hypotheses one has $\text{h}_{t}\mathds1_{E}\in L^{2}(\X) \cap C_{b}(\X)$, since $\supp(\mu) = \X$, and $\left|D\text{h}_{t}\mathds1_{E}\right|\in L^{1}(\X)$.
Now, let $X\in\frd^\infty(\X)$ such that $|X|\le1$. We have
\[
\underbrace{\int_{{\X}}\left(\text{h}_{t}\mathds1_{E}\right)\mathds1_{E}\div(X) \, \d\mu}_{\text{(A)}}=\underbrace{\int_{{\X}}\mathds1_{E}\div\left(\left(\text{h}_{t}\mathds1_{E}\right)X\right)\text{d}\mu}_{\text{(B)}}-\underbrace{\int_{{\X}}\mathds1_{E}\text{d}\left(\text{h}_{t}\mathds1_{E}\right)(X) \, \text{d}\mu}_{\text{(C)}}.
\]
Let us consider the term (B) right above: we see that
\begin{align*}
\int_{{\X}}\mathds1_{E}\div\left(\left(\text{h}_{t}\mathds1_{E}\right)X\right)\text{d}\mu =-\int_{{\X}} \, \text{d} D{\mathds1_{E}}\left(\text{h}_{t}\mathds1_{E} X \right)  =-\int_{{\X}}\left(\text{h}_{t}\mathds1_{E}\right) \, \text{d} D{\mathds1_{E}}(X).
\end{align*}
Therefore, it follows that
\[
\left|\int_{{\X}}\mathds1_{E}\div\left(\left(\text{h}_{t}\mathds1_{E}\right)X\right)\text{d}\mu\right|\le \int_{\X}\left(\text{h}_{t}\mathds1_{E}\right)\text{d}\|D\mathds1_{E}\|.
\]
As for (C), notice that, by the fact that $|\text{d}\left(\text{h}_{t}\mathds1_{E}\right)(X)| \le |X| |D \h_{t} \uno_{E}|$ and \eqref{eq:BE-BV-dual},
\begin{align*}
\left|\int_{{\X}}\mathds1_{E}\text{d}\left(\text{h}_{t}\mathds1_{E}\right)(X)\text{d}\mu\right| & \le \int_{\X} \uno_{E} |X| |D \h_{t} \uno_{E}| \, \d \mu \le \int_{\X} \uno_{E} |D \h_{t} \uno_{E}| \, \d \mu \\
& \le e^{-Kt}\int_{{\X}}\text{h}_{t} \mathds1_{E} \, \text{d}\|D\mathds1_{E}\|.
\end{align*}
Hence, we find that (A) satisfies the following bound:
\begin{equation*}
\left | \int_{{\X}}\left(\text{h}_{t}\mathds1_{E}\right)\mathds1_{E}\div(X) \, \d\mu \right | \le \int_{\X} \text{h}_{t}\mathds1_{E} \, \text{d}\|D\mathds1_{E}\| + e^{- Kt} \int_{\X} \h_{t}\uno_{E} \, \d \|D \uno_{E}\|.
\end{equation*}
Now, upon passing to the supremum over $X\in\frd^\infty(\X)$
with $|X|_{\infty}\le1$, we get
\begin{equation}
\|D\left(\text{h}_{t}\mathds1_{E}\right)\mathds1_{E}\|(\X)\le \int_{\X} \text{h}_{t}\mathds1_{E} \, \text{d}\|D\mathds1_{E}\| + e^{- Kt} \int_{\X} \h_{t}\uno_{E} \, \d \|D \uno_{E}\|.\label{eq:ineq-perimeter}
\end{equation}
We notice that by (\ref{eq:ht-contraction}) we have $0\le\text{h}_{t}\mathds1_{E}(x)\le1$
for every $x\in{\X}$. Hence, the sequence $\left(\text{h}_{t}\mathds1_{E}\right)_{t\ge0}$
is uniformly bounded in $L^{\infty}\left({\X},\|D\mathds1_{E}\|\right)$,
and then there exist a positive decreasing sequence $t_{j}\searrow0$
and a function $\hatE\in L^{\infty}\left({\X},\|D\mathds1_{E}\|\right)$
such that $\text{h}_{t_{j}}\mathds1_{E}\overset{*}{\rightharpoonup}\hatE$
in $L^{\infty}\left({\X},\|D\mathds1_{E}\|\right)$.

Clearly, (\ref{eq:ineq-perimeter}) holds for $t=t_{j}$ as well;
then, by the lower semicontinuity of the total variation, we can let $j\rightarrow\infty$ and we get
\begin{align*}
\|D\mathds1_{E}\|(\X) & \le\underset{j\rightarrow\infty}{\lim\inf}\|D\left(\text{h}_{t_{j}}\mathds1_{E}\right)\mathds1_{E}\|(\X) \le\underset{j\rightarrow\infty}{\lim\inf} (1 + e^{- Kt_{j}}) \int_{\X} \text{h}_{t_{j}}\mathds1_{E} \, \text{d}\|D\mathds1_{E}\| \\
 & =2\int_{\X}\hatE \, \text{d}\|D\mathds1_{E}\|.
\end{align*}
Therefore, we obtain
\begin{equation*}
\mean{\X}\hatE \, \d \|D \uno_{E}\| \ge \frac{1}{2}.
\end{equation*}

In order to prove the opposite inequality, we focus on $\mathds1_{E^{c}}$.
Arguing as before, we notice that $0\le\text{h}_{t}\mathds1_{E^{c}}\le1$
and so there exists a weak$^{*}$ limit point $\hatEc$ of this family in $L^{\infty}\left({\X},\|D\mathds1_{E}\|\right)$.
The convergence $\text{h}_{t_{j}}\mathds1_{E}\overset{*}{\rightharpoonup}\hatE$
entails that, up to a subsequence, there exists $\hatEc$
such that $\text{h}_{t_{j}}\mathds1_{E^{c}}\overset{*}{\rightharpoonup}\hatEc$,
satisfying $\hatEc=1-\hatE$.

\allowdisplaybreaks

Since in general $\mathds1_{E}\in BV({\X})$ does not mean that $\mathds1_{E^{c}}\in BV({\X})$, we need to verify the previous computations for $\mathds1_{E^{c}}$. For any $X\in\frd^{\infty}({\X})$ we have
\begin{align*}
\int_{{\X}}\left(\text{h}_{t}\mathds1_{E^{c}}\right)\mathds1_{E^{c}}\div(X)\text{d}\mu & =\int_{{\X}}\left(\text{h}_{t}\left(1-\mathds1_{E}\right)\right)\left(1-\mathds1_{E}\right)\div(X)\text{d}\mu\\
 & =\int_{{\X}}\left(1-\text{h}_{t}\mathds1_{E}\right)\left(1-\mathds1_{E}\right)\div(X)\text{d}\mu\\
 & =\int_{{\X}}\left(1-\text{h}_{t}\mathds1_{E}-\mathds1_{E}+\left(\text{h}_{t}\mathds1_{E}\right)\mathds1_{E}\right)\div(X)\text{d}\mu\\
 & =-\int_{{\X}}\left(\text{h}_{t}\mathds1_{E}\right)\div(X)\text{d}\mu-\int_{{\X}}\mathds1_{E}\div(X)\text{d}\mu\\&+\int_{{\X}}\left(\text{h}_{t}\mathds1_{E}\right)\mathds1_{E}\div(X)\text{d}\mu\\
 & =-\int_{{\X}}\left(\text{h}_{t}\mathds1_{E}\right)\div(X)\text{d}\mu-\int_{{\X}}\mathds1_{E}\div(X)\text{d}\mu+\\
 & +\int_{{\X}}\mathds1_{E} \, \div\left(\left(\text{h}_{t}\mathds1_{E}\right)X\right)\text{d}\mu-\int_{{\X}}\mathds1_{E}\text{d}\left(\text{h}_{t}\mathds1_{E}\right)(X)\text{d}\mu\\
 & =\int_{{\X}}\left(1-\mathds1_{E}\right)\text{d}\left(\text{h}_{t}\mathds1_{E}\right)(X)\text{d}\mu-\int_{{\X}}\mathds1_{E}\div\left(\text{h}_{t}\left(1-\mathds1_{E}\right)X\right)\text{d}\mu\\
 & =\int_{{\X}}\mathds1_{E^{c}}\text{d}\left(\text{h}_{t}\mathds1_{E}\right)(X)\text{d}\mu+\int_{{\X}} \text{h}_{t}\mathds1_{E^{c}} \d D \mathds1_{E} \left(X\right).
\end{align*}
Now we pass to the supremum over $X\in\frd^{\infty}({\X})$
with $\left\Vert X\right\Vert _{\infty}\le1$ in order to find
\[
\|D\left(\text{h}_{t}\mathds1_{E^{c}}\right)\mathds1_{E^{c}}\|({\X})\le\left(1+e^{-Kt}\right)\int_{{\X}}\text{h}_{t}\mathds1_{E^{c}}\text{d}\|D\mathds1_{E}\|.
\]
We consider again the sequence $t_{j}\searrow0$ in such a way that
$\text{h}_{t_{j}}\mathds1_{E^{c}}\overset{*}{\rightharpoonup}\hatEc$
and we pass to the limit as $j\rightarrow\infty$ exploiting the lower
semicontinuity to get
\begin{align*}
\|D\mathds1_{E}\|({\X}) & \le\underset{j\rightarrow\infty}{\lim\inf}\|D\left(\text{h}_{t_{j}}\mathds1_{E^{c}}\right)\mathds1_{E^{c}}\|({\X}) \le\underset{j\rightarrow\infty}{\lim\inf}\left(1+e^{-Kt_{j}}\right)\int_{{\X}}\text{h}_{t_{j}}\mathds1_{E^{c}}\text{d}\|D\mathds1_{E}\|\\
 & \le2\int_{{\X}}\hatEc\text{d}\|D\mathds1_{E}\|.
\end{align*}
Hence, we get
\begin{equation*}
\mean{\X} \hatEc \, \d \|D \uno_{E}\| \ge \frac{1}{2},
\end{equation*}
and so
\begin{equation*}
\mean{\X} \hatE \, \d \|D \uno_{E}\| \le \frac{1}{2},
\end{equation*}
Thus, we finally obtain \eqref{eq:weak_star_mean}.
\end{proof}

\begin{corollary} \label{cor:unique_1_2}
Let $E\subset \X$ be a measurable set such that either $\uno_{E}\in BV(\X)$ or
$\uno_{E^{c}}\in BV(\X)$. Assume that the weak$^*$ limit points $\hatE$ of $\h_{t} \uno_{E}$ in $L^{\infty}(\X, \|D \uno_{E}\|)$ are constant functions. Then, there exists a unique weak$^*$ limit $\displaystyle \hatE = \frac{1}{2}$.
\end{corollary}
\begin{proof}
Since any weak$^{*}$ limit point $\hatE$ is constant, \eqref{eq:weak_star_mean} implies that
\begin{equation*}
\hatE = \mean{\X} \hatE \, \d \|D \uno_{E}\| = \frac{1}{2}.
\end{equation*}
\end{proof}

In view of these results, in the rest of this section we shall assume that, for any measurable set $E$ such that $\uno_{E} \in BV(\X)$ or $\uno_{E^{c}} \in BV(\X)$, we have $\displaystyle \hatE \equiv \frac{1}{2}$. In addition, the $\rcd$ metric measure space $(\X, d, \mu)$ is chosen to be locally compact and such that $\supp(\mu) = \X$.

Under these assumptions, it is possible to refine many of the result of Section \ref{sec:refined_Leibniz_rule} and Section \ref{sec:general_Gauss_Green}.

\begin{proposition}Let $X\in{\cal DM}^{\infty}({\X})$ and let $E\subset\X$
be a measurable set such that $\uno_{E} \in BV(\X)$ or $\uno_{E^{c}} \in BV(\X)$. Then, we have
\begin{equation} \label{eq:trace_uno_E_equality_1_bis}
\ban{X, \nu_{E}}_{\partial E}^{-} = \ban{X, \nu_{E}}_{\partial E}^{+} \ \ \|D\mathds1_{E}\|\text{-a.e. on} \ \tildef{E^{1}} \cup \tildef{E^{0}},
\end{equation}
\begin{equation} \label{eq:tripartition_div_E}
\|\div(X)\|\left({\X}\backslash\left(\tildef{E^{0}}\cup\tildef{E^{1/2}}\cup\tildef{E^{1}}\right)\right)=0,
\end{equation}
and
\begin{equation} \label{eq:repr-1Ediv}
\tildef{\mathds1_{E}}=\left(\mathds1_{\tildef{E^{1}}}+\frac{1}{2}\mathds1_{\tildef{E^{1/2}}}\right) \quad \|\div(X)\|\text{-a.e. in} \ \X.
\end{equation}
\end{proposition}

\begin{proof}
Since $\hatE \equiv 1/2$, \eqref{eq:trace_uno_E_equality} implies \eqref{eq:trace_uno_E_equality_1_bis}. Analogously, \eqref{eq:repr-1E_tilde_bar_div} yields easily \eqref{eq:repr-1Ediv} and \eqref{eq:tripartition_div_E}.
\end{proof}

This result shows that there is a ``tripartition'' of ${\X}$
up to a $\|\div(X)\|$-negligible set, that is, $\tildef{\mathds1_{E}}(x)\in\left\{ 0,\frac{1}{2},1\right\} $
for $\|\div(X)\|$--almost every $x\in{\X}$, which in turn allows us
to refine the Gauss--Green formulas.

\begin{theorem} \label{thm:Leibniz_rule_E_22}
Let $X\in{\cal DM}^{\infty}({\X})$ and let $E\subset\X$
be a measurable set such that $\uno_{E} \in BV(\X)$ or $\uno_{E^{c}} \in BV(\X)$. Then, we have 
\begin{align}
\div\left(X_E\right) & =\mathds1_{\tildef{E^{1}}}\div(X)+\ban{X, \nu_{E}}_{\partial E}^{-}\|D\mathds1_{E}\|,\label{eq:Leib-final1_extra}\\
\div\left(X_E\right) & =\mathds1_{\tildef{E^{1/2}}\cup\tildef{E^{1}}}\div(X)+\ban{X, \nu_{E}}_{\partial E}^{+}\|D\mathds1_{E}\|,\label{eq:Leib-final2_extra}
\end{align}
and 
\begin{equation}
\mathds1_{\tildef{E^{1/2}}}\div(X)=\left[\ban{X, \nu_{E}}_{\partial E}^{-}-\ban{X, \nu_{E}}_{\partial E}^{+}\right]\|D\mathds1_{E}\|.\label{eq:div-bd-final_extra}
\end{equation}
\end{theorem}

\begin{proof}
We notice that \eqref{eq:tripartition_div_E} implies $\|\div(X)\|(\tildef{\partial^{*}E} \setminus \tildef{E^{1/2}}) = 0$. Hence, thanks to this fact and $\hatE \equiv 1/2$, it is clear that \eqref{eq:Leib-final1}, \eqref{eq:Leib-final2} and \eqref{eq:div-bd-final} imply \eqref{eq:Leib-final1_extra}, \eqref{eq:Leib-final2_extra} and \eqref{eq:div-bd-final_extra}, respectively.
\end{proof}

Finally, arguing analogously as in Section \ref{sec:general_Gauss_Green} and employing the assumption $\hatE \equiv 1/2$, we easily obtain the following refined versions of the Gauss--Green and integration by parts formulas.

\begin{theorem}[Gauss--Green formulas III] \label{thm:Gauss_Green_formulas_ref}
Let $X\in{\cal DM}^{\infty}({\X})$ and let $E\subset\X$ be a
bounded set of finite perimeter. Then, we have 
\begin{align}
\div(X)(\tildef{E^{1}}) & =-\int_{\partial E} \ban{X, \nu_{E}}_{\partial E}^{-}\,\d\|D\uno_{E}\|,\label{eq:GG_general_int_ref}\\
\div(X)(\tildef{E^{1}}\cup\tildef{E^{1/2}}) & = - \int_{\partial E} \ban{X, \nu_{E}}_{\partial E}^{+}\,\d\|D\uno_{E}\|.\label{eq:GG_general_ext_ref}
\end{align}
\end{theorem} 

\begin{theorem}[Integration by parts formulas III]\label{thm:IBP_general_ref}
Let $X \in \DM^{\infty}(\X)$, $E\subset\X$ be a measurable set such that $\uno_{E} \in BV(\X)$ or $\uno_{E^{c}} \in BV(\X)$, and $\varphi \in \Lip_b(\X)$ such that $\supp(\uno_{E} \varphi)$ is bounded. Then, we have 
\begin{align}
\int_{\tildef{E^{1}}} \varphi \, \d \div(X) + \int_{E} \d\varphi(X) \, \d \mu & =-\int_{\partial E} \varphi \ban{X, \nu_{E}}_{\partial E}^{-}\,\d\|D\uno_{E}\|,\label{eq:IBP_general_int_ref}\\
\int_{\tildef{E^{1}}\cup\tildef{E^{1/2}}} \varphi \, \d \div(X) + \int_{E} \d\varphi(X) \, \d \mu & =-\int_{\partial E} \varphi \ban{X, \nu_{E}}_{\partial E}^{+}\,\d\|D\uno_{E}\|.\label{eq:IBP_general_ext_ref}
\end{align}
\end{theorem}

\bigskip
\bigskip

\appendix
\section{Appendix}\label{sec-appendix}

In this section we recall the basic notions from the theory of $L^p$--normed $L^\infty$ modules developed in \cite{gi2}. Hence, from this point on, we follow the notation adopted in \cite{gi2}, by denoting $L^p(\X) := L^p(\X, \mu)$ by $L^p(\mu)$.

Naively, an $L^\infty$--module is a Banach space $(\mm,\|\cdot\|_\mm)$ seen as a module over the Abelian ring of essentially bounded functions.

\medskip

To get things started, let us consider a $\sigma$--finite measure space $(\X, \mathcal{A}, \mu)$.

\begin{definition}\cite[Definition 1.2.1]{gi2} A Banach space $(\mm,\|\cdot\|_\mm)$ is an $L^\infty(\mu)$\textit{--premodule} provided there is bilinear map
\begin{align*}
 L^\infty(\mu)\times\mm&\to\mm,\\
 (f,v)&\mapsto f\cdot v,
\end{align*}
namely the \textit{pointwise multiplication}, such that
\begin{align*}
 (fg)\cdot v&=f\cdot(g\cdot v),\\
 \uno\cdot v&=v,\\
 \|f\cdot v\|_\mm&\le\|f\|_{L^\infty(\mu)}\|v\|_\mm,
\end{align*}
for every $v\in\mm$ and $f,g\in L^\infty(\mu)$, where $\uno=\uno_\X$.

An $L^\infty(\mu)$--premodule becomes an $L^\infty(\mu)$\textit{--module} if the following additional properties are satisfied:

\smallskip

-- \underline{\textit{Locality}}: for every $v\in\mm$ and $A_n\in\mathfrak{B}(\X)$, $n\in\N$, one has

\begin{equation*}
 \uno_{A_n}\cdot v=0\:\forall\,n\in\N\quad\Longrightarrow\quad\uno_{\bigcup_{n\in\N}A_n}\cdot v=0.
\end{equation*}

-- \underline{\textit{Gluing}}: for any sequences $(v_n)_{n\in\N}\subset\mm$ and $(A_n)_{n\in\N}\subset\mathfrak{B}(\X)$ such that

\begin{equation*}
 \uno_{A_i\cap A_j}\cdot v_i=\uno_{A_i\cap A_j}\cdot v_j\:\forall\,i,j\in\N\qquad\mathrm{and}\qquad\underset{n\to\infty}{\overline{\lim}}\left\Vert\sum_{i=1}^n\uno_{A_i}\cdot v_i\right\Vert_{\mm}<\infty,
\end{equation*}

there exists $v\in\mm$ such that

\begin{equation*}
 \uno_{A_i}\cdot v=\uno_{A_i}\cdot v_i\:\forall\,i\in\N\qquad\mathrm{and}\qquad\|v\|_{\mm}\le\underset{n\to\infty}{\underline{\lim}}\left\Vert\sum_{i=1}^n\uno_{A_i}\cdot v_i\right\Vert_{\mm}.
\end{equation*}
\end{definition}

As expectable, with the definition of $L^\infty(\mu)$--module it comes a natural notion of module morphisms.

\begin{definition}
 Let $\mm,\nn$ be two $L^\infty(\mu)$--modules. We say that $T:\mm\to\nn$ is a \textit{module morphism} whenever is a bounded linear map from $\mm$ to $\nn$ viewed as Banach spaces, satisfying
 
 \begin{equation*}
  T(f\cdot v)=f\cdot T(v),\quad\forall\,v\in\mm,f\in L^\infty(\mu).
 \end{equation*}

The set of all module morphisms $T:\mm\to\nn$ will be denoted by $\textsc{Hom}(\mm,\nn)$.
\end{definition}

It can be seen that $\textsc{Hom}(\mm,\nn)$ has a canonical structure of $L^\infty(\mu)$--module, whose norm - as a Banach space - is just the operator norm

\begin{equation}\label{norm-hom}
 \|T\|\coloneqq\underset{\|v\|_\mm\le1}{\sup}\|T(v)\|_\nn.
\end{equation}

Since $L^1(\mu)$ has a natural structure of $L^\infty(\mu)$--module, a notion of duality can be given in the following sense.

\begin{definition}
 Let $\mm$ be an $L^\infty(\mu)$--module. The \textit{dual module} $\mm^*$ is defined as
 \begin{equation*}
  \mm^*\coloneqq\textsc{Hom}(\mm,L^1(\mu)).
 \end{equation*}
\end{definition}

Of course, by \eqref{norm-hom} one immediately gets
\begin{equation*}
 \|T\|_{\mm^*}=\underset{\|v\|_\mm\le1}{\sup}\|T(v)\|_{L^1(\mu)}.
\end{equation*}

It is interesting to notice that by virtue of this definition, one has that the dual of  $L^p(\mu)$ is precisely $L^q(\mu)$, where $p,q\in[1,\infty]$ with $\frac{1}{p}+\frac{1}{q}=1$; see \cite[Example 1.2.7]{gi2}.

\smallskip

If it is possible to endow an $L^\infty(\mu)$--module with an ``$L^p$--norm'', then one has the following definition.

\begin{definition} \label{def:L_p_module}
 Let $\mm$ be an $L^{\infty}(\mu)$--premodule and $p\in[1,\infty]$. We say that $\mm$ is an $L^p(\mu)$\textit{--normed premodule} if there is a non--negative map $|\cdot|_{*}:\mm\to L^p(\mu)$ such that
 \begin{align*}
  \||v|_{*}\|_{L^p(\mu)}&=\|v\|_\mm,\\
  |f\cdot v|_{*}&=|f||v|_{*}
 \end{align*}
$\mu$--almost everywhere for all $f\in L^\infty(\mu)$ and all $v\in\mm$.

We shall call $|\cdot|_{*}$ the \textit{pointwise $L^p(\mu)$--norm} or, more simply, the \textit{pointwise norm}.

When an $L^p(\mu)$--normed premodule is also an $L^\infty(\mu)$--module, it will be called an $L^p(\mu)$\textit{--normed module}.
\end{definition}

It is easy to see that $|\cdot|_{*}:\mm\to L^p(\mu)$ is continuous, thanks to the simple inequality

\begin{equation*}
 \||v|_{*}-|w|_{*}\|_{L^p(\mu)}\le\||v-w|_{*}\|_{L^p(\mu)}=\|v-w\|_\mm,
\end{equation*}

valid for all $v,w\in\mm$. Also, $|\cdot|_{*}$ is local in the sense that for any $v\in\mm$ and $E\in\mathfrak{B}(\X)$,
\begin{equation*}
 v=0\quad\mu\text{--almost\;everywhere\;on}\;E\quad\Longleftrightarrow\quad|v|_{*}=0\quad\mu\text{--almost\;everywhere\;on}\;E,
\end{equation*}
\cite[Proposition 1.2.12]{gi2}.

\begin{remark}\label{dual-mod-norm}
 By \cite[Proposition 1.2.14]{gi2}, if $\mm$ is an $L^p(\mu)$--normed module, $p\in[1,\infty]$, then its dual module $\mm^*$ is an $L^q(\mu)$--normed module, $\frac{1}{p}+\frac{1}{q}=1$ with pointwise norm defined by
 \begin{equation*}
  |L|_{*}\coloneqq\underset{v\in\mm;\:|v|\le1\;\mu\text{-a.e.}}{\text{ess-}\sup}|L(v)|,
 \end{equation*}
where we have now denoted by $|\cdot|$ the pointwise norm on $\mm$. Then, by duality one also finds
\begin{equation*}
 |v|=\underset{L\in\mm^*;\;|L|_{*}\le1\;\mu\text{-a.e.}}{\text{ess-}\sup}|L(v)|.
\end{equation*}
\end{remark}

%

Finally, we recall the notions of generating sets and span over $L^\infty(\mu)$--modules.

\begin{definition} \label{def:span} Let $\mathscr{M}$ be an $L^{\infty}(\mu)$--module, $V \subset \mathscr{M}$ a subset and $A \in \mathfrak{B}(\X)$.
The span of $V$ on $A$, denoted by ${\rm Span}_{A}(V)$, is the subset of $\mathscr{M}$ made of vectors $v$ concentrated on $A$ with the following property: there are $(A_{n}) \subset \mathfrak{B}(\X)$, $n \in \N$, disjoint such that $A = \bigcup_{i} A_{i}$ and for every $n$ elements $v_{1,n}, \dots , v_{m_{n},n} \in \mathscr{M}$ and functions $f_{1,n}, \dots , f_{m_{n},n} \in L^{\infty}(\mu)$ such that
$$\chi_{A_{n}} v = \sum_{i = 1}^{m_{n}} f_{i,n} v_{i,n}.$$
We refer to ${\rm Span}_{A}(V)$ as the space spanned by $V$ on $A$, or simply spanned by $V$ if $A = \X$.
Similarly, we refer to the closure $\overline{{\rm Span}_{A}(V)}$ of ${\rm Span}_{A}(V)$ as the space generated by $V$ on $A$, or simply as the space generated by $V$ if $A = \X$.
\end{definition}

\vspace{1cm}

\begin{adjustwidth}{1cm}{1cm}{\bf Acknowledgements.} This research was partially supported by the Academy of Finland and by the INDAM-GNAMPA project 2017 ``Campi vettoriali, superfici e perimetri in geometrie singolari''. The authors wish to thank Simone Di Marino and Nicola Gigli for their useful comments and suggestions, along with the anonymous referees for carefully reading the manuscript and pointing out some amendments on the previous version.
\end{adjustwidth}

\phantomsection

\vspace{1cm}

\textsc{Vito Buffa}:
40121, Bologna, Italy.

E--mail: bff.vti@gmail.com. \textsf{ORCID iD}: 0000-0003-4175-4848.

\medskip

\textsc{Giovanni E. Comi}: 
Universit\"at Hamburg, Fakult\"at f\"ur Mathematik, Informatik und Naturwissenschaften, Fachbereich Mathematik -- Bundesstra\ss e 55, 20146 Hamburg, Germany.

E--mail: giovanni.comi@uni-hamburg.de.

\medskip

\textsc{Michele Miranda Jr.}:
Universit\`a degli Studi di Ferrara, Dipartimento di Matematica e Informatica --
Via Machiavelli 30, 44121 Ferrara, Italy.

E--mail: michele.miranda@unife.it.

\end{document}